\documentclass{amsart}

\usepackage[foot]{amsaddr} 

\usepackage[utf8]{inputenc} 

\usepackage{verbatim}
\usepackage{amssymb}
\usepackage{amsthm}
\usepackage{amsmath}
\usepackage{amsfonts}
\usepackage{latexsym}
\usepackage{mathtools} 
\usepackage{graphicx}

\usepackage{enumitem}
\usepackage[english]{babel}
\usepackage[utf8]{inputenc}

\usepackage[usenames,dvipsnames]{xcolor} 
\definecolor{dkblue}{RGB}{1,31,91} 
\usepackage[colorlinks=true, pdfstartview=FitV, linkcolor=dkblue, citecolor=dkblue, urlcolor=dkblue]{hyperref}     

\usepackage{todonotes}

\definecolor{myred}{rgb}{0.7,0.1,0.1}
\definecolor{mygreen}{rgb}{0.1,0.7,0.1}
\definecolor{myblue}{rgb}{0.2,0.2,0.5}

\theoremstyle{plain}
\newtheorem{thm}{Theorem}
\newtheorem{remark}[thm]{Remark}
\newtheorem{prop}[thm]{Proposition}

\newtheorem{lemma}[thm]{Lemma}
\newtheorem{defn}[thm]{Definition}

\numberwithin{equation}{section}
\numberwithin{thm}{section}

\newcommand{\pv}{\text{pv}\hspace{-0.1cm}}

\newcommand{\fzerone}{\mathcal{F}^{0,1}}
\newcommand{\fzeronenu}{\mathcal{F}^{0,1}_\nu}
\newcommand{\foneone}{\dot{\mathcal{F}}^{1,1}}
\newcommand{\foneonenu}{\dot{\mathcal{F}}^{1,1}_\nu}

\newcommand{\ftwoonenu}{\dot{\mathcal{F}}^{2,1}_\nu}

\newcommand{\xoneonenu}{\|\bm{X}\|_{\dot{\mathcal{F}}^{1,1}_\nu}}
\newcommand{\xtwoonenu}{\|\bm{X}\|_{\dot{\mathcal{F}}^{2,1}_\nu}}

\newcommand{\paren}[1]{\left(#1\right)}
\newcommand{\jump}[1]{\left[#1\right]}

\newcommand{\PD}[2]{\frac{\partial#1}{\partial#2}}
\newcommand{\p}{\partial}

\newcommand{\at}[2]{\left. #1 \right|_{#2}}

\newcommand{\mb}[1]{\mathbf{#1}}
\newcommand{\mc}[1]{\mathcal{#1}}

\newcommand{\bm}[1]{\boldsymbol{#1}}
\newcommand{\abs}[1]{\left\lvert #1 \right\rvert}
\newcommand{\norm}[1]{\left\lVert #1 \right\rVert}
\newcommand{\dual}[2]{\left\langle #1,#2 \right\rangle}
\newcommand{\thetaeta}{\Big(\frac{\theta\!+\!\eta}{2}\Big)}

\newcommand{\dissconst}{\mathcal{C}}
\newcommand{\eqdef}{\overset{\mbox{\tiny{def}}}{=}}
\newcommand{\mH}{\mathcal{H} }
\newcommand{\sign}{\text{sgn}}
\newcommand{\derivdiff}{\bm{\mathcal{D}}^2}
\newcommand{\nuTIME}{\nu(t)}

\begin{document}



\title[The Peskin Problem with Viscosity Contrast]{The Peskin Problem with Viscosity Contrast}

\author[E. Garc\'ia-Ju\'arez]{Eduardo Garc\'ia-Ju\'arez$^{\dagger,\ast}$}
\thanks{$^{\ast}$partially supported by the grant MTM2017-89976-P (Spain), by the ERC through the Starting Grant project H2020-EU.1.1.-639227, and by the AMS-Simons Travel Grant.}


\author[Y. Mori]{Yoichiro Mori$^{\dagger,\mathsection}$}
\thanks{$^{\mathsection}$partially supported by the NSF grant DMS-1907583, 2042144 (USA) and the Math+X award from the Simons Foundation.}

\author[R. M. Strain]{Robert M. Strain$^{\dagger,\mathparagraph}$}
\address{$^\dagger$Department of Mathematics, University of Pennsylvania, David Rittenhouse Lab., 209 South 33rd St., Philadelphia, PA 19104, USA. 
$^{\ast}$\href{mailto:edugar@math.upenn.edu}{edugar@math.upenn.edu}
$^{\mathsection}$\href{mailto:y1mori@math.upenn.edu}{y1mori@math.upenn.edu}
$^{\mathparagraph}$\href{mailto:strain@math.upenn.edu}{strain@math.upenn.edu}}
\thanks{$^{\mathparagraph}$partially supported by the NSF grant DMS-1764177 (USA)}


\begin{abstract}
The Peskin problem models the dynamics of a closed elastic filament immersed in an incompressible fluid.  In this paper, we consider the case when the inner and outer viscosities are possibly different.  This viscosity contrast adds further non-local effects to the system through the implicit non-local relation between the net force and the free interface.  We prove the first global well-posedness result for the Peskin problem in this setting.  The result applies for \textit{medium size} initial interfaces in critical spaces and shows instant analytic smoothing.  We carefully calculate the medium size constraint on the initial data.   These results are new even without viscosity contrast. 
\end{abstract}

\setcounter{tocdepth}{1}

\maketitle
\tableofcontents

\section{Introduction}

Fluid structure interaction (FSI) problems in which an elastic structure interacts with a surrounding fluid are found in many areas of science and engineering. Many numerical algorithms have been developed for such problems, and the scientific computing of FSI problems continues to be a very active area of research \cite{MR2242805,MR2009378,TRYGGVASON2001708,richter2017fluid}. The {\em Peskin problem}, considered in this paper, is arguably one of the simplest FSI problems, and has been used extensively in physical modeling as well as in the development of numerical algorithms as a prototypical test problem.

\subsection{Formulation}
Consider the following fluid problem in $\mathbb{R}^2$. 
A closed elastic string $\Gamma$ encloses a simply connected bounded domain $\Omega_1\subset\mathbb{R}^2$
filled with a Stokes fluid with viscosity $\mu_1$. The outside region $\Omega_2=\mathbb{R}^2\backslash (\Omega_1\cup \Gamma)$ is filled with 
a Stokes fluid of viscosity $\mu_2$. The equations satisfied are:
\begin{align}
\label{uint}
\mu_1 \Delta \bm{u}-\nabla p&=\bm{0} \text{ in } \Omega_1 \\
\label{uout}
\mu_2 \Delta \bm{u}-\nabla p&=\bm{0} \text{ in } \Omega_2 \\
\label{incomp}
\nabla \cdot \bm{u}&=0 \text{ in } \mathbb{R}^2\backslash \Gamma.
\end{align}
Here $\bm{u}$ is the velocity field and $p$ is the pressure. 

We must specify the interface conditions at $\Gamma$.
Parametrize $\Gamma$ 
by the material or Lagrangian coordinate $\theta \in\mathbb{S}=\mathbb{R}/(2\pi\mathbb{Z})$, and let $\bm{\mathcal{X}}(\theta,t)$
denote the coordinate position of $\Gamma$ at time $t$. 
The parametrization is in the counter-clockwise direction, 
so that the interior region $\Omega_1$ is on the left hand side of the tangent vector 
$\partial \bm{\mathcal{X}}/\partial \theta$.
For any quantity $w$ defined on 
$\Omega_1$ and $\Omega_2$, we set:
\begin{equation*}
\jump{w}=\at{w}{\Gamma_{\rm 1}}-\at{w}{\Gamma_{\rm 2}}
\end{equation*}
where $\at{w}{\Gamma_{\rm 1}}$ and $\at{w}{\Gamma_{\rm 2}}$ are the trace values of $w$ at $\Gamma$ evaluated 
from $\Omega_1$ (interior) and $\Omega_2$ (exterior) sides of $\Gamma$.
Let $\bm{n}$ be the outward pointing unit normal vector on $\Gamma$:
\begin{equation*}
\bm{n}=-\frac{\p_\theta \bm{\mathcal{X}}^\perp}{\abs{\p_\theta \bm{\mathcal{X}}}}, \; \p_\theta \bm{\mathcal{X}}=\PD{\bm{\mathcal{X}}}{\theta},\;\p_\theta \bm{\mathcal{X}}^\perp=
\mc{R}\p_\theta \bm{\mathcal{X}},\; \mc{R}=\begin{bmatrix} 0 & -1 \\ 1 & 0 \end{bmatrix},
\end{equation*}
where $\mc{R}$ is the $\pi/2$ rotation matrix.
The interface conditions are:
\begin{align}
\label{ujump}
\PD{\bm{\mathcal{X}}}{t}&=\bm{u}(\bm{\mathcal{X}},t),
\\ 
\label{noslip}
\jump{\bm{u}}&=0,
\\
\label{stressjump}
\jump{\Sigma\bm{n}}&=\bm{F}_{\rm el}\abs{\p_\theta \bm{\mathcal{X}}}^{-1}, 
\; \Sigma=\begin{cases}
\mu_1 \paren{\nabla \bm{u}+(\nabla \bm{u})^{\rm T}}-pI &\text{ in } \Omega_1\\
\mu_2\paren{\nabla \bm{u}+(\nabla \bm{u})^{\rm T}}-pI &\text{ in }\Omega_2
\end{cases},
\end{align}
where $I$ is the $2\times 2$ identity matrix.
The first condition is the no-slip boundary condition and the second is the stress balance condition where $\Sigma$
is the fluid stress and $\bm{F}_{\text{el}}$
is the elastic force exerted by the string $\Gamma$. We let:
\begin{equation}\label{linearF}
\bm{F}_{\rm el}=k_0\p_\theta^2 \bm{\mathcal{X}}, \quad k_0>0,
\end{equation}
where $k_0$ is the elasticity constant of the string $\Gamma$.

In the far field, $\bm{x}\to \infty$, we impose the condition that $\bm{u}\to 0$ and $p\to 0$.  This completes the specification of the Peskin problem.



Let us rewrite the above problem using boundary integral equations. Given some function $\bm{F}$ defined on $\Gamma$, 
we express the solution to our problem as the following single layer potential on $\mathbb{S}=[-\pi, \pi]$:
\begin{align}\label{singlelayer}
\bm{u}(\bm{x},t)&=\int_{\mathbb{S}} G(\bm{x}-\bm{\mathcal{X}}(\eta))\bm{F}(\eta)d\eta,\\
\label{G}
G(\bm{x})&=\frac{1}{4\pi}\paren{-\log \abs{\bm{x}}I+\frac{\bm{x}\otimes\bm{x}}{\abs{\bm{x}}^2}}, \; \quad
\bm{x}=(x_1,x_2)^{\rm T}\in \mathbb{R}^2,
\end{align}
where $G$ is the Stokeslet, the fundamental solution of the 2D Stokes problem.  For additionally $\bm{y}=(y_1,y_2)^{\rm T}\in \mathbb{R}^2$ we use  the notation
$$
\bm{x}\otimes\bm{y}=\begin{bmatrix}x_1 y_1 & x_1 y_2 \\ x_2y_1  & x_2 y_2\end{bmatrix}.
$$
We note that $\bm{\mathcal{X}}$ and $\bm{F}$ (and other variables) depend on $t$, 
but we will often suppress this dependence to avoid cluttered notation.
We note that the single layer potential does not have a velocity jump across the interface, and thus, 
the boundary condition \eqref{noslip} is automatically satisfied. We thus have:
\begin{equation}\label{Xteqn}
\PD{\bm{\mathcal{X}}}{t}(\theta)=\int_{\mathbb{S}} G(\Delta \bm{\mathcal{X}})\bm{F}(\eta)d\eta,
\end{equation}
where we use the notation
\begin{equation*}
 \Delta \bm{\mathcal{X}}=\bm{\mathcal{X}}(\theta)-\bm{\mathcal{X}}(\eta).
\end{equation*}
On the other hand, the stress interface condition \eqref{stressjump} is not automatically satisfied, 
and this will lead to an equation for $\bm{F}$. Let us compute the stress associated with the single layer expression \eqref{singlelayer}.
The stress $\Sigma$ in $\Omega_2$ is given by:
\begin{equation*}
\Sigma_{ij}(\bm{x})=\mu_2\int_\mathbb{S} \mathcal{T}_{ijk}(\bm{x}-\bm{\mathcal{X}}(\eta))F_k(\eta)d\eta, 
\end{equation*}
with 
\begin{equation}\label{Tijk}
\mathcal{T}_{ijk}=-\frac{1}{\pi}\frac{x_ix_jx_k}{\abs{\bm{x}}^4},
\end{equation}
where the subscripts denote the components of the respective tensors/vectors, such as $\bm{F}=(F_1,F_2)^{\rm T}$, and the summation convention is in effect for repeated indices. In $\Omega_1$, the stress is given by
\begin{equation*}
\Sigma_{ij}(\bm{x},t)=\mu_1\int_\mathbb{S} \mathcal{T}_{ijk}(\bm{x}-\bm{\mathcal{X}}(\eta))F_k(\eta)d\eta.
\end{equation*}
Thus, the trace values of the normal stresses are given by the following equations:
\begin{align*}
\at{\Sigma_{ij}(\bm{\mathcal{X}}(\theta))n_j(\theta)}{\Gamma_{\rm 2}}
&\!=\!\mu_2\paren{-\frac{1}{2}F_i\abs{\p_\theta \bm{\mathcal{X}}}^{-1}+\pv\int_{\mathbb{S}} \mathcal{T}_{ijk}(\Delta \bm{\mathcal{X}})F_k(\eta) n_j(\theta) d\eta},\\
\at{\Sigma_{ij}(\bm{\mathcal{X}}(\theta))n_j(\theta)}{\Gamma_{\rm 1}}
&=\mu_1 \paren{\frac{1}{2}F_i\abs{\p_\theta\bm{\mathcal{X}}}^{-1}
	+\pv\int_{\mathbb{S}} \mathcal{T}_{ijk}(\Delta \bm{\mathcal{X}})F_k(\eta) n_j(\theta) d\eta}.
\end{align*}
The stress jump condition \eqref{stressjump} thus reduces to (for $i=1,2$):
\begin{equation*}
F_i(\theta)+2A_\mu\int_{\mathbb{S}} \mathcal{T}_{ijk}(\Delta \bm{\mathcal{X}})F_k(\eta) \p_\theta \mathcal{X}^\perp_j(\theta) d\eta
=\frac{2}{\mu_1+\mu_2}F_{{\rm el},i}(\theta),
\end{equation*}
where 
\begin{equation}\label{Amu}
A_\mu=\frac{\mu_2-\mu_1}{\mu_1+\mu_2}.
\end{equation}
We define 
\begin{equation}\notag
    \mathcal{S}_i(\bm{F},\bm{\mathcal{X}})(\theta)=-\p_\theta \mathcal{X}^\perp_j(\theta)\int_{\mathbb{S}} \mathcal{T}_{ijk}(\Delta \bm{\mathcal{X}})F_k(\eta)  d\eta.
\end{equation}
We will frequently write it in vector notation as follows
\begin{equation}\label{viscosityjump}
\begin{aligned}
\bm{F}(\theta)=2A_\mu \bm{\mathcal{S}}(\bm{F},\bm{\mathcal{X}})(\theta)+2A_e\bm{\tilde{F}}_{\text{el}}(\theta),
\end{aligned}
\end{equation}
where
\begin{equation}\label{S}
\begin{aligned}
\bm{\mathcal{S}}(\bm{F},\bm{\mathcal{X}})(\theta)=-\partial_\theta \bm{\mathcal{X}}(\theta)^\perp\cdot\int_{\mathbb{S}}  \mathcal{T}(\bm{\mathcal{X}}(\theta)-\bm{\mathcal{X}}(\eta))\cdot \bm{F}(\eta)d\eta,
\end{aligned}
\end{equation}
with
\begin{equation}\label{FadAel}
\bm{\tilde{F}}_{\text{el}}=\frac{1}{k_0}\bm{F}_{\text{el}}, \hspace{0.2cm} A_e=\frac{k_0}{\mu_2+\mu_1}.
\end{equation}
We point out that the above boundary integral equation has a unique solution $\bm{F}$ given $\bm{F}_{{\rm el}}$ for sufficiently smooth $\bm{\mathcal{X}}$.

The Peskin problem thus reduces to the integral equations \eqref{Xteqn} and \eqref{viscosityjump} for $\bm{\mathcal{X}}$, where $G$, $\mathcal{T}$, $A_\mu$, $\bm{\mathcal{S}}$, $A_{e}$, and $\bm{\tilde{F}}_{\text{el}}$ are given by \eqref{G}, \eqref{Tijk}, \eqref{Amu}, \eqref{S}, and \eqref{FadAel}, with $\bm{F}_{\text{el}}$ given by \eqref{linearF}. Note also that, when $A_\mu=0$, i.e., $\mu_1=\mu_2$, equation \eqref{viscosityjump} reduces to $\bm{F}=2A_e\bm{F}_{\rm el}$, 
and we may just work with the single equation \eqref{Xteqn}.


Assuming that the stationary solutions are sufficiently smooth, it can be shown by an easy calculation that the 
only stationary solutions are those in which $\bm{\mc{X}}$ is a uniformly parametrized circle and the velocity field is $\bm{u}=0$ (see Section 5.1 of \cite{MR3935476}). 
Thus, all of the equilibrium configurations of \eqref{Xteqn} and \eqref{viscosityjump} are spanned by \begin{equation}\label{vectors}
\bm{e_r}(\theta)=\begin{bmatrix}
\cos{\theta}\\
\sin{\theta}
\end{bmatrix},\hspace{0.2cm}\bm{e_t}(\theta)=\begin{bmatrix}
-\sin{\theta}\\
\cos{\theta}
\end{bmatrix},\hspace{0.2cm}\bm{e_1}=\begin{bmatrix}
1\\0
\end{bmatrix},\hspace{0.2cm}\bm{e_2}=\begin{bmatrix}
0\\1
\end{bmatrix}.
\end{equation}

\subsection{Critical Regularity and Related Results}

A general guideline for seeking the most natural and largest class of initial data for a given problem is to identify its scaling,
and consider a function space that is critical (invariant) with respect to this scaling.
The Peskin problem given above by \eqref{Xteqn} and \eqref{viscosityjump} is invariant under dilation, 
and thus to make proper sense of scaling one must first fix a reference scale.
Consider the scaling parameter $\lambda>0$. The
domain scales accordingly from the torus $\mathbb{S}=[-\pi,\pi]$ to $\mathbb{S}/\lambda=[-\pi/\lambda,\pi/\lambda]$.
Then, we choose as the reference scale the length of uniformly parameterized circles, which we pick to be $2\pi$.
Given the additional rotation and translation invariance of the problem, let us consider the particular choice
\begin{equation*}
    \bm{X}_{*,\lambda}(\theta)=\lambda^{-1}\bm{X}_*(\lambda \theta),
\end{equation*}
where $\bm{X}_*(\theta)=\bm{e_r}(\theta)$. Then,  the system \eqref{Xteqn}, \eqref{viscosityjump} is written in terms of the difference $\bm{X}(\theta,t)=\bm{\mathcal{X}}(\theta,t)-\bm{X}_*(\theta)$. It is not hard to check that the following scale-invariance holds: if $\bm{X}(\theta,t)$ is a solution, then $\bm{X}_\lambda(\theta,t)=\lambda^{-\alpha}\bm{X}(\lambda \theta,\lambda t)$ is also a solution if and only if $\alpha=1$. 

The analytical study of the Peskin problem was initiated in \cite{MR3882225,MR3935476}, in which the case of equal viscosity $\mu_1=\mu_2$ was studied. In \cite{MR3882225}, well-posedness is established in $\bm{\mc{X}}\in C([0,T];H^{5/2}(\mathbb{S})), T>0$ with initial data $\bm{\mc{X}}_0$ in $H^{5/2}(\mathbb{S})$ whereas in \cite{MR3935476}, the solution resides in $\bm{\mc{X}}\in C([0,T];C^{1,\alpha}(\mathbb{S})), \alpha>0, T>0$ with initial data $\bm{\mc{X}}_0$ in $h^{1,\alpha}(\mathbb{S}), \alpha>0$ (this space is the completion of smooth functions in the $C^{1,\alpha}$ norm). These spaces are subcritical with respect to the above scaling. Indeed, in the $L^2$ Sobolev scale, $H^{3/2}(\mathbb{S})$ (or $C([0,T];H^{3/2}(\mathbb{S}))$) is the critical space, whereas in the scale of (H\"older) continuous functions, $C^1(\mathbb{S})$ (or $C([0,T]; C^1(\mathbb{S}))$) is the critical scale. In this sense, the results in \cite{MR3935476} are only barely subcritical. The semilinear parabolic methods \cite{MR3012216} that are used in \cite{MR3935476} relies crucially on subcriticality, however, and do not seem to be readily extendible to the critical regularity exponent.

In this paper, we consider the Peskin problem in which the viscosities $\mu_1$ and $\mu_2$ are not necessarily equal. Furthermore, we establish a solution theory with initial data $\bm{\mc{X}}_0$ in the Wiener space $\mc{F}^{1,1}(\mathbb{S})$, the space of functions whose derivatives have a Fourier series that is absolutely summable (see Section \ref{Notation}). This space is critical with respect to the scaling of the Peskin problem identified above. 

In contrast to \cite{MR3882225,MR3935476}, our theory is restricted to initial data that is sufficiently close to the stationary states, i.e., the uniformly parametrized circles. The papers \cite{MR3882225,MR3935476} establish local-in-time well-posedness in their respective function spaces subject to the following arc-chord condition on the initial data:
\begin{equation}\label{arcchord}
\abs{\bm{\mc{X}}_0}_*\equiv \inf_{\theta,\eta\in \mathbb{S},\theta\neq\eta} \frac{\abs{\bm{\mc{X}}_0(\theta)-\bm{\mc{X}}_0(\eta)}}{\abs{\theta-\eta}}>0.
\end{equation}
In this sense, our results might be better compared to the results on asymptotic stability of the uniformly parametrized circle obtained in  \cite{MR3882225,MR3935476}.
The uniformly parametrized circle is proved to be exponentially stable in the above $L^2$ Sobolev and H\"older scales respectively, and in the latter paper, it is proved that the solution is in $C^\infty(\mathbb{S})$ for all positive time. In this paper, we improve upon this result to prove that the solution is analytic for positive time.

Local-in-time well-posedness for initial data in $\mc{F}^{1,1}$ merely satisfying condition \eqref{arcchord} is an open question that we do not address in this paper. It is notable, however, that the arc-chord condition \eqref{arcchord} is invariant under the scaling described above. In \cite{MR3935476}, it is shown that, if the solution ceases to exist as $t$ approaches $t_*<\infty$, then following must hold:
\begin{equation*}
\lim_{t\to t_*} \varrho_\alpha(\bm{\mc{X}})=\infty, \; \varrho_\alpha(\bm{\mc{X}})=\frac{\norm{\p_\theta \bm{\mc{X}}}_{C^\alpha}}{\abs{\bm{\mc{X}}}_*} \text{ for any } \alpha>0.
\end{equation*}
On the other hand, if $\varrho_\alpha(\bm{\mc{X}})$ remains bounded for all time for some $\alpha>0$, then $\bm{\mc{X}}$ must converge to a uniformly parametrized circle. A similar criterion, in which the numerator of $\varrho_\alpha$ is replaced with a critical norm such as the $\mc{F}^{1,1}$ norm, would be a major improvement that should lead to a better understanding of the global-in-time dynamics of the Peskin problem.

Another extension of the Peskin problem is to consider the following elastic force in place of \eqref{linearF}:
\begin{equation}\label{nonlinearF}
 \bm{F}_{\rm el}=\p_\theta\paren{\mc{T}(\abs{\p_\theta \bm{\mathcal{X}}})\frac{\p_\theta\bm{\mathcal{X}}}{\abs{\p_\theta \bm{\mathcal{X}}}}}
 \end{equation}
 where $\mc{T}(s)$ is a tension coefficient that must satisfy the structure condition $\mc{T}>0$ and $d\mc{T}/ds>0$.
 Note that the above expression is reduced to \eqref{linearF} if we take $\mc{T}(s)=k_0s$, hence $k_0=\mathcal{T}(1)=d\mathcal{T}/ds$.
 In the case of equal viscosity $\mu_1=\mu_2$, a local-in-time well-posedness theory for initial data satisfying \eqref{arcchord} under the more general force \eqref{nonlinearF}
 is established in \cite{rodenberg_thesis} in the H\"older scale
 similarly to \cite{MR3935476}, using nonlinear parabolic methods \cite{MR3012216}. It is expected that the results and methods of this paper 
 can be extended to this more general case.
 
Finally, we mention \cite{1904.09528} in which the author considers a regularization of the Peskin problem inspired by the immersed boundary method, extending the techniques in \cite{MR3882225}. Such studies may form the basis for numerical analysis of the Peskin problem. 
 
The surface tension problem, in which the interface is not elastic but only exerts a surface tension, 
may be the most closely related class of problems for which there are extensive analytical studies. We note that our problem is distinct from 
the surface tension problem; in contrast to an elastic interface considered in the Peskin problem, an interface with surface tension only 
does not resist stretching. This difference manifests itself in the different energy dissipation laws satisfied by the respective problems 
(see Section 1.1 of \cite{MR3935476}). We refer the reader to \cite{MR2571494,MR3524106,MR2953069} 
for an extensive survey of the analytical study of the surface tension problem.

There is also an increasing number of analytical studies on fluid-structure interaction problems in which an elastic structure interacts with a fluid, 
related to the Peskin problem considered here \cite{MR3656704,MR2349865,MR2644917,MR3608168,MR3017292,MR2812947,MR2898775,2005.12036,boulakia2012existence}.  The equations dealt with in these studies are typically more complicated than 
those of the Peskin problem; the sharp results obtained for the simpler Peskin problem should serve as a guide to what is possibly true 
for the more complicated model problems.
 
From an analytical perspective, the Muskat problem is perhaps the closest Nonlinear PDE to our problem for which there is a large body of analytical studies.  However it models a very different physical setting: two immiscible and incompressible fluids in a porous media governed by Darcy's law. 
On the other hand, for a nearly flat interface in the presence of gravity both problems have the same symbol at the linear level.
The authors of \cite{CCGS13} introduced the use of the Wiener algebra to obtain global well-posedness results for the Muskat problem at critical regularity. Moreover, the size restriction on the initial data was given by an explicit constant that is independent of any physical parameter. These techniques were extended in \cite{CCGRS16} and  \cite{1902.02318} to deal with the three-dimensional setting and the case of viscosity jumps, respectively. Other results for the Muskat problem that only require \textit{medium size} initial data in critical spaces (as opposed to the more standard arbitrarily small data condition) \cite{Cam17,Cam20} rely on the maximum principle; these methods have thus far not been shown to be well-suited to deal with viscosity contrasts.

In this paper, we will use spaces related to the Wiener algebra that allow us to perform careful and detailed estimates on the nonlinear terms to control explicitly the size constraint on the initial data (see Figure \ref{fig:pic}).
As opposed to the Muskat problem, here the problem is not only described by the shape of the interface: the parametrization corresponds to the distribution of material points, and thus it matters.  As a consequence, we have to develop further techniques to deal with a system of equations (for both components of the curve).  Interestingly, a careful understanding of the linear system, together with an appropriate change of framework, allows us to decouple the frequencies associated to the projection of the interface onto the space of equilibria from the others.  Indeed, we overcome a major difficulty of
the very recent result \cite{GGPS19}, that deals with the Muskat problem for closed interfaces (i.e. bubbles), and obtain the global existence and uniqueness result for the Peskin problem with viscosity contrast at critical regularity.

\subsection{Notation and Functional Spaces}\label{Notation}

We summarize here the notation and functional spaces that will be used throughout the paper.

For a vector $\bm{x}=(x_1,x_2)^{\rm T}\in \mathbb{C}^2$ we denote
\begin{equation*}
\bm{x}^\perp \eqdef
\mc{R}\bm{x},\quad \mc{R} \eqdef\begin{bmatrix} 0 & -1 \\ 1 & 0 \end{bmatrix},\quad \mc{R}^{-1}=\begin{bmatrix} 0 & 1 \\ -1 & 0 \end{bmatrix}.
\end{equation*}
We denote the Euclidean norm as
\begin{equation*}
|\bm{x}|=\sqrt{\bm{x}^T\overline{\bm{x}}}=\sqrt{|x_1|^2+|x_2|^2},
\end{equation*}
and for a matrix $A=(a_{ij})_{1\leq i,j\leq 2}$ we use the induced matrix norm
\begin{equation}\label{matrixnorm}
\|A\|=\sigma_{\max}(A),
\end{equation}
where $\sigma_{\max}(A)$ is the largest singular value of $A$.
For a vector such as $\bm{X}_*$ we will write $\bm{X}_{*,j}$ to be the $j$-th component of that vector.

We now define the periodic Hilbert transform of a function $f$ with period $2P$ as 
\begin{equation}\label{defHilbertTransform}
    \mH(f)(\theta)\eqdef\frac{1}{2P}\pv\int_{-P}^{P}\frac{f(\theta-\eta)}{\tan{\Big(\frac{\eta}{2P/\pi}\Big)}} d\eta=\frac{1}{4P}\pv\int_{-P}^{P}\frac{f(\theta-\eta)-f(\theta+\eta)}{\tan{\Big(\frac{\eta}{2P/\pi}\Big)}} d\eta.
\end{equation}
Unless stated otherwise, throughout the paper we will use the case $P=\pi$. In this case, we also define the Fourier transform of a periodic function $f$ with domain $\mathbb{S}=[-\pi, \pi]$ as: 
\begin{equation}\notag 
    \mathcal{F}(f)(k)\eqdef\widehat{f}(k)=\frac{1}{2\pi}\int_{-\pi}^\pi f(\theta)e^{-i k\theta}d\theta, \quad k \in \mathbb{Z}.
\end{equation}
Further $\mathcal{F}(\mH(f))(k) = -i \sign(k) \widehat{f}(k)$.  Then we define the operator $\Lambda$ using the Fourier transform as 
$
\mathcal{F}(\Lambda f)(k)\eqdef |k| \widehat{f}(k).
$
And we observe that $\mH(\partial_\theta f )(\theta)= \Lambda f$.    

We denote $f*g$ as the standard convolution of $f$ and $g$.
We use the iterated convolution notation
\begin{equation}\label{iteratedConvlution}
*^kf = \underbrace{f *  \cdots *f}_\text{$k-1$ convolutions of $k$ copies of $f$}
\end{equation}
Thus for instance $*^2f = f*f$.  

We also use the following notation for the discrete delta function,
$\delta_a(k)$, which is the function which is $=1$ when $k=a$ and $=0$ elsewhere.
Throughout the paper we will further denote
\begin{equation}\label{delta2notation}
\delta_{1,-1}(k)=\delta_1(k)+\delta_{-1}(k).
\end{equation} 
We further define the high frequency cut-off operator $\mathcal{J}_M$ for $M\ge 0$ by
\begin{equation}\label{CutOffHigh}
    \widehat{\mathcal{J}_M X}(k) \eqdef 1_{|k|\leq M}\widehat{X}(k).
\end{equation}
where $1_A$ is the standard indicator function of the set $A$, so that $1_A(x)=1$ if $x\in A$ and $1_A(x)=0$ if $x\notin A$.  

For two vectors $\bm{X}(\theta), \bm{Y}(\theta) \in \mathbb{R}^2$  we define 
\begin{equation}\label{innerProductNotation}
\dual{\bm{X}}{\bm{Y}} =\int_\mathbb{S}\bm{X}(\theta)\cdot\bm{Y}(\theta)d\theta.
\end{equation}

Generalizing the Wiener algebra of functions with  absolutely convergent Fourier series as in \cite{GGPS19},
we further define the homogeneous $\dot{\mathcal{F}}^{s,1}_\nu$ and nonhomogeneous $\mathcal{F}^{s,1}_\nu$ norms as
\begin{equation}\label{fsneonenu}
\|\bm{X}\|_{\dot{\mathcal{F}}^{s,1}_\nu}=\sum_{k\in\mathbb{Z}\setminus\{0\}}e^{\nu(t) |k| }|k|^s|\widehat{\bm{X}}(k)|,\qquad s\in \mathbb{R},
\end{equation}
and
\begin{equation}\label{foneonenuinh}
\|\bm{X}\|_{\mathcal{F}^{s,1}_\nu}=|\widehat{\bm{X}}(0)|+\sum_{k\in\mathbb{Z}\setminus\{0\}}e^{\nu(t) |k| }|k|^s|\widehat{\bm{X}}(k)|,\qquad s\geq0,
\end{equation}
with \begin{equation}\label{nu}
    \nu(t)=\nu_m\frac{t}{1+t}\ge 0,
\end{equation} and $\nu_m>0$ chosen arbitrarily small. Note that $\nu(0)=0$, $\nu(t)>0$ for all $t>0$.  Further $\nu'(t)\le \nu_m$ and $\nu(t)\le \nu_m$ bounded for all time.
When $\nu\equiv0$, we denote $\dot{\mathcal{F}}^{s,1}_0=\dot{\mathcal{F}}^{s,1}$ and $\mathcal{F}^{s,1}_0=\mathcal{F}^{s,1}$.
These are the main norms that we will use in this paper.  Note that when $s=1$ the $\dot{\mathcal{F}}^{s,1}$ norm is critical for the Peskin problem.

In this paper we write $A \lesssim B$ if $A \le C B$ for some inessential constant $C>0$.  We also write $A \approx B$ if both $A \lesssim B$ and $B \lesssim A$ hold.
Throughout the paper, we will denote
\begin{equation}\label{constant.def}
    C_i=C_i(\xoneonenu)=C_i\left(\xoneonenu; \nu_m\right)>0, \quad (i=1,2,\ldots),
\end{equation}
as functions that are increasing in $\xoneonenu\geq0$ and might depend on the analyticity constant $\nu_m$, 
 with the properties that $C_i(\xoneonenu)\approx 1$  for all  $\nu_m\geq0$ and $\lim_{\xoneonenu\to0^+}C_i(\xoneonenu;0)= 1$. We will also denote 
 \begin{equation*}
 D_i=D_i(\xoneonenu)=D_i\left(\xoneonenu; A_\mu,\nu_m\right)>0, \quad (i=1,2,\ldots),  
 \end{equation*}
 as functions that are increasing in $\xoneonenu\geq0$ and might depend on the physical parameter $A_\mu$ and the analyticity constant $\nu_m$, 
 with the properties that $D_i(\xoneonenu)\approx 1$  for all  $A_\mu\in(-1,1)$ and all $\nu_m\geq0$, and $\lim_{\xoneonenu\to0^+}D_i(\xoneonenu;0,0)= 1$.

\subsection{Main Results}\label{MainResults}
In this section we will state the main result of this paper: namely, that  membranes whose initial interface has critical regularity (in terms of the scaling of the problem), and that are not too far from an equilibrium configuration, become instantaneously analytic and converge exponentially fast to the equilibrium. Without loss of generality, we assume that the initial area enclosed by the membrane is $\pi$. We get the result under an explicit medium-size condition for the initial deviation and for general viscosity contrast $A_\mu\in(-1,1)$.

\begin{defn}[Strong solution]\label{strongsol}
Let
\begin{equation*}
    \bm{\mathcal{X}}\in  C([0,T];\mathcal{F}^{1,1}) \cap C^1((0,T];\mathcal{F}^{0,1}), 
\end{equation*}
and 
\begin{equation*}
    |\bm{\mathcal{X}}|_*(t)=\inf_{\theta, \eta\in\mathbb{S},\theta\neq\eta}\frac{|\bm{\mathcal{X}}(\theta,t)-\bm{\mathcal{X}}(\eta,t)|}{|\theta-\eta|}>0,
\end{equation*}
for $0\leq t\leq T$.
Then, $\bm{\mathcal{X}}$ is a {\em{strong solution}} to the  viscosity-contrast Peskin problem with initial value $\bm{\mathcal{X}}(0)=\bm{\mathcal{X}}_0$ if it satisfies \eqref{Xteqn}, \eqref{viscosityjump} for $0<t\leq T$ and $\bm{\mathcal{X}}(t)\to \bm{\mathcal{X}}_0$ in $\mathcal{F}^{1,1}$ as $t\to0$.  
\end{defn}

\begin{thm}[Main Result]\label{MainTheorem}
	Let $A_\mu\in(-1,1)$ and  $\bm{\mathcal{X}}_0\in \mathcal{F}^{1,1}$. Let $\bm{X}_{0,c}$  be the projection of $\bm{\mathcal{X}}_0$ onto the vector space spanned by \eqref{vectors} and $\bm{X}_0=\bm{\mathcal{X}}_0-\bm{X}_{0,c}$, thus $\bm{X}_0$ is mean zero and $\widehat{\bm{\mathcal{X}}}_0(0)=\widehat{\bm{X}}_{0,c}(0)$.
	Assume that initially the deviation $\bm{X}_{0}$ satisfies the medium-size condition 
    \begin{equation}\label{conditiontheorem}
     \|\bm{X}_0\|_{\dot{\mathcal{F}}^{1,1}}<k(A_\mu),
    \end{equation}
    where  $k(A_\mu)>0$ is defined in \eqref{kdef} (see also \eqref{kmuform} and Figure \ref{fig:pic}), and that the area enclosed by $\bm{\mathcal{X}}_0$ is $\pi$.
	Then, for any $T>0$, there exists a {\em{unique global strong solution}}   $\bm{\mathcal{X}}(t)$ to the system  \eqref{Xteqn} and \eqref{viscosityjump}, which lies in the space
	\begin{equation*}	
	\bm{\mathcal{X}}\in  C([0,T];\mathcal{F}^{1,1}_\nu) \cap C^1((0,T];\mathcal{F}^{0,1}_\nu) \cap L^1([0,T];\dot{\mathcal{F}}^{2,1}_\nu),
	\end{equation*}
	with $\nu(t)$ given by \eqref{nu}.  
	In particular, it becomes instantaneously analytic. Moreover, the following energy inequality is satisfied for $0\leq t\leq T$:
	\begin{equation}\label{balanceX}
    	\xoneonenu(t)+  \frac{A_e}4 \dissconst\int_0^t \xtwoonenu(\tau) d\tau \leq \|\bm{X}_0\|_{\foneone},
	\end{equation}
	with $\dissconst=\dissconst\big(\|\bm{X}_0\|_{\foneone},A_\mu,\nu_m\big)>0$ defined in \eqref{C}. 
       In addition, 
	\begin{equation}\label{decay2}
	    \begin{aligned}
	        \xoneonenu(t)\leq \|\bm{X}_0\|_{\foneone}e^{-\frac{A_e}4 \dissconst  t}.
	    \end{aligned}
	\end{equation}
	The zero frequency $\widehat{\bm{X}_c}(0)$ remains uniformly bounded for all times as follows 
	\begin{equation*}
	     |\widehat{\bm{X}_c}(0)|\leq |\widehat{\bm{X}}_{0,c}(0)|+
	     \tilde{C}
	     \|\bm{X}_0\|_{\foneone}^2,
	\end{equation*}
with $\tilde{C}=\tilde{C}\big(\|\bm{X}_0\|_{\foneone},A_\mu\big)>0$ given in \eqref{Ctilde}, while
\begin{equation}\label{auxone}
    1-\frac12\|\bm{X}\|_{\foneonenu}^2\leq |\widehat{\bm{X}}_{c}(1)|^2\leq 1+\frac12\|\bm{X}\|_{\foneonenu}^2.
\end{equation}
\end{thm}

We remark that the decay to zero of the deviation $\bm{X}$ in \eqref{decay2} together with \eqref{auxone} shows the exponentially fast convergence to a uniformly parametrized circle with the same area as the initial one.

\begin{figure}[h]\centering
	\includegraphics[width=.85\textwidth]{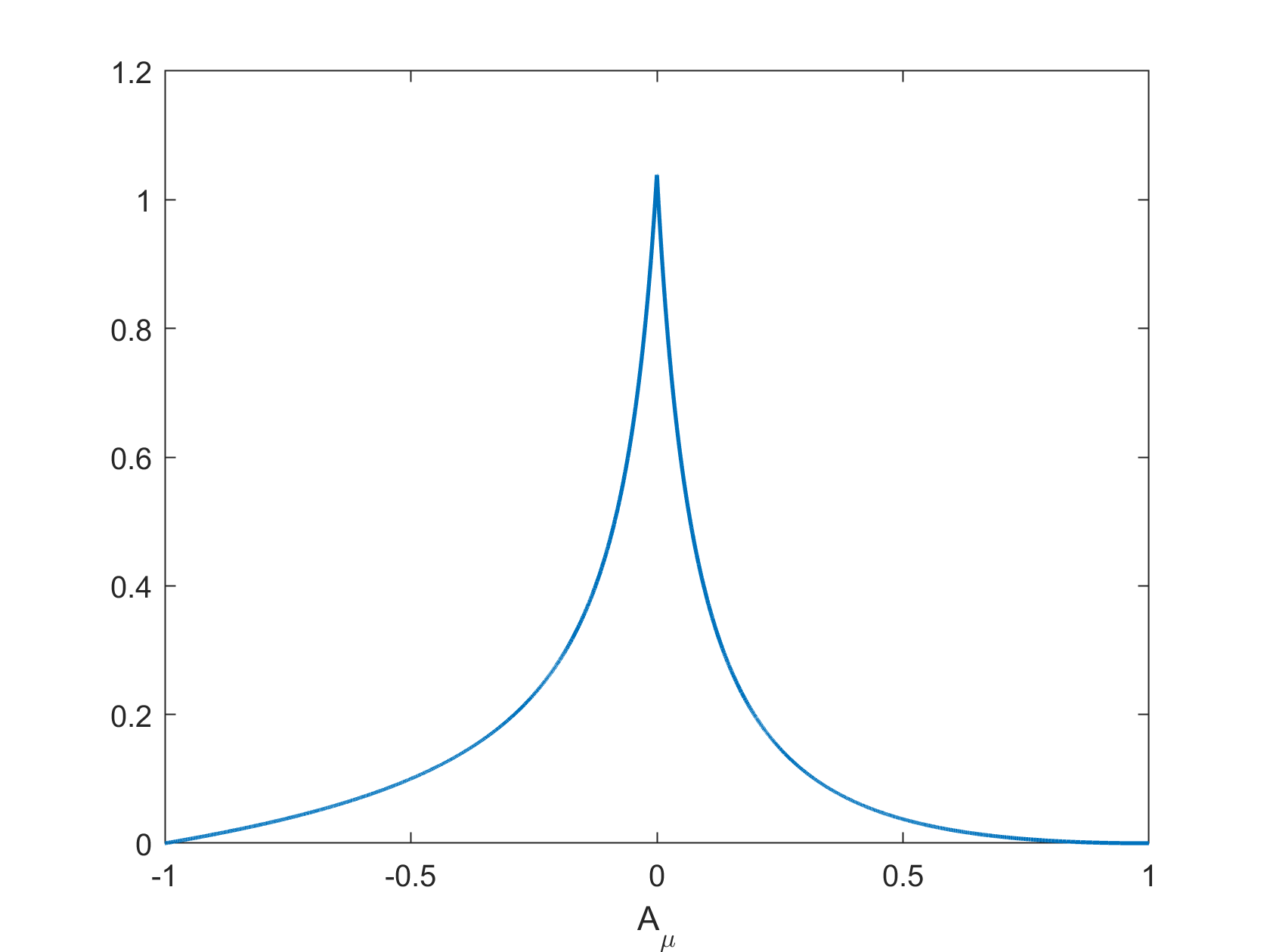}\vspace{-0.4cm}
	\caption{$10^3 k(A_\mu)$}\label{fig:pic}
\end{figure}

\begin{remark} In our results, we assume that both viscosities $\mu_1$ and $\mu_2$ are positive and hence $-1<A_\mu<1$. We remark on the endpoint cases of $A_\mu=\pm 1$, which formally correspond to the cases when $\mu_1=0$ or $\mu_2=0$. As can be seen from Fig. \ref{fig:pic}, the allowed size of the deviation from $\bm{X}_0$ tends to $0$ as $A_\mu\to \pm 1$, which may indicate potential difficulties in formulating a well-posed mathematical problem for the endpoint cases. From a physical standpoint, it does not make sense to set the viscosity to $0$ in either $\Omega_1$ or $\Omega_2$, and thus a proper treatment of these endpoint cases will require a rethinking of the physical situation under consideration. The case $A_{\mu}=-1$ or $\mu_2=0$ may be thought of as corresponding to the problem in which a droplet of Stokesian fluid is floating in vacuum. One significant difference between this and the Peskin problem is that in the former problem a droplet in linear translation or rigid rotation experiences no external forces. The force balance and continuity equations will thus have to be supplemented by auxiliary conditions that assure uniqueness, after which this problem is likely to be well-posed. In the case $A_\mu=1$ or $\mu_1=0$, $\Omega_1$ might be considered to be vacuum. It is not clear if this problem is well-posed. We will not pursue these issues further.

\end{remark}


\subsection{Outline}

The rest of the paper is structured as follows.  In Section \ref{formulation}, we first decompose in \ref{nonlinearexp} the system \eqref{Xteqn}, \eqref{viscosityjump} into zero order, linear, and nonlinear parts around the equilibrium configuration, and then in Section \ref{secLinearization} we perform the linearization of the problem and show its parabolic structure. Section \ref{evolF11} shows how this structure leads to dissipation and in Section \ref{completesys} we summarize the system of equations in its final form.
Section \ref{sec:secanalytic} contains the crucial nonlinear estimates needed to prove Theorem \ref{MainTheorem}. Finally, Section \ref{secMain} is dedicated to the proof of Theorem \ref{MainTheorem} via a regularization argument and also shows the uniqueness of the solutions.

\section{Linearization around the steady state}\label{formulation}

We will linearize the system \eqref{G}-\eqref{FadAel}, with $\bm{F}_{\text{el}}$ given by \eqref{linearF}, around a time-dependent uniformly parametrized circle with center $(c(t),d(t))$ and radius $R(t)$: 
\begin{equation}\label{unifcircles}
\bm{X}_c(\theta,t)=a(t)\bm{e_r}(\theta)+b(t)\bm{e_t}(\theta)+c(t)\bm{e_1}+d(t)\bm{e_2},
\end{equation}
\begin{equation}\notag
R^2(t)=a^2(t)+b^2(t),
\end{equation}
where $a(t)$, $b(t)$, $c(t)$ and $d(t)$ are arbitrary time dependent functions and $\bm{e_r}(\theta)$, $\bm{e_t}(\theta)$, $\bm{e_1}(\theta)$, $\bm{e_2}(\theta)$ are defined in \eqref{vectors}. 
For notational convenience, we will suppress the time dependence of the coefficients.
\vspace{0.2cm}

\subsection{Nonlinear Expansion}\label{nonlinearexp}

We will denote by $\bm{X}(\theta)$ the deviation from the circle $\bm{X}_c(\theta)$ as $\bm{X}(\theta)=\bm{\mathcal{X}}(\theta)-\bm{X}_c(\theta)$.  We define further
$$
\Delta \bm{X} \eqdef \bm{X}(\theta)-\bm{X}(\eta),
$$
and
\begin{equation}\label{deltaeta}
\Delta_\eta \bm{X}(\theta) \eqdef \frac{\bm{X}(\theta)-\bm{X}(\eta)}{2\sin{\left(\frac{\theta-\eta}{2}\right)}}.
\end{equation}
In particular, we have
\begin{equation*}
\Delta_\eta \bm{X}_c(\theta)=a\hspace{0.05cm}\bm{e_t}\thetaeta-b\hspace{0.05cm} \bm{e_r}\thetaeta,
\end{equation*}
since
\begin{equation*}
\begin{aligned}
\Delta_\eta \bm{e_r}(\theta)&=\begin{bmatrix}
-\sin{\left(\frac{\theta+\eta}{2}\right)}\\
\cos{\left(\frac{\theta+\eta}{2}\right)}
\end{bmatrix}\!=\bm{e_t}\thetaeta,\\
\Delta_\eta \bm{e_t}(\theta)&=\begin{bmatrix}
-\cos{\left(\frac{\theta+\eta}{2}\right)}\\
-\sin{\left(\frac{\theta+\eta}{2}\right)}
\end{bmatrix}\!=-\bm{e_r}\thetaeta,
\end{aligned}
\end{equation*}
where we have used the trigonometric identities 
\begin{equation*}
\begin{aligned}
\sin{(\theta)}-\sin{(\eta)}&=2\sin{\Big(\frac{\theta-\eta}2\Big)}\cos{\Big(\frac{\theta+\eta}{2}\Big)},\\
\cos{(\theta)}-\cos{(\eta)}&=-2\sin{\Big(\frac{\theta-\eta}2\Big)}\sin{\Big(\frac{\theta+\eta}{2}\Big)}.
\end{aligned}
\end{equation*}
Recalling \eqref{unifcircles} and using the identities
\begin{equation*}
\begin{aligned}
\p_\theta\bm{e_r}(\theta)&=\bm{e_r}(\theta)^\perp=\bm{e_t}(\theta),\\ \p_\theta\bm{e_t}(\theta)&=\bm{e_t}(\theta)^\perp=-\bm{e_r}(\theta),
\end{aligned}
\end{equation*}
one has that
\begin{equation*}
\begin{aligned}
\p_\theta\bm{X}_c(\theta)&=a\bm{e_t}(\theta)-b\bm{e_r}(\theta),\\
\p_\theta\bm{X}_c(\theta)^\perp&=-a\bm{e_r}(\theta)-b\bm{e_t}(\theta).
\end{aligned}
\end{equation*}
The trigonometric identities $\cos(a+b)=\cos(a)\cos(b)-\sin(a)\sin(b)$ and $\sin(a+b)=\sin(a)\cos(b)-\cos(a)\sin(b)$ further give that
\begin{equation*}
\begin{aligned}
\bm{e_r}(\theta)\cdot\bm{e_t}\Big(\frac{\theta+\eta}{2}\Big)&=\sin{\Big(\frac{\theta-\eta}{2}\Big)},\\
\bm{e_r}(\theta)\cdot\bm{e_r}\Big(\frac{\theta+\eta}{2}\Big)&=\cos{\Big(\frac{\theta-\eta}{2}\Big)},\\
\bm{e_t}(\theta)\cdot\bm{e_t}\Big(\frac{\theta+\eta}{2}\Big)&=\cos{\Big(\frac{\theta-\eta}{2}\Big)},\\
\bm{e_t}(\theta)\cdot\bm{e_r}\Big(\frac{\theta+\eta}{2}\Big)&=-\sin{\Big(\frac{\theta-\eta}{2}\Big)}.
\end{aligned}
\end{equation*}
These calculations imply the following computations for a circle that will be used frequently throughout the paper
\begin{equation}\label{comput1}
\partial_\theta \bm{X}_c(\theta)^\perp\cdot\Delta_\eta \bm{X}_c(\theta)=-R^2\sin{\Big(\frac{\theta-\eta}{2}\Big)},
\end{equation}
\begin{equation}\label{comput2} \partial_\theta \bm{X}_c(\theta)\cdot\Delta_\eta \bm{X}_c(\theta)=R^2\cos{\Big(\frac{\theta-\eta}{2}\Big)},
\end{equation}
\begin{multline}\label{comput3}
\Delta_\eta \bm{X}_c(\theta)\otimes \Delta_\eta \bm{X}_c(\theta)=\frac{a^2}2\begin{bmatrix}
1-\cos{(\theta+\eta)}&-\sin{(\theta+\eta)}\\
-\sin{(\theta+\eta)}&1+\cos{(\theta+\eta)}
\end{bmatrix}\\
+\frac{b^2}2\begin{bmatrix}
1\!+\!\cos{(\theta\!+\!\eta)}&\sin{(\theta\!+\!\eta)}\\
\sin{(\theta\!+\!\eta)}&1\!-\!\cos{(\theta\!+\!\eta)}
\end{bmatrix}\!+\!ab\begin{bmatrix}
\sin{(\theta\!+\!\eta)}&\!-\!\cos{(\theta\!+\!\eta)}\\
-\cos{(\theta\!+\!\eta)}&\!-\!\sin{(\theta\!+\!\eta)}
\end{bmatrix}.
\end{multline}
The matrices in the last line above have been simplified using the identities $\sin^2(a)=\frac{1-\cos(2a)}2$, $\cos^2(a)=\frac{1+\cos(2a)}2$, and $\sin(2a)=2\sin(a)\cos(a)$.

Next, we perform a Taylor expansion of the nonlinear terms around the a time-dependent uniformly parametrized circle \eqref{unifcircles} under the assumption that $|\Delta_\eta \bm{X}(\theta)|<1$. First, we start with the magnitude of the curve
\begin{equation*}
\begin{aligned}
|\Delta \bm{X}\!+\!\Delta \bm{X}_c|^2\!=\!4R^2\sin^2{\left(\frac{\theta\!-\!\eta}{2}\right)}\Big(1\!+\!\frac2{R^2}\Delta_\eta \bm{X}_c(\theta)\cdot\Delta_\eta \bm{X}(\theta)\!+\!\frac1{R^2}|\Delta_\eta \bm{X}(\theta)|^2\Big).
\end{aligned}
\end{equation*}
Recalling the expression for $G(\Delta \bm{\mathcal{X}})$ in \eqref{G}, we expand each term
\begin{equation}\label{G1exp}
\begin{aligned}
\log{|\Delta \bm{X}\!+\!\Delta \bm{X}_c|}&\!=\!\log{\!\left(\!2R\left|\sin{\!\left(\!\frac{\theta\!-\!\eta}{2}\right)}\!\right|\right)}\\
&\quad+\!\frac12\log{\!\big(1\!+\!\frac{2}{R^2}\Delta_\eta \bm{X}_c(\theta)\!\cdot\!\Delta_\eta \bm{X}(\theta)\!+\!\frac{1}{R^2}|\Delta_\eta \bm{X}(\theta)|^2 \big)}
\\
&\hspace{-2.3cm}=\log{\left(\!2R\left|\sin\!{\left(\frac{\theta\!-\!\eta}{2}\right)}\right|\right)}\!+\frac{1}{R^2}\Delta_\eta \bm{X}_c(\theta)\!\cdot\!\Delta_\eta \bm{X}(\theta)\!+\!\mathcal{R}_1(\Delta_\eta \bm{X}(\theta)),
 \end{aligned}
\end{equation}
where
\begin{equation}\label{R1}
\mathcal{R}_1(\Delta_\eta \bm{X}(\theta))
\!=\!\!\sum_{n=1}^\infty\!\!\!\sum_{\substack{m=0 \\ n+m\geq 2}}^n\!\!\!\!\!\begin{pmatrix}
n\\m
\end{pmatrix} \!\!\frac{(-1)^{n-1}}{2n R^{2n}}(2\Delta_\eta \bm{X}_c(\theta)\cdot \Delta_\eta \bm{X}(\theta))^{n-m}|\Delta_\eta \bm{X}(\theta)|^{2m}.
\end{equation}
We expand the denominator in the second term of \eqref{G} as
\begin{multline}\label{expand.d}
\frac{1}{|\Delta \bm{X}+\Delta \bm{X}_c|^2} = \frac1{4R^2\sin^2{\left(\frac{\theta-\eta}{2}\right)}}\Big(1\!-\!\frac2{R^2}\Delta_\eta \bm{X}_c(\theta)\cdot \Delta_\eta \bm{X}(\theta)\Big)
\\
+
\frac1{4R^2\sin^2{\left(\frac{\theta-\eta}{2}\right)}}\mathcal{R}_2(\Delta_\eta \bm{X}(\theta)),
\end{multline}
with the notation
\begin{equation}\label{R2}
\mathcal{R}_2(\Delta_\eta \bm{X}(\theta))
=
\!\sum_{n=1}^\infty\!\!\sum_{\substack{m=0 \\ n+m\geq 2}}^n\!\!\!\begin{pmatrix}
n\\m
\end{pmatrix}\! \frac{(-1)^{n}}{R^{2n}}(2\Delta_\eta \bm{X}_c(\theta)\cdot \Delta_\eta \bm{X}(\theta))^{n-m}|\Delta_\eta \bm{X}(\theta)|^{2m}.
\end{equation}
Therefore, we can write 
\begin{equation}\label{G2exp}
\begin{aligned}
\frac{(\Delta \bm{X}+\Delta \bm{X}_c)\otimes(\Delta \bm{X}+\Delta \bm{X}_c)}{|\Delta \bm{X}+\Delta \bm{X}_c|^2}=A_0+A_L+A_N,
\end{aligned}
\end{equation}
with
\begin{equation*}
\begin{aligned}
A_0&=\frac{1}{R^2}\Delta_\eta \bm{X}_c(\theta)\otimes \Delta_\eta \bm{X}_c(\theta),\\
A_L&=-\frac2{R^4}\Delta_\eta \bm{X}_c(\theta)\cdot \Delta_\eta \bm{X}(\theta)\Delta_\eta \bm{X}_c(\theta)\otimes \Delta_\eta \bm{X}_c(\theta)
+\frac1{R^2}\Delta_\eta \bm{X}_c(\theta)\otimes\Delta_\eta \bm{X}(\theta)\\
&\quad+\frac1{R^2}\Delta_\eta \bm{X}(\theta)\otimes\Delta_\eta \bm{X}_c(\theta),
\end{aligned}
\end{equation*}
and the nonlinear term is given by
\begin{equation*}
\begin{aligned}
A_N&=\frac1{R^2}\Delta_\eta \bm{X}(\theta)\otimes\Delta_\eta \bm{X}(\theta)
\left(1-\frac2{R^2}\Delta_\eta \bm{X}_c(\theta)\cdot \Delta_\eta \bm{X}(\theta)+\mathcal{R}_2(\Delta_\eta \bm{X}(\theta))\right)
\\
&\quad+\frac1{R^2}\Delta_\eta \bm{X}_c(\theta)\!\otimes\!\Delta_\eta \bm{X}(\theta)
\!\left(\!\!-\!\frac{2}{R^2}\Delta_\eta \bm{X}_c(\theta)\!\cdot\! \Delta_\eta \bm{X}(\theta)\!+\!\mathcal{R}_2(\Delta_\eta \bm{X}(\theta)) \right)
\\
&\quad+\frac1{R^2}\Delta_\eta \bm{X}(\theta)\!\otimes\!\Delta_\eta \bm{X}_c(\theta)\!
\left(\!\!-\!\frac{2}{R^2}\Delta_\eta \bm{X}_c(\theta)\!\cdot\! \Delta_\eta \bm{X}(\theta)\!+\!\mathcal{R}_2(\Delta_\eta \bm{X}(\theta)) \right)
\\
&\quad+\frac1{R^2}\Delta_\eta \bm{X}_c(\theta)\otimes \Delta_\eta \bm{X}_c(\theta) \mathcal{R}_2(\Delta_\eta \bm{X}(\theta)).
\end{aligned}
\end{equation*}
Joining the expansions \eqref{G1exp} and \eqref{G2exp}, we  split $G(\Delta \bm{\mathcal{X}})$ in \eqref{G} into zero order, linear, and nonlinear parts in terms of $\bm{X}$ as follows
\begin{multline}\label{Gsplit}
G(\Delta \bm{\mathcal{X}})
\\
=G_0(\Delta_\eta\bm{X}_c(\theta))+G_L(\Delta_\eta\bm{X}_c(\theta),\Delta_\eta\bm{X}(\theta))+G_N(\Delta_\eta\bm{X}_c(\theta),\Delta_\eta\bm{X}(\theta)),
\end{multline}
where
\begin{multline}\label{G0}
G_0(\Delta_\eta\bm{X}_c(\theta))
\\
=\frac{1}{4\pi}\left(-\log{\left(2R\left|\sin{\left(\frac{\theta-\eta}{2}\right)}\right|\right)}I +\frac1{R^2}\Delta_\eta \bm{X}_c(\theta)\otimes\Delta_\eta \bm{X}_c(\theta)\right),
\end{multline}
\begin{equation}\label{G1}
\begin{aligned}
G_L(\Delta_\eta\bm{X}_c(\theta),\Delta_\eta\bm{X}(\theta))&=\frac{1}{4\pi R^2}\Big(\!-\!\Delta_\eta \bm{X}_c(\theta)\!\cdot\!\Delta_\eta \bm{X}(\theta)I\!\\
&\quad-\!\frac2{R^2}\Delta_\eta \bm{X}_c(\theta)\!\cdot\! \Delta_\eta \bm{X}(\theta)\Delta_\eta \bm{X}_c(\theta)\!\otimes\! \Delta_\eta \bm{X}_c(\theta)\\
&\quad+\Delta_\eta \bm{X}_c(\theta)\otimes\Delta_\eta \bm{X}(\theta)+\Delta_\eta \bm{X}(\theta)\otimes\Delta_\eta \bm{X}_c(\theta) \Big),
\end{aligned}
\end{equation}
\begin{multline}\label{GN}
G_N(\Delta_\eta\bm{X}_c(\theta),\Delta_\eta\bm{X}(\theta))=\frac{-1}{4\pi}\mathcal{R}_1(\Delta_\eta \bm{X}(\theta))I\\+\frac1{4\pi R^2}\Delta_\eta \bm{X}(\theta)\otimes\Delta_\eta \bm{X}(\theta)\Big(1-\frac2{R^2}\Delta_\eta \bm{X}_c(\theta)\cdot \Delta_\eta \bm{X}(\theta)+\mathcal{R}_2(\Delta_\eta \bm{X}(\theta))\Big)\\
+\frac1{4\pi R^2}\Big(\Delta_\eta \bm{X}_c(\theta)\otimes\Delta_\eta \bm{X}(\theta)+\Delta_\eta \bm{X}(\theta)\otimes\Delta_\eta \bm{X}_c(\theta)\Big)
\\
\quad \times \big(-\frac2{R^2}\Delta_\eta \bm{X}_c(\theta)\cdot \Delta_\eta \bm{X}(\theta)+\mathcal{R}_2(\Delta_\eta \bm{X}(\theta)) \big)
\\
+\frac{1}{4\pi R^2}\Delta_\eta \bm{X}_c(\theta)\otimes \Delta_\eta \bm{X}_c(\theta) \mathcal{R}_2(\Delta_\eta \bm{X}(\theta)).
\end{multline}
Consider the splitting of the solution $\mb{F}(\theta)$ to \eqref{viscosityjump} into zero order, linear, and nonlinear parts as
\begin{equation}\label{Fsplit}
\bm{F}(\theta)=\bm{F}_0(\theta)+\bm{F}_L(\theta)+\bm{F}_N(\theta).
\end{equation}
(We will prove bounds for these terms in Section \ref{secF}.) 
Introducing the splittings \eqref{Gsplit} and \eqref{Fsplit} in \eqref{Xteqn}, we obtain that
\begin{equation}\label{nonlinearexpansion}
\begin{aligned}
\bm{\mathcal{X}}_t(\theta)&=\bm{\mathcal{O}}(\bm{X}_c)(\theta)+\bm{\mathcal{L}}(\bm{X}_c,\bm{X})(\theta)+\bm{\mathcal{N}}(\bm{X}_c,\bm{X})(\theta),
\end{aligned}
\end{equation}
where we recall that 
$\bm{\mathcal{X}}(\theta)=\bm{X}(\theta)+\bm{X}_c(\theta)$ and we use the notation
\begin{equation*}
\begin{aligned}
\bm{\mathcal{O}}(\bm{X}_c)(\theta)=\int_{\mathbb{S}} G_0(\Delta_\eta\bm{X}_c(\theta))\bm{F}_0(\eta)d\eta,
\end{aligned}
\end{equation*}
\begin{equation*}
\begin{aligned}
\!\bm{\mathcal{L}}(\bm{X}_c,\bm{X})(\theta)=\int_{\mathbb{S}} G_0(\Delta_\eta\bm{X}_c(\theta))\bm{F}_L(\eta)d\eta+\int_{\mathbb{S}} G_L(\Delta_\eta\bm{X}_c(\theta),\Delta_\eta\bm{X}(\theta))\bm{F}_0(\eta)d\eta,
\end{aligned}
\end{equation*}
\begin{equation*}
\begin{aligned}
\!\bm{\mathcal{N}}(\bm{X}_c,\bm{X})(\theta)\!&=\!\!\!\int_{\mathbb{S}}\!\! \Big(G_L(\Delta_\eta\bm{X_c}(\theta),\Delta_\eta\bm{X}(\theta))\bm{F}_L(\eta)\!+\!G_N(\Delta_\eta\bm{X}_c(\theta),\Delta_\eta\bm{X}(\theta))\bm{F}_0(\eta)\\
&\quad+\!G_N(\Delta_\eta\bm{X}_c(\theta),\Delta_\eta\bm{X}(\theta))\bm{F}_L(\eta)\!+\!G(\bm{\mathcal{X}}(\theta)\!-\!\bm{\mathcal{X}}(\eta))\bm{F}_N(\eta)\Big)d\eta.
\end{aligned}
\end{equation*}
We have thus expanded the evolution equation \eqref{Xteqn} distinguishing the zero order, linear, and nonlinear in $\bm{X}$ contributions.

\subsection{Linearized System}\label{secLinearization}

We proceed to show that the linearized system gives rise to a diffusion operator on $\bm{X}$. Since the linear structure is the same for any uniformly parametrized circle (see \cite{MR3935476}), we will use now \eqref{unifcircles} with $a=1$, $b=c=d=0$ and $R=1$ to simplify the computations, and for clarity of notation we will denote this circle by $\bm{X}_\star$.

We will now linearize equations \eqref{Xteqn} and \eqref{viscosityjump}. 
We first determine $\bm{F}_0$,  the value of $\bm{F}$ at the steady state:
\begin{align}
\label{Fstar1}
&0=\bm{\mathcal{O}}(\bm{X}_\star)(\theta)=\int_{\mathbb{S}} G_0(\Delta_\eta\bm{X}_\star(\theta))\bm{F}_0(\eta)d\eta,\\
\label{Fstar2}
&\bm{F}_0(\theta)-2A_\mu\bm{\mathcal{S}_0}(\bm{F}_0,\bm{X}_\star)(\theta)=2A_e \bm{\tilde{F}}_{\text{el},0}(\theta),
\end{align}
where
\begin{equation*}
\bm{\tilde{F}}_{\text{el},0}(\theta)=\p_\theta^2\bm{X}_\star(\theta),
\end{equation*}
and
\begin{equation*}
\bm{\mathcal{S}_0}(\bm{F}_0,\bm{X}_\star)(\theta)\!=\!\frac{1}{\pi }\int_\mathbb{S} \p_\theta\bm{X}_\star^\perp\cdot\Delta_\eta \bm{X}_\star(\theta)\Delta_\eta \bm{X}_\star(\theta)\Delta_\eta \bm{X}_\star(\theta)\cdot\frac{\bm{F}_0(\eta)}{2\sin{\Big(\frac{\theta-\eta}{2}\Big)}}d\eta.
\end{equation*}
Rewriting $\Delta_\eta \bm{X}_\star(\theta)\Delta_\eta \bm{X}_\star(\theta)\cdot\bm{F}_0(\eta)=\Delta_\eta \bm{X}_\star(\theta)\otimes\Delta_\eta \bm{X}_\star(\theta)\bm{F}_0(\eta)$,
and recalling the computations \eqref{comput1} and \eqref{comput3}, one finds that
\begin{equation*}
\bm{\mathcal{S}_0}(\bm{F}_0,\bm{X}_\star)(\theta)\!=\!\frac{-1}{2\pi }\int_\mathbb{S}M(\theta, \eta)\bm{F}_0(\eta)d\eta,
\end{equation*}
where from \eqref{comput3} we have
\begin{equation}\label{matrixM}
M(\theta,\eta)\!=\!\Delta_\eta \bm{X}_\star(\theta)\!\otimes\! \Delta_\eta \bm{X}_\star(\theta)\!=\!\frac{1}2\begin{bmatrix}
1\!-\!\cos{(\theta\!+\!\eta)}&-\sin{(\theta\!+\!\eta)}\\
-\sin{(\theta\!+\!\eta)}&1\!+\!\cos{(\theta\!+\!\eta)}
\end{bmatrix},
\end{equation}
and therefore $\bm{F}_0$ is defined by
\begin{equation*}
\bm{F}_0(\theta)+\frac{A_\mu}{\pi}\int_\mathbb{S}M(\theta, \eta)\bm{F}_0(\eta)d\eta=2A_e \p_\theta^2\bm{X}_\star(\theta).
\end{equation*}
Since $\p_\theta^2\bm{X}_\star=-\bm{X}_\star$ and noting that 
\begin{equation*}
\frac{1}{\pi}\int_\mathbb{S} M(\theta,\eta)\bm{X}_\star(\eta)d\eta=-\bm{X}_\star(\theta),
\end{equation*}
it is easily seen that 
\begin{equation}\label{F0star}
\bm{F}_0(\theta)=\frac{2A_e}{1-A_\mu}\p_\theta^2 \bm{X}_\star(\theta).
\end{equation}
Now, recalling \eqref{G0}, it can be checked that \eqref{F0star} satisfies in fact \eqref{Fstar1}: 
\begin{multline*}
4\pi\left(\frac{1-A_\mu}{2A_e} \right)
\int_{\mathbb{S}} G_0(\Delta_\eta\bm{X}_\star(\theta))\bm{F}_0(\eta)d\eta
\\
=
-\int_\mathbb{S}\log{\left(2\left|\sin{\left(\frac{\theta-\eta}{2}\right)}\right|\right)}\p_\eta^2\bm{X}_\star(\eta)d\eta 
+
\int_\mathbb{S}M(\theta,\eta)\p_\eta^2\bm{X}_\star(\eta)d\eta,
\end{multline*}
so integration by parts in the first term yields \eqref{Fstar1}.

We now proceed to compute the linear term $\bm{\mathcal{L}}(\bm{X}_\star,\bm{X})(\theta)$ in \eqref{nonlinearexpansion}. For convenience, we write it as follows
\begin{equation}\label{XtUint}
\bm{\mathcal{L}}(\bm{X}_\star,\bm{X})(\theta)\!=\!\int_{\mathbb{S}}\!\! G_0(\Delta_\eta \bm{X}_\star)\bm{F}_L(\eta)d\eta
+\!\int_{\mathbb{S}}(\nabla G(\Delta \bm{X}_\star)\bm{F}_0(\eta))\Delta \bm{X}d\eta,
\end{equation}
where $G_0$ and $G$ are defined in \eqref{G0} and \eqref{G}, respectively.
To simplify the second integral above, note that:
\begin{equation*}
\begin{split}
\PD{G_{ij}}{x_1}(\Delta \bm{X}_\star)e_{{\rm r},j}(\eta)
&=\PD{G_{i1}}{x_1}(\Delta \bm{X}_\star)e_{\rm r,1}(\eta)+\PD{G_{i2}}{x_1}(\Delta \bm{X}_\star)e_{\rm r,2}(\eta)\\
&=\PD{G_{i1}}{x_1}(\Delta \bm{X}_\star)\p_\eta X_{\star,2}-\PD{G_{i2}}{x_1}(\Delta \bm{X}_\star)\p_\eta X_{\star, 1}\\
&=-\PD{G_{i2}}{x_2}(\Delta \bm{X}_\star)\p_\eta X_{\star,2}-\PD{G_{i2}}{x_1}(\Delta \bm{X}_\star)\p_\eta X_{\star, 1}\\
&=\p_\eta G_{i2}(\Delta \bm{X}_\star).
\end{split}
\end{equation*}
Here $e_{{\rm r},j}$ is the $j$-th component of the vector $e_{{\rm r}}$ etc. Further, in the third equality above, we used the fact that the Stokeslet is divergence free:
\begin{equation*}
\PD{G_{i1}}{x_1}+\PD{G_{i2}}{x_2}=0.
\end{equation*}
Likewise, we have:
\begin{equation*}
\PD{G_{ij}}{x_2}(\Delta \bm{X}_\star)e_{{\rm r},j}(\eta)=-\p_\eta G_{i1}(\Delta \bm{X}_\star).
\end{equation*}
We thus have:
\begin{equation*}
\begin{split}
\int_{\mathbb{S}}(\nabla G(\Delta \bm{X}_\star)\bm{F}_0(\eta))\Delta \bm{X}d\eta&=-\frac{2A_e}{1\!-\!A_\mu}\!\int_{\mathbb{S}}(\mc{R}^{-1}\p_\eta G(\bm{X}_{\star}(\theta)\!-\!\bm{X}_{\star}(\eta)))\Delta \bm{X}d\eta\\
&=\frac{2A_e}{1\!-\!A_\mu}\int_{\mathbb{S}}(\mc{R}^{-1}G(\bm{X}_{\star}(\theta)\!-\!\bm{X}_{\star}(\eta)))\p_\eta \bm{X}(\eta)d\eta\\
&=\frac{2A_e}{1\!-\!A_\mu}\!\int_{\mathbb{S}}G(\bm{X}_{\star}(\theta)\!-\!\bm{X}_{\star}(\eta))(\mc{R}^{-1}\p_\eta \bm{X}(\eta))d\eta.
\end{split}
\end{equation*}
Since $G(\bm{X}_{\star}(\theta)-\bm{X}_{\star}(\eta))\eqdef G_0(\Delta_\eta\bm{X}_\star)$, then \eqref{XtUint} simplifies to:
\begin{equation}\label{Uint}
\bm{\mathcal{L}}(\bm{X}_\star,\bm{X})(\theta)=\int_{\mathbb{S}} G_0(\Delta_\eta \bm{X}_\star)\paren{\bm{F}_L(\eta)+\frac{2A_e}{1\!-\!A_\mu}\mc{R}^{-1}\p_\eta \bm{X}}d\eta.
\end{equation}
This is our specification of the linearized operator.

We will now determine $\bm{F}_L$ as in \eqref{Fsplit}, that is, the linear part of $\bm{F}$ in \eqref{viscosityjump}. We find
\begin{equation}\label{Flintemp}
\begin{split}
&\bm{F}_L(\theta)+\frac{A_\mu}{\pi}\int_\mathbb{S}M(\theta, \eta)\bm{F}_L(\eta)d\eta=2A_e\p_\theta^2\bm{X} -2A_\mu(\bm{Q}+\bm{S}),\\
&Q_i=-\int_{\mathbb{S}} \mathcal{T}_{ijk}(\Delta \bm{X}_\star)F_{0,k}(\eta) \mc{R}^{-1}_{jl}\p_\theta X_l(\theta) d\eta,\\
&S_i=-\int_{\mathbb{S}} \PD{\mathcal{T}_{ijk}}{x_m}(\Delta \bm{X}_\star) \Delta \bm{X}_m F_{0,k}(\eta) \mc{R}^{-1}_{jl}\p_\theta X_{\star,l}(\theta)d\eta.
\end{split}
\end{equation}
Let us compute $\bm{Q}$.  We start with
\begin{equation*}
-\mathcal{T}_{ijk}(\Delta \bm{X}_\star)F_{0, k}=-\frac{2A_e}{1-A_\mu}\frac{\Delta X_{\star, i}\Delta X_{\star, j}\Delta X_{\star, k}e_{{\rm r},k}(\eta)}{\pi\abs{\Delta \bm{X}_\star}^4}
=\frac{2A_e}{1-A_\mu}\frac{\Delta X_{\star, i}\Delta X_{\star, j}}{2\pi\abs{\Delta \bm{X}_\star}^2},
\end{equation*}
where we used
\begin{equation}\label{DXdote2}
\frac{\Delta \bm{X}_\star\cdot \bm{e}_{\rm r}(\eta)}{\abs{\Delta \bm{X}_\star}^2}=-\frac{1}{2}.
\end{equation}
Therefore, we have
\begin{equation*}
\begin{split}
\bm{Q}&=\frac{A_e}{1-A_\mu}\frac{1}{\pi }\int_{\mathbb{S}}\frac{\Delta \bm{X}_{\star}\otimes\Delta \bm{X}_{\star}}{\abs{\Delta \bm{X}_\star}^2}d\eta\mc{R}^{-1}\p_\theta\bm{X}(\theta)
\\
&=\frac{A_e}{1-A_\mu}\frac{1}{\pi }\int_{\mathbb{S}}M(\theta,\eta)d\eta\mc{R}^{-1}\p_\theta\bm{X}(\theta)=\frac{A_e}{1-A_\mu}\mc{R}^{-1}\p_\theta\bm{X}(\theta),
\end{split}
\end{equation*}
where we used \eqref{matrixM} in the second equality. 
We next compute $\bm{S}$,
\begin{equation*}
\begin{split}
&-\PD{\mathcal{T}_{ijk}}{x_m}(\Delta X_\star)\Delta X_mF_{0,k}(\eta) \mc{R}^{-1}_{jl}\p_\theta X_{\star,l}(\theta)\\
=&-\frac{1}{\pi}\frac{2A_e}{1-A_\mu}\paren{\frac{\Delta X_i\Delta X_{\star, j}\Delta X_{\star, k}e_{{\rm r},j}(\theta)e_{{\rm r},k}(\eta)}{\abs{\Delta \bm{X}_\star}^4}}\\
&\frac{-2A_e/\pi}{1-A_\mu}\!\!\paren{\!\frac{\Delta X_{\star, i}\Delta X_j\Delta X_{\star, k}e_{{\rm r},j}(\theta)e_{{\rm r},k}(\eta)}{\abs{\Delta \bm{X}_\star}^4}
	\!+\!\frac{\Delta X_{\star, i}\Delta X_{\star, j}\Delta X_k e_{{\rm r},j}(\theta)e_{{\rm r},k}(\eta)}{\abs{\Delta \bm{X}_\star}^4}}\\
&+\frac{8A_e/\pi}{1-A_\mu}\paren{\frac{\Delta X_{\star, i}\Delta X_{\star, j}\Delta X_{\star, k}e_{{\rm r},j}(\theta)e_{{\rm r},k}(\eta)\Delta X_{\star, m}\Delta X_m}{\abs{\Delta \bm{X}_\star}^6}}.\\
\end{split}
\end{equation*}
We simplify each term as follows
\begin{equation*}
\begin{split}
&\frac{\Delta X_i\Delta X_{\star, j}\Delta X_{\star, k}e_{{\rm r},j}(\theta)e_{{\rm r},k}(\eta)}{\abs{\Delta \bm{X}_\star}^4}=-\frac{1}{4}\Delta X_i, \\
&\frac{\Delta X_{\star, i}\Delta X_j\Delta X_{\star, k}e_{{\rm r},j}(\theta)e_{{\rm r},k}(\eta)}{\abs{\Delta \bm{X}_\star}^4}
+\frac{\Delta X_{\star, i}\Delta X_{\star, j}\Delta X_k e_{{\rm r},j}(\theta)e_{{\rm r},k}(\eta)}{\abs{\Delta \bm{X}_\star}^4}\\
&=\frac{-\Delta X_{\star, i}\Delta X_je_{{\rm r},j}(\theta)}{2\abs{\Delta \bm{X}_\star}^2}
+\frac{\Delta X_{\star, i}\Delta X_k e_{{\rm r},k}(\eta)}{2\abs{\Delta \bm{X}_\star}^2}=-\frac{\Delta X_{\star, i}\Delta X_{\star, j} \Delta X_j}{2\abs{\Delta \bm{X}_\star}^2},\\
&\frac{\Delta X_{\star, i}\Delta X_{\star, j}\Delta X_{\star, k}e_{{\rm r},j}(\theta)e_{{\rm r},k}(\eta)\Delta X_{\star, m}\Delta X_m}{\abs{\Delta \bm{X}_\star}^6}
=-\frac{\Delta X_{\star, i}\Delta X_{\star, j} \Delta X_j}{4\abs{\Delta \bm{X}_\star}^2},
\end{split}
\end{equation*}
where above we made repeated use of \eqref{DXdote2} and
\begin{equation*}
\frac{\Delta \bm{X}_\star\cdot \mathcal{R}^{-1}\p_\theta \bm{X}_\star(\theta)}{\abs{\Delta \bm{X}_\star}^2}=
\frac{\Delta \bm{X}_\star\cdot e_{{\rm r}}(\theta)}{\abs{\Delta \bm{X}_\star}^2}=\frac{1}{2}.
\end{equation*}
Thus we have
\begin{equation*}
-\PD{\mathcal{T}_{ijk}}{x_m}(\Delta X_\star)\Delta X_mF_{0,k}(\eta) \mc{R}^{-1}_{jl}\p_\theta X_{\star,l}=\frac{A_e/2\pi}{1-A_\mu}\paren{\Delta X_i-\frac{2\Delta X_{\star, i}\Delta X_{\star, j} \Delta X_j}{\abs{\Delta \bm{X}_\star}^2}}.
\end{equation*}
Substituting this back into the expression for $\bm{S}$ in \eqref{Flintemp}, we have:
\begin{equation}\notag
\begin{split}
\bm{S}&=\frac{A_e/2\pi}{1-A_\mu}\int_{\mathbb{S}} (I-2M(\theta,\eta))\Delta \bm{X}d\eta=-\frac{A_e/2\pi}{1-A_\mu}\int_{\mathbb{S}} (I-2M(\theta,\eta))\bm{X}(\eta)d\eta\\
&=\frac{A_e/2\pi}{1-A_\mu} \paren{-\dual{\bm{e}_{\rm r}}{\bm{X}}\bm{e}_{\rm r}+\dual{\bm{e}_{\rm t}}{\bm{X}}\bm{e}_{\rm t}},
\end{split}
\end{equation}
where we used \eqref{matrixM} in the first equality and we are using the notation \eqref{innerProductNotation} for the inner product.

Equation \eqref{Flintemp} thus reduces to:
\begin{equation*}
\begin{split}
&\bm{F}_L(\theta)+\frac{A_\mu}{\pi}\int_{\mathbb{S}} M(\theta,\eta)\bm{F}_L(\eta)d\eta\\
=&2A_e\p_\theta^2\bm{X}(\theta) -\frac{2A_e A_\mu}{1-A_\mu}\paren{\mc{R}^{-1}\p_\theta \bm{X}(\theta)+\frac{1}{2\pi}\paren{-\dual{\bm{e}_{\rm r}}{\bm{X}}\bm{e}_{\rm r}(\theta)+\dual{\bm{e}_{\rm t}}{\bm{X}}\bm{e}_{\rm t}(\theta)}}.
\end{split}
\end{equation*}
We must solve the above equation for $\bm{F}_L$ in terms of $\bm{X}$. Suppose $\dual{\bm{e}_{\rm r}}{\bm{X}}=\dual{\bm{e}_{\rm t}}{\bm{X}}=0$.
Then, it is easily checked that:
\begin{equation}\label{F1}
\bm{F}_L(\theta)=2A_e\p_\theta^2\bm{X}(\theta) -\frac{2A_eA_\mu}{1-A_\mu}\mc{R}^{-1}\p_\theta \bm{X}(\theta).
\end{equation}
We may further compute $\bm{F}_L$ when $\bm{X}$ is either $\bm{e}_{\rm r}$ or $\bm{e}_{\rm t}$. Noting that
\begin{equation}\notag
M(\theta,\eta)\bm{e}_{\rm r}(\eta)=\frac{1}{2}(\bm{e}_{\rm r}(\eta)-\bm{e}_{\rm r}(\theta)), \; M(\theta,\eta)\bm{e}_{\rm t}(\eta)=\frac{1}{2}(\bm{e}_{\rm t}(\eta)+\bm{e}_{\rm t}(\theta)), 
\end{equation}
we see by an easy calculation that
\begin{equation}\notag
\text{If } \bm{X}=\bm{e}_{\rm r,t}, \text{ then } \bm{F}_L=-\frac{2A_e}{1-A_\mu}\bm{e}_{\rm r,t}.
\end{equation}
Note that
\begin{equation}\notag
2A_e\p_\theta^2\bm{e}_{\rm r,t} -\frac{2A_eA_\mu}{1-A_\mu}\mc{R}^{-1}\p_\theta \bm{e}_{\rm r,t}=-\frac{2A_e}{1-A_\mu}\bm{e}_{\rm r,t}.
\end{equation}
This shows that the expression for $\bm{F}_L$ in \eqref{F1} is in fact valid without the restriction $\dual{\bm{e}_{\rm r}}{\bm{X}}=\dual{\bm{e}_{\rm t}}{\bm{X}}=0$.
Substituting \eqref{F1} into \eqref{Uint} yields:
\begin{equation}\label{Ffinal}
\bm{\mathcal{L}}(\bm{X}_\star,\bm{X})(\theta)=2A_e\int_{\mathbb{S}} G_0(\Delta_\eta \bm{X}_\star)\paren{\p_{\eta}^2\bm{X}(\eta)+\mc{R}^{-1}\p_\eta \bm{X}(\eta)}d\eta.
\end{equation}
Finally, since
\begin{equation*}
G_0(\Delta_\eta \bm{X}_\star)=-\frac{1}{4\pi}\Big(\log{\left|2\sin{\left(\frac{\theta-\eta}{2}\right)}\right|}\Big)I+M(\theta,\eta),
\end{equation*}
and
\begin{equation*}
\begin{aligned}
&\int_\mathbb{S}M(\theta,\eta)\Big(\partial_\eta^2\bm{X}(\eta)+\mathcal{R}^{-1}\partial_\eta\bm{X}(\eta)\Big)d\eta\\
&=\int_\mathbb{S}\Big(\partial^2_\eta M(\theta,\eta)-\p_{\eta}M(\theta,\eta)\mathcal{R}^{-1}\Big)\bm{X}(\eta)d\eta=0,
\end{aligned}
\end{equation*}
we have that
\begin{equation*}
\begin{aligned}
\bm{\mathcal{L}}(\bm{X}_\star,\bm{X})(\theta)&=-\frac{A_e}{2\pi}\int_{\mathbb{S}} \log{\left|2\sin{\left(\frac{\theta-\eta}{2}\right)}\right|}\paren{\p_{\eta}^2\bm{X}(\eta)+\mc{R}^{-1}\p_\eta \bm{X}(\eta)}d\eta\\
&=-\frac{A_e}{2\pi}\int_{\mathbb{S}} 
\frac{\p_{\eta}\bm{X}(\eta)+\mc{R}^{-1} \bm{X}(\eta)}{2\tan{\big(\frac{\theta-\eta}{2}\big)}}d\eta,
\end{aligned}
\end{equation*}
which is given by a Hilbert transform
\begin{equation}\label{linearization}
\bm{\mathcal{L}}(\bm{X}_\star,\bm{X})(\theta)=-\frac{A_e}2\mathcal{H}\big(\p_{\eta}\bm{X}(\eta)+\mc{R}^{-1} \bm{X}(\eta)\big)(\theta).
\end{equation}
Therefore, the system \eqref{nonlinearexpansion} can be written as follows
\begin{equation}\label{finalsystem}
\bm{\mathcal{X}}_t(\theta)=-\frac{A_e}2\big(\Lambda \bm{X}(\theta)+\mathcal{H}\mc{R}^{-1} \bm{X}(\theta)\big)+\bm{\mathcal{N}}(\bm{X}_c,\bm{X})(\theta).
\end{equation}
Notice that $\bm{X}_c$ is a uniformly parametrized circle with time-dependent radius $R(t)$, as opposed to the $\bm{X_*}$ used in this subsection to obtain the linearization.
We will use the system \eqref{finalsystem} to study the global in time dynamics of the Peskin problem in the rest of this paper.

\subsection{Evolution of the $\foneonenu$ Norm of $\bm{X}$}\label{evolF11}

We first notice that, because $\bm{X}(\theta)$ is real-valued, it must hold that $\widehat{\bm{X}}(-k)=\overline{\widehat{\bm{X}}(k)}$. Therefore, the norm \eqref{fsneonenu} can be written in terms of positive frequencies alone
\begin{equation*}
\xoneonenu=2\sum_{k\geq1}e^{\nu(t)k}k|\widehat{\bm{X}}(k)|.
\end{equation*}
The  system \eqref{finalsystem}  in Fourier variables reads for $k\ge 0$ as follows
\begin{equation}\label{evolutionFourier}
\widehat{\bm{\mathcal{X}}}_t(k)=-\frac{A_e}2L(k)\widehat{\bm{X}}(k)+\mathcal{F}\big(\bm{\mathcal{N}}(\bm{X}_c,\bm{X})\big)(k).
\end{equation}
Here we recall that $\bm{\mathcal{X}}=\bm{X}+\bm{X}_c$.  Further the diffusion matrix is given by
\begin{equation*}
L(k)=\begin{bmatrix}
k&-i\hspace{0.05cm}{\sign}(k)\\
i\hspace{0.05cm}{\sign}(k)&k
\end{bmatrix},\quad k\geq 1, 
\quad
L(0)=\begin{bmatrix}
0&0\\
0&0
\end{bmatrix}.
\end{equation*}
The diagonalization of this matrix for $k\ge 1$  shows that
\begin{equation*}
L(k)=P(k)\mathcal{D}(k)P(k)^{-1},
\end{equation*}
where for $k\ge 1$ we have
\begin{equation*}
P(k)=\frac1{\sqrt{2}}\begin{bmatrix}
-i\hspace{0.05cm}{\sign}(k)&1\\1&-i\hspace{0.05cm}{\sign}(k)
\end{bmatrix},\quad P(k)^{-1}=\overline{P(k)},
\quad
\mathcal{D}(k)=\begin{bmatrix}
	k+1&0\\0&k-1
	\end{bmatrix}.
\end{equation*}
And when $k=0$ we define
\begin{equation*}
P(0)=\frac1{\sqrt{2}}\begin{bmatrix}
0&1\\1&0
\end{bmatrix},\quad P(0)^{-1}=2\overline{P(0)}, 
\quad
\mathcal{D}(0)=\begin{bmatrix}
	0&0\\0&0
	\end{bmatrix}.
\end{equation*}
This leads us to define the following change of variables
\begin{equation}\label{Y}
\widehat{\bm{Y}}(k)\eqdef P(k)^{-1}\widehat{\bm{X}}(k),\quad \widehat{\bm{Y}}_c(k) \eqdef P(k)^{-1}\widehat{\bm{X}}_c(k),
\end{equation}
with $\bm{\mathcal{Y}}\eqdef \bm{Y}+\bm{Y_c}$.  
The system \eqref{evolutionFourier} for $k\ge 0$ then becomes 
\begin{equation}\label{evolutionFourierY}
\widehat{\bm{\mathcal{Y}}}_t(k)=-\frac{A_e}2\mathcal{D}(k)\widehat{\bm{Y}}(k)+P(k)^{-1}\mathcal{F}\big(\bm{\mathcal{N}}(\bm{X}_c,\bm{X})\big)(k).
\end{equation}
The relationship between $\bm{X}$ and $\bm{Y}$ in space variables is given by the Hilbert transform \eqref{defHilbertTransform}, using also $\mathcal{H}^2(Y_j) = -Y_j$, as follows
\begin{equation}\label{transformation.physical}
\bm{X}(\theta)=\frac1{\sqrt{2}}\begin{bmatrix}
\mathcal{H}Y_1(\theta)+Y_2(\theta)\\
Y_1(\theta)+\mathcal{H}Y_2(\theta)
\end{bmatrix},
\quad
\bm{Y}(\theta)=\frac1{\sqrt{2}}\begin{bmatrix}
-\mathcal{H}X_1(\theta)+X_2(\theta)\\
X_1(\theta)-\mathcal{H}X_2(\theta)
\end{bmatrix}.
\end{equation}
Because, for $k\neq0$, $P(k)$ is a unitary matrix, it holds that $\|P(k)\|=\|P(k)^{-1}\|=1$, and therefore
\begin{equation*}
|\widehat{\bm{Y}}(k)|= |\widehat{\bm{X}}(k)|,\qquad k\neq0,
\end{equation*}
and thus 
\begin{equation}\label{equivalence}
    \|\bm{X}\|_{\foneonenu}=\|\bm{Y}\|_{\foneonenu}.
\end{equation}
We will use this norm equivalence several times in the following.

Notice that the first Fourier coefficient of a uniformly parametrized circle \eqref{unifcircles} is given by 
\begin{equation*}
\widehat{\bm{X}}_c(1)=\frac12\begin{bmatrix}
a+bi\\
-ia+b
\end{bmatrix},
\end{equation*}
and in  the $\bm{Y}$ variable 
\begin{equation}\label{circleY}
\widehat{\bm{Y}}_c(1)=\frac1{\sqrt{2}}\begin{bmatrix}
0\\
a+bi
\end{bmatrix}.
\end{equation}
Note that $\mH(\cos\theta) = \sin\theta$ and $\mH(\sin\theta) = -\cos\theta$.  
Then from the transformation \eqref{transformation.physical}  uniformly parametrized circles \eqref{vectors} in the $\bm{\mathcal{Y}}$ variable are spanned by
\begin{equation}\label{vectorsY}
\bm{\tilde{e}_r}(\theta)=\sqrt{2}\begin{bmatrix}
0\\
\cos{\theta}
\end{bmatrix},
\hspace{0.1cm}
\bm{\tilde{e}_t}(\theta)=\sqrt{2}
\begin{bmatrix}
0\\
-\sin{\theta}
\end{bmatrix},
\hspace{0.1cm}
\bm{\tilde{e}_1}=\frac{1}{\sqrt{2}}\begin{bmatrix}
0\\1
\end{bmatrix},
\hspace{0.1cm}
\bm{\tilde{e}_2}=\frac{1}{\sqrt{2}}\begin{bmatrix}
1\\0
\end{bmatrix}.
\end{equation}
Further the second component of $\widehat{\bm{Y}}(1)$ becomes zero after the operation of  $\mathcal{D}(1)$ is applied, which corresponds to the fact that uniformly parametrized circles are steady-states. 
Therefore, we will split the curve $\bm{\mathcal{Y}}(\theta)=\bm{Y_c}(\theta)+\bm{Y}(\theta)$, with 
$$\widehat{\bm{Y}}(0)=\begin{bmatrix}
0\\0
\end{bmatrix},\hspace{0.2cm} \widehat{Y}_2(1)=0,$$
since those frequencies are contained in the time-dependent circle \eqref{unifcircles}. In other words, $\bm{Y}$ is the projection of $\bm{\mathcal{Y}}$ onto the orthogonal complement of the vector space spanned by \eqref{vectorsY}.
In fact, the system of equations \eqref{evolutionFourierY} does not provide dissipation for the zero frequency of $\bm{\mathcal{Y}}$ nor for the second component of its first frequency (i.e., for uniformly parametrized circles). We thus can only expect decay for $\bm{Y}$. It is convenient then to write the equations of those frequencies in \eqref{evolutionFourierY} apart:
\begin{equation}\label{systemauxY}
\begin{aligned}
\widehat{\bm{\mathcal{Y}}}_t(0)=\partial_t\widehat{\bm{Y}}_c(0)&=   P(0)^{-1}\mathcal{F}\big(\bm{\mathcal{N}}(\bm{X}_c,\bm{X})\big)(0),
\\
\partial_t\widehat{\mathcal{Y}}_{1}(1)=\partial_t\widehat{Y}_{1}(1)&=-A_e\widehat{Y}_1(1)+\Big(P(1)^{-1}\mathcal{F}\big(\bm{\mathcal{N}}(\bm{X}_c,\bm{X})\big)(1)\Big)_1,
\\
\partial_t\widehat{\mathcal{Y}}_{2}(1)=\partial_t\widehat{Y}_{c,2}(1)&=\Big(P(1)^{-1}\mathcal{F}\big(\bm{\mathcal{N}}(\bm{X}_c,\bm{X})\big)(1)\Big)_2,
\\
\widehat{\bm{\mathcal{Y}}}_t(k)=\widehat{\bm{Y}}_t(k)&=-\frac{A_e}2\mathcal{D}(k)\widehat{\bm{Y}}(k)+P(k)^{-1}\mathcal{F}\big(\bm{\mathcal{N}}(\bm{X}_c,\bm{X})\big)(k), \hspace{0.2cm} k\geq2.
\end{aligned}
\end{equation}
Therefore, we study the evolution in time of $\|\bm{Y}\|_{\foneonenu}$, which is given by
\begin{equation*}
\begin{aligned}
\frac{d}{dt}\|\bm{Y}\|_{\foneonenu}&=\frac{d}{dt}\Big(2\sum_{k\geq1}e^{\nu(t)k}k\sqrt{\widehat{Y}_1(k)\overline{\widehat{Y}_1(k)}+\widehat{ Y}_2(k)\overline{\widehat{ Y}_2(k)}}\Big)\\
&\hspace{-2cm}=2\sum_{k\geq1}\nu'(t) k^2 e^{\nu(t)k}|\widehat{\bm{Y}}(k)|+2\sum_{k\geq1}e^{\nu(t)k}k\frac{\partial_t\widehat{\bm{Y}}(k)^T \overline{\widehat{\bm{Y}}(k)}+\widehat{\bm{Y}}(k)^T \overline{\partial_t\widehat{\bm{Y}}(k)}  }{2|\widehat{\bm{Y}}(k)|},
\end{aligned}
\end{equation*}
and introducing the time derivative \eqref{evolutionFourierY}, with $\bm{\mathcal{N}}=\bm{\mathcal{N}}(\bm{X}_c,\bm{X})=\bm{\mathcal{N}}(\bm{\mathcal{X}})$
we have
\begin{equation*}
\begin{aligned}
\frac{d}{dt}\|\bm{Y}\|_{\foneonenu}&=2\sum_{k\geq1}\nu'(t) k^2 e^{\nu(t)k}|\widehat{\bm{Y}}(k)|
\\
&\hspace{-1cm}\quad-2A_e\sum_{k\geq1}e^{\nu(t)k}k \frac{(k+1)|\widehat{Y}_1(k)|^2+(k-1)|\widehat{Y}_2(k)|^2            }{2|\widehat{\bm{Y}}(k)|}\\
&\hspace{-1cm}\quad+2\sum_{k\geq1}e^{\nu(t)k}k \frac{\big(P(k)^{-1}\widehat{\bm{\mathcal{N}}(\bm{\mathcal{X}})}(k)\big)^T \overline{\widehat{\bm{Y}}(k)} \! +\! \widehat{\bm{Y}}(k)^T\overline{\big(P(k)^{-1}\widehat{\bm{\mathcal{N}}(\bm{\mathcal{X}})}(k)\big)} }{2|\widehat{\bm{Y}}(k)|}.
\end{aligned}
\end{equation*}
Noticing that for $k\geq1$ we have
\begin{multline*}
    -A_e k\Big((k+1)|\widehat{Y}_1(k)|^2+(k-1)|\widehat{Y}_2(k)|^2\Big)\frac{1}{|\widehat{\bm{Y}}(k)|}\\\
    =
    -A_e k(k-1)|\widehat{\bm{Y}}(k)-2A_e k\frac{|\widehat{Y}_1(k)|^2}{|\widehat{\bm{Y}}(k)|},
\end{multline*}
 we can then see a diffusion term coming from the linear part:
\begin{multline}\label{diff}
\frac{d}{dt}\|\bm{Y}\|_{\foneonenu}
\leq 
-A_e\sum_{k\geq1}e^{\nu(t)k}k(k-1)|\widehat{\bm{Y}}(k)|-2A_e\sum_{k\geq1}e^{\nu(t)k}k\frac{|\widehat{Y}_1(k)|^2}{|\widehat{\bm{Y}}(k)|} 
\\
+2\sum_{k\geq1}\!\nu'(t) k^2 e^{\nu(t)k}|\widehat{\bm{Y}}(k)|\!+\!2\sum_{k\geq1}\!e^{\nu(t)k}k |\big(P(k)^{-1}\widehat{\bm{\mathcal{N}}(\bm{\mathcal{X}})}(k)\big)|.
\end{multline}
The balance above does not include the control of $\widehat{\bm{Y}}_{c}$.  We will  show in Subsection \ref{existence} that the evolution of $\widehat{\bm{Y}}_c(0)$, that is, of the center, can be controlled by all the other frequencies.  
Moreover, the incompressibility condition \eqref{incomp} allows us to control $\widehat{Y}_{c,2}(1)$ as follows
\begin{equation}\label{incompcon}
\begin{aligned}
V_0&=\pi=\frac12\int_{-\pi}^\pi \bm{\mathcal{X}}(\theta)\wedge \partial_\theta \bm{\mathcal{X}}(\theta)d\theta=\frac12\int_{-\pi}^\pi \Big(\mathcal{X}_1 \partial_\theta \mathcal{X}_2-\mathcal{X}_2\partial_\theta \mathcal{X}_1\Big)d\theta\\
&=\frac14\int_{-\pi}^\pi\Big(\big(\mathcal{H}\mathcal{Y}_1+\mathcal{Y}_2\big)\big(\partial_\theta\mathcal{Y}_1+\Lambda\mathcal{Y}_2\big)-\big(\mathcal{Y}_1+\mathcal{H}\mathcal{Y}_2\big)\big(\Lambda\mathcal{Y}_1+\partial_\theta\mathcal{Y}_2\big)d\theta\Big).
\end{aligned}
\end{equation}
Performing the products and taking into account the following equalities
\begin{equation*}
\begin{aligned}[alignment]
\int_{-\pi}^\pi\mathcal{H}\mathcal{Y}_i\Lambda \mathcal{Y}_j d\theta &=\int_{-\pi}^\pi \mathcal{Y}_i\partial_\theta\mathcal{Y}_j d\theta, \quad 
\int_{-\pi}^\pi\mathcal{H}\mathcal{Y}_j\p_\theta \mathcal{Y}_j d\theta &=
-\int_{-\pi}^\pi \mathcal{Y}_j\Lambda \mathcal{Y}_j d\theta,
\end{aligned}
\end{equation*}
we obtain that
\begin{equation*}
\begin{aligned}
\pi&=\frac12\int_{-\pi}^\pi \big(\mathcal{Y}_2\Lambda\mathcal{Y}_2-\mathcal{Y}_1\Lambda\mathcal{Y}_1\big)d\theta=\pi \big(\widehat{\mathcal{Y}_2\Lambda \mathcal{Y}_2}(0)-\widehat{\mathcal{Y}_1\Lambda \mathcal{Y}_1}(0)\big)\\
&=\pi\sum_{k\in\mathbb{Z}}\big( |k| \widehat{\mathcal{Y}}_2(k)\widehat{\mathcal{Y}}_2(-k)-|k|\widehat{\mathcal{Y}}_1(k)\widehat{\mathcal{Y}}_1(-k)\big)\\
&=\pi\!\sum_{k\in\mathbb{Z}}|k|\Big(|\widehat{Y}_{c,2}(k)|^2\!+\!\widehat{Y}_{c,2}(k)\widehat{Y}_2(-k)\!+\!\widehat{Y}_2(k)\widehat{Y}_{c,2}(-k)\!+\!|\widehat{Y}_2(k)|^2\!-\!|\widehat{Y}_1(k)|^2\Big),
\end{aligned}
\end{equation*}
where we have used that $\widehat{Y}_{c,1}(k)=0$ for $k\neq0$.
We can also eliminate the terms $\widehat{Y}_{c,2}(k)\widehat{Y}_{2}(-k)$ and $\widehat{Y}_{2}(k)\widehat{Y}_{c,2}(-k)$, since $\widehat{Y}_2(1)=0$ and $\widehat{Y}_{c,2}(k)=0$ for $k\neq0,\pm 1$. Therefore,
\begin{equation*}
\begin{aligned}
\frac12=\frac{a^2+b^2}{2}+\sum_{k\geq1}k\big(|\widehat{Y}_2(k)|^2-|\widehat{Y}_1(k)|^2\big).
\end{aligned}
\end{equation*}
And so the incompressibility condition translates to the constraint
\begin{equation}\label{radius}
|\widehat{Y}_{c,2}(1)|^2=\frac{a^2+b^2}2=\frac{R^2}2=\frac12-\sum_{k\geq1}k\big(|\widehat{Y}_2(k)|^2-|\widehat{Y}_1(k)|^2\big).
\end{equation}
Then, we can obtain an upper bound as follows
\begin{equation*}
\begin{aligned}
|\widehat{Y}_{c,2}(1)|^2&\leq\frac12+\sum_{k\geq1}k\big(|\widehat{Y}_2(k)|^2+|\widehat{Y}_1(k)|^2\big)= \frac12+\sum_{k\geq1}\Big(k^{1/2}|\widehat{\bm{Y}}(k)|\Big)^2\\
&\leq \frac12+\Big(\sum_{k\geq1}k^{1/2}|\widehat{\bm{Y}}(k)|\Big)^2= \frac12+\frac{1}{4}\|\bm{Y}\|_{\dot{\mathcal{F}}^{\frac12,1}}^2\\
&\leq \frac12+\frac{1}{4}\|\bm{Y}\|_{\dot{\mathcal{F}}^{\frac12,1}}^2,
\end{aligned}
\end{equation*}
and analogously we find the lower bound
\begin{equation}\label{radius.lower.bd}
\frac{R^2}2=|\widehat{Y}_{c,2}(1)|^2\geq\frac12-\frac{1}{4}\|\bm{Y}\|_{\dot{\mathcal{F}}^{\frac12,1}}^2.
\end{equation}
Recalling the relationship between $\bm{X}$ and $\bm{Y}$ in \eqref{equivalence}, we finally obtain that
\begin{equation*}
\frac12-\frac14\|\bm{X}\|_{\foneonenu}^2\leq |\widehat{Y}_{c,2}(1)|^2\leq \frac12+\frac14\|\bm{X}\|_{\foneonenu}^2,
\end{equation*}
and, since $|\widehat{\bm{Y}_c}(1)|^2=\frac{R^2}{2}=\frac{|\widehat{\bm{X}_c}(1)|^2}{2}$,
\begin{equation*}
1-\frac12\|\bm{X}\|_{\foneonenu}^2\leq |\widehat{\bm{X}}_{c}(1)|^2\leq 1+\frac12\|\bm{X}\|_{\foneonenu}^2,
\end{equation*}
so using the notation $R^2=a^2+b^2$, we have that
\begin{equation}\label{radiusbound}
\frac1{\sqrt{1+\frac12\|\bm{X}\|_{\foneonenu}^2}}\leq \frac1{R}\leq \frac{1}{\sqrt{1-\frac12\|\bm{X}\|_{\foneonenu}^2}}.
\end{equation}
The upper bound above motivates us to define
\begin{equation}\label{C1}
C_1\eqdef C_1(\|\bm{X}\|_{\foneonenu})=\frac{1}{\sqrt{1-\frac12\|\bm{X}\|_{\foneonenu}^2}}.
\end{equation}
We will later use \eqref{radiusbound} to control the size of $R$ when $\|\bm{X}\|_{\foneonenu}(t)\to 0$ as $t\to \infty$.

\subsection{Complete System.}\label{completesys}

We finally summarize the final form of the system of equations that describes our problem. The system given by \eqref{Xteqn} and \eqref{viscosityjump} for $\bm{\mathcal{X}}$  was replaced by the equations \eqref{systemauxY} on the Fourier coefficients of the associated variable $\bm{\mathcal{Y}}$ from \eqref{Y}. 
We recall that we decompose $\bm{\mathcal{Y}}$ into a time-dependent circle $\bm{Y_c}$ plus the deviation from the circle given by $\bm{Y}$.   In other words, we decompose $\bm{\mathcal{Y}}$ into its projection onto the vector space spanned by \eqref{vectorsY} represented by $\bm{Y_c}$ and its orthogonal complement represented by $\bm{Y}$. Therefore, recalling \eqref{circleY}, we have
\begin{equation}\label{eq1}
\widehat{\bm{Y}}(0)=0,\hspace{0.3cm}\widehat{Y_{2}}(1)=0,\hspace{0.3cm} \widehat{\bm{Y_c}}(k)=0\hspace{0.2cm} k\neq 0,1,\hspace{0.3cm}\widehat{Y_{c,1}}(1)=0.
\end{equation}
Now, for $k=1$ and $k\geq2$ separately, we have that
\begin{equation}\label{eqs}
\begin{aligned}
\widehat{\p_t Y_1}(1)&=-A_e \widehat{Y_1}(1)+\Big(P(1)^{-1}\widehat{\bm{\mathcal{N}}(\bm{X}_c,\bm{X})}(1)\Big)_1,\\
\widehat{\p_t \bm{Y}}(k)&=-\frac{A_e}2\mathcal{D}(k)\widehat{\bm{Y}}(k)+P(k)^{-1}\widehat{\bm{\mathcal{N}}(\bm{X}_c,\bm{X})}(k),
\end{aligned}
\end{equation}
where $\bm{X}_c$ and $\bm{X}$ are given in terms of $\bm{Y}_c$ and $\bm{Y}$ in  \eqref{Y}. In the following paragraphs, we will write one or the other without distinction for simplicity of notation.
The incompressibility condition \eqref{incompcon} yielded \eqref{radius}.  Thus in particular 
\begin{equation}\label{freq12}
    \sqrt{\frac12-\frac{1}{4}\|\bm{Y}\|_{\foneonenu}^2}\leq |\widehat{Y}_{c,2}(1)|\leq \sqrt{\frac12+\frac{1}{4}\|\bm{Y}\|_{\foneonenu}^2}.
\end{equation}
To close the system, notice that $\widehat{\bm{Y_c}}(0)=P(k)^{-1}\widehat{\bm{X}_c}(0)$ and, from \eqref{Xteqn}, we have
\begin{equation*}
\widehat{\p_t \bm{X}_c}(0)=\frac1{2\pi}\int_\mathbb{S}\int_\mathbb{S}G(\Delta \bm{X}_c+\Delta \bm{X})\bm{F}(\eta)d\eta d\theta,
\end{equation*}
with $\bm{F}$ defined by \eqref{viscosityjump}. We can also write the equation for $\widehat{\bm{X}_c}(0)$ using \eqref{nonlinearexpansion} or \eqref{evolutionFourier} and recalling that the zero frequency of the linear part vanishes,
\begin{equation}\label{zerofreq}
    \widehat{\p_t\bm{X}_c}(0)=\widehat{\bm{\mathcal{N}}(\bm{X}_c,\bm{X})}(0).
\end{equation}
We notice that the evolution of the zero frequency $\widehat{\bm{Y_c}}(0)$, corresponding to the center, is decoupled from all the other equations (in terms of the $\widehat{\bm{Y_c}}(0)$ variable), because $\widehat{\bm{X}_c}(0)$  does not appear on the right hand side of \eqref{Xteqn} and \eqref{viscosityjump} and therefore also \eqref{zerofreq}. 
This can be seen from the fact that in \eqref{Xteqn} $G$ only depends on the difference $\Delta \bm{\mathcal{X}}=\bm{\mathcal{X}}(\theta)-\bm{\mathcal{X}}(\eta)$ and in \eqref{viscosityjump} the expression for $\bm{\mathcal{S}}$ only depends on $\p_\theta\bm{\mathcal{X}}$ and $\Delta \bm{\mathcal{X}}$.
In summary, the system to determine $\bm{\mathcal{Y}}$ (equivalently determining $\bm{\mathcal{X}}$ via \eqref{Y}) consists of \eqref{eq1}, \eqref{eqs}, \eqref{radius}, and \eqref{zerofreq}. That is, all together we have that 
\begin{equation}\label{systemfinal}
\begin{aligned}
    \widehat{\bm{Y}}(0)&=0,\hspace{0.3cm}\widehat{Y}_{2}(1)=0,\hspace{0.3cm} \widehat{\bm{Y}}_{c}(k)=0\hspace{0.2cm} k\neq 0,1,\hspace{0.3cm}\widehat{Y}_{c,1}(1)=0,\\
    \partial_t\widehat{\bm{Y}}_{c}(0)&=P(0)^{-1}\widehat{\bm{\mathcal{N}}(\bm{X}_{c},\bm{X})}(0),\\
    \partial_t\widehat{Y}_{1}(1)&=-A_e \widehat{Y}_{1}(1)+\Big(P(1)^{-1}\widehat{\bm{\mathcal{N}}(\bm{X}_{c},\bm{X})}(1)\Big)_1,\\
    \partial_t\widehat{\bm{Y}}(k)&=-\frac{A_e}2\mathcal{D}(k)\widehat{\bm{Y}}(k)+P(k)^{-1}\widehat{\bm{\mathcal{N}}(\bm{X}_{c},\bm{X})}(k),\qquad  k\geq 2,\\
    |\widehat{Y}_{c,2}(1)|^2&=\frac12-\sum_{k\geq1}k\big(|\widehat{Y}_{2}(k)|^2-|\widehat{Y}_{1}(k)|^2\big),
    \end{aligned}
\end{equation} 
with $\bm{F}$ given in \eqref{viscosityjump}, and $\bm{Y}$, $\bm{X}$ related by \eqref{Y}.

To prove Theorem \ref{MainTheorem} (see Section \ref{secMain}) we will use system \eqref{systemfinal} to obtain the energy balance \eqref{diff} to show the decay of $\bm{Y}$. We will need to perform \textit{a priori} estimates on the nonlinear terms, which in particular requires us to prove bounds for $\bm{F}$ due to the viscosity contrast.
Those estimates are performed in the next section. The decay for $\bm{Y}$ will allow us to control the evolution of the zero frequency, that is, of the center.

\section{A priori estimates}\label{sec:secanalytic}

In this section we perform the \textit{a priori} estimates on $\bm{X}$ and $\bm{F}$ that will be used in the proof of our main Theorem \ref{MainTheorem}. First, in Subsection \ref{nonlinearestimates}, we estimate the nonlinear terms in \eqref{nonlinearexpansion} in terms of $\bm{X}$ and $\bm{F}$. Next, in Subsection \ref{secF}, we obtain the \textit{a priori} estimates for $\bm{F}$ in \eqref{viscosityjump} in terms of $\bm{X}$. In order to get the result with critical regularity, we have to get uniform bounds for some Fourier multipliers given by principal values (see Lemma \ref{lemmaI}).

\subsection{\textit{A Priori} Estimates on $\bm{X}$}\label{secanalytic}

\begin{prop}\label{nonlinearestimates}
Assume that $\bm{F}_0$, $\bm{F}_L$, $\bm{F}_N\in \mathcal{F}^{0,1}_\nu$ and $\bm{X}\in\dot{\mathcal{F}}^{2,1}_\nu$. Then, the nonlinear term  
$\bm{\mathcal{N}}=\bm{\mathcal{N}}(\bm{X}_c,\bm{X})(\theta)=\bm{\mathcal{N}}(\bm{\mathcal{X}})$ in \eqref{nonlinearexpansion} satisfies the following estimate in $\dot{\mathcal{F}}^{1,1}_\nu$,
\begin{equation}\label{Nestimate}
\begin{aligned}
\|\bm{\mathcal{N}}\|_{\foneonenu}&\leq 11\sqrt{2}D_1\|\bm{X}\|_{\foneonenu}\|\bm{F}_L\|_{\fzeronenu}\\
&\quad+ \frac{147}{2}D_2\|\bm{F}_0\|_{\fzeronenu}\|\bm{X}\|_{\foneonenu}\|\bm{X}\|_{\ftwoonenu}+\frac{9}{4}D_3\|\bm{F}_N\|_{\fzeronenu},
\end{aligned}
\end{equation}
where $D_i=D_i(\|\bm{X}\|_{\foneonenu},\nu_m)\approx 1$ are increasing functions of $\|\bm{X}\|_{\foneonenu}$ and $\nu_m$ such that 
$$\lim_{\|\bm{X}\|_{\foneonenu}\to0^+}D_i(\|\bm{X}\|_{\foneonenu},0)=1,$$
and are defined in \eqref{D1D2D3}.
\end{prop}

In the proof, the following multiplier will come up frequently:
\begin{equation}\label{mmult}
    \begin{aligned}
     m(k,\eta)\eqdef \frac{1-\frac{\sin{(k\eta/2)}}{k\tan{(\eta/2)}}e^{-i k\eta/2}}{2\sin{(\eta/2)}},\qquad |k|\geq1,
    \end{aligned}
\end{equation}
and we define $m(0,\eta)=0$.

Now let $n\geq1$, $k=k_0$, $k_1,\dots, k_{2n}$ be integers that further satisfy $|k_j-k_{j+1}|\ge 1$ for all $j=0,1,\ldots, 2n-1$.
We define the integral of type $I_n=I_n(k,k_1,\dots,k_{2n})$ by
    	\begin{equation}\label{In.Integral}
I_n \eqdef \textup{pv}\!\int_{-\pi}^\pi m(k-k_{1},\eta)
\prod_{j=1}^{2n-1}\frac{\sin{((k_j-k_{j+1})\eta/2)}}{(k_j-k_{j+1})\sin{(\eta/2)}}
e^{-i(k_1 + k_{2n})\eta/2}d\eta.
	\end{equation}	
We further define $I_n=0$ if $k_j=k_{j+1}$ for any $j=0,1,\ldots, 2n-1$.  We will also consider the integral, $I'_n = I'_n(k_1,\dots,k_{2n})$, under the same conditions
    	\begin{equation}\label{Sn.Integral}
I'_n \eqdef \textup{pv}\!\int_{-\pi}^\pi 
\frac{\sin{((k_1+k_{2n})\eta/2)}}{\sin{(\eta/2)}}
\prod_{j=1}^{2n-1}\frac{\sin{((k_j-k_{j+1})\eta/2)}}{(k_j-k_{j+1})\sin{(\eta/2)}}d\eta.
	\end{equation}	
We again define $I'_n=0$ if $k_j=k_{j+1}$ for any $j=1,\ldots, 2n-1$. 
	In the proofs of the a priori estimates in this section we will frequently use the following lemma.

\begin{lemma}\label{lemmaI}
We recall \eqref{mmult}, \eqref{In.Integral} and \eqref{Sn.Integral}.
    	Then, the following uniform bounds hold:
	\begin{equation*}
	\left| I_n(k,k_1,\dots,k_{2n})\right| \leq 2\pi,
	\end{equation*}
	and 
	\begin{equation*}
	\left| I'_n(k_1,\dots,k_{2n})\right| \leq 2\pi.
	\end{equation*}
\end{lemma}

This lemma will be proven at the end of this section.

\begin{proof}[Proof of Proposition \ref{nonlinearestimates}]
We first take a derivative of $\bm{\mathcal{N}}(\bm{X}_c,\bm{X})(\theta)$ in \eqref{nonlinearexpansion}	and let
\begin{equation}\label{Ni}
\begin{aligned}
\p_\theta\bm{\mathcal{N}}(\bm{X}_c,\bm{X})(\theta)=\bm{\mathcal{N}}_1(\theta)+\bm{\mathcal{N}}_2(\theta)+\bm{\mathcal{N}}_3(\theta)+\bm{\mathcal{N}}_4(\theta),
\end{aligned}
\end{equation}
where
\begin{equation*}
\begin{aligned}
\bm{\mathcal{N}}_1(\theta)&=\int_{\mathbb{S}} \p_\theta\big( G_L(\Delta_\eta\bm{X}_c(\theta),\Delta_\eta\bm{X}(\theta))\big)\bm{F}_L(\eta)d\eta,
\end{aligned}
\end{equation*}
\begin{equation*}
\begin{aligned}
\bm{\mathcal{N}}_2(\theta)&=\int_{\mathbb{S}}\p_\theta\big( G_N(\Delta_\eta\bm{X}_c(\theta),\Delta_\eta\bm{X}(\theta))\big)\bm{F}_0(\eta)d\eta,
\end{aligned}
\end{equation*}
\begin{equation*}
\begin{aligned}
\bm{\mathcal{N}}_3(\theta)&=\int_{\mathbb{S}}\p_\theta\big( G_N(\Delta_\eta\bm{X}_c(\theta),\Delta_\eta\bm{X}(\theta))\big)\bm{F}_L(\eta)d\eta,
\end{aligned}
\end{equation*}
\begin{equation*}
\begin{aligned}
\bm{\mathcal{N}}_4(\theta)&=\int_{\mathbb{S}}\p_\theta\big( G(\bm{\mathcal{X}}(\theta)-\bm{\mathcal{X}}(\eta))\big)\bm{F}_N(\eta)d\eta.
\end{aligned}
\end{equation*}
We will bound $\bm{\mathcal{N}}_i$ in $\fzeronenu$ for $i=1,2,3,4$.

\vspace{0.2cm}

\noindent\underline{$\bm{\mathcal{N}}_1$ estimates:} Taking a derivative in \eqref{G1}, we obtain that 
\begin{equation}\label{N1split}
\bm{\mathcal{N}}_1(\theta)=\sum_{i=0}^{10} \bm{\mathcal{N}}_{1,i}(\theta),
\end{equation}
and we proceed to bound each of these terms in $\fzeronenu$. We note that each term $\bm{\mathcal{N}}_{1,i}$ corresponds to when the derivative hits a different term inside \eqref{G1}.  The terms $\bm{\mathcal{N}}_{1,i}$ are written in \eqref{N11}, \eqref{N12term}, \eqref{N13},
\eqref{N14},
and
\eqref{N18etc}
in the following.

The first term $\bm{\mathcal{N}}_{1,1}(\theta)$ is given by
\begin{equation*}
\bm{\mathcal{N}}_{1,1}(\theta)=\frac{-1}{4\pi R^2}\int_\mathbb{S}\p_\theta\Delta_\eta \bm{X}_c(\theta)\cdot\Delta_\eta \bm{X}(\theta)\bm{F}_L(\eta)d\eta,
\end{equation*}
We first take the derivative of $\Delta_\eta \bm{X}_c(\theta)$ in \eqref{deltaeta} to obtain
\begin{equation*}
    \begin{aligned}
     \p_\theta\Delta_\eta \bm{X}_c(\theta)
     =
     \frac{\p_\theta \bm{X}_c(\theta)-\left(\bm{X}_c(\theta)-\bm{X}_c(\eta)\right)/\big(2\tan{\big(\frac{\theta-\eta}{2}\big)}\big)}{2\sin{\big(\frac{\theta-\eta}{2}\big)}}.
    \end{aligned}
\end{equation*}
Further define the operator $\derivdiff(\bm{X}_c)$ (and analogously $\derivdiff(\bm{X})$) to be $\p_\theta\Delta_\eta \bm{X}_c(\theta)$ as above
after taking the change of variables ${\eta\leftarrow \theta-\eta}$ as follows
\begin{equation}\label{deriv.diff.notation}
    \derivdiff(\bm{X}_c)(\theta,\eta)
    \eqdef
    \frac{\p_\theta \bm{X}_c(\theta)-\frac{\bm{X}_c(\theta)-\bm{X}_c(\theta-\eta)}{2\tan{(\eta/2)}}}{2\sin{(\eta/2)}}.
\end{equation}
Then we make the change of variables $\eta\leftarrow \theta-\eta$ to obtain
\begin{equation}\label{N11}
\bm{\mathcal{N}}_{1,1}(\theta)
=
\frac{-1}{4\pi R^2}\int_\mathbb{S}
\derivdiff(\bm{X}_c)(\theta,\eta)^T
\Delta_{\theta-\eta} \bm{X}(\theta)\bm{F}_L({\theta-\eta})d\eta,
\end{equation}
where we used transpose notation instead of a dot for future convenience in the notation.
We will also make extensive use of the following identities 
\begin{equation}\label{multiplier}
\begin{aligned}
\widehat{\Delta_{\theta-\eta} \bm{X}}(k)&=\frac{1-e^{-ik\eta}}{2\sin{(\eta/2)}}\widehat{\bm{X}}(k)
=
\frac{\sin{(k\eta/2)}}{k\sin{(\eta/2)}}e^{-ik\eta/2}\widehat{\p_\theta\bm{X}}(k),\\
\widehat{\Delta_{\theta-\eta} \bm{X}_c}(k)&=\frac{1-e^{-ik\eta}}{2\sin{(\eta/2)}}\widehat{\bm{X}_c}(k)=\frac{\sin{(k\eta/2)}}{k\sin{(\eta/2)}}e^{-ik\eta/2}\widehat{\p_\theta \bm{X}_c}(k).
\end{aligned}
\end{equation}
We remark that both terms above are $=0$ when $k=0$.   We further have
\begin{equation}\label{multiplierSecond}
        \widehat{\derivdiff(\bm{X}_c)}(k)=m(k,\eta)\widehat{\p_\theta\bm{X}}_c(k),
\end{equation}
where $m(k,\eta)$ is given by \eqref{mmult}.

Regarding the Fourier coefficients of the derivative of the circle \eqref{unifcircles} we have
\begin{equation}\label{circlefourier}
\widehat{\p_\theta\bm{X}_c}(k)=\frac{a+ib}{2}\delta_1(k)\begin{bmatrix}
i\\1
\end{bmatrix}-\frac{a-ib}{2}\delta_{-1}(k)\begin{bmatrix}
i\\-1
\end{bmatrix}.
\end{equation}
Taking Fourier transform in \eqref{N11}, we obtain that
\begin{equation*}
\begin{aligned}
\widehat{\bm{\mathcal{N}}_{1,1}}(k)&
=
\frac{-1}{4\pi R^2}\int_\mathbb{S}
\widehat{\derivdiff(\bm{X}_c)}(k)^T*\widehat{\Delta_{\theta-\eta} \bm{X}}
(k)*e^{-ik\eta}\widehat{\bm{F}_L}(k)d\eta
\\
=\frac{-1}{4\pi R^2}&\!\int_\mathbb{S}
\sum_{k_1\in\mathbb{Z}}\sum_{k_2\in\mathbb{Z}}\!
\widehat{\derivdiff(\bm{X}_c)}(k\!-\!k_1)^T\widehat{\Delta_{\theta-\eta} \bm{X}}
(k_1\!-\!k_2)e^{-i k_2\eta}\widehat{\bm{F}_L}(k_2)d\eta,
\end{aligned}
\end{equation*}
and plugging in \eqref{multiplier} and \eqref{multiplierSecond} we have that
\begin{equation*}
\begin{aligned}
\widehat{\bm{\mathcal{N}}_{1,1}}(k)&=\frac{-1}{4\pi R^2}\!\!\sum_{k_1\in\mathbb{Z}}\sum_{k_2\in\mathbb{Z}}\widehat{\p_\theta \bm{X}_c}(k\!-\!k_1)^T\widehat{\p_\theta\bm{X}}(k_1\!-\!k_2)\widehat{\bm{F}_L}(k_2)I_1(k,k_1,k_2),
\end{aligned}
\end{equation*}
with $I_1$ given by \eqref{In.Integral}.  
By Lemma \ref{lemmaI} we have $|I_1(k,k_1,k_2)|\leq 2\pi$.  Then we get that
\begin{equation}\label{N11f}
\begin{aligned}
|\widehat{\bm{\mathcal{N}}_{1,1}}(k)|&\leq\frac{1}{2 R^2}\sum_{k_1\in\mathbb{Z}}\sum_{k_2\in\mathbb{Z}}|\widehat{\p_\theta \bm{X}_c}(k-k_1)^T\widehat{\p_\theta\bm{X}}(k_1-k_2)||\widehat{\bm{F}_L}(k_2)|.
\end{aligned}
\end{equation}
Then, it follows from \eqref{circlefourier} that
\begin{equation}\label{pXcpX}
|\widehat{\p_\theta \bm{X}_c}(k-k_1)|\leq \frac{R}{\sqrt{2}}\delta_{1,-1}(k-k_1).
\end{equation}
We will now also use the notation \eqref{delta2notation}. 
In particular we have
\begin{equation}\notag
|\widehat{\p_\theta \bm{X}_c}(k-k_1)^T\widehat{\p_\theta\bm{X}}(k_1-k_2)|\leq \frac{\sqrt{2}}{2}R(\delta_1(k-k_1)+\delta_{-1}(k-k_1))|\widehat{\p_\theta\bm{X}}(k_1-k_2)|.
\end{equation}
 Therefore, we can write that
\begin{equation*}
\begin{aligned}
|\widehat{\bm{\mathcal{N}}_{1,1}}(k)|&\leq\frac{\sqrt{2}}{4R}\!\sum_{k_1\in\mathbb{Z}}\sum_{k_2\in\mathbb{Z}}\delta_{1,-1}(k-k_1)|\widehat{\p_\theta\bm{X}}(k_1\!-\!k_2)||\widehat{\bm{F}_L}(k_2)|.
\end{aligned}
\end{equation*}
We multiply by $e^{\nu(t)k}=e^{\nu(t) (k-k_1)}e^{\nu(t) (k_1-k_2)}e^{\nu(t) k_2}$ to get
\begin{multline*}
e^{\nu(t)k}|\widehat{\bm{\mathcal{N}}_{1,1}}(k)|\leq\\
\frac{\sqrt{2}}{4R}\!\sum_{k_1\in\mathbb{Z}}\sum_{k_2\in\mathbb{Z}}e^{\nu(t) (k-k_1)}\delta_{1,-1}(k-k_1)e^{\nu(t) (k_1-k_2)}|\widehat{\p_\theta\bm{X}}(k_1\!-\!k_2)|e^{\nu(t) k_2}|\widehat{\bm{F}_L}(k_2)|,
\end{multline*}
so Young's inequality for convolutions and the estimate \eqref{radiusbound} yield the bound 
\begin{equation}\label{N11bound}
\|\bm{\mathcal{N}}_{1,1}\|_{\fzeronenu}\leq\frac{e^{\nu_m }\sqrt{2}\|\bm{X}\|_{\foneonenu}}{2\sqrt{1-\frac12\|\bm{X}\|_{\foneonenu}^2}}\|\bm{F}_L\|_{\fzeronenu}.
\end{equation}
This is our desired estimate for $\bm{\mathcal{N}}_{1,1}$.

We now proceed to estimate $\bm{\mathcal{N}}_{1,2}$ as follows
\begin{equation}\label{N12term}
\bm{\mathcal{N}}_{1,2}(\theta)
=
\frac{-1}{4\pi R^2}\int_\mathbb{S}\Delta_{\theta-\eta} \bm{X}_c(\theta)^T
\derivdiff(\bm{X})(\theta,\eta)
\bm{F}_L({\theta-\eta})d\eta,
\end{equation}
whose Fourier transform is given by
\begin{equation*}
\begin{aligned}
\widehat{\bm{\mathcal{N}}_{1,2}}(k)\!=\!\frac{-1}{4\pi R^2}\!\!\int_\mathbb{S}\!\sum_{k_1\in\mathbb{Z}}\sum_{k_2\in\mathbb{Z}}\!\!\widehat{\Delta_{\theta-\eta} \bm{X}_c}
(k_1\!-\!k_2)^T\!
\widehat{\derivdiff(\bm{X})}
(k\!-\!k_1)
e^{\!-ik_2\eta}\widehat{\bm{F}_L}(k_2)d\eta.
\end{aligned}
\end{equation*}
 Using again \eqref{multiplier} and \eqref{multiplierSecond}, we can write it as follows
\begin{equation*}
\begin{aligned}
\widehat{\bm{\mathcal{N}}_{1,2}}(k)&=\frac{-1}{4\pi R^2}\sum_{k_1\in\mathbb{Z}}\sum_{k_2\in\mathbb{Z}}\widehat{\p_\theta \bm{X}_c}(k_1-k_2)^T
\widehat{\p_\theta\bm{X}}(k-k_1)
\widehat{\bm{F}_L}(k_2)I_1(k,k_1,k_2),
\end{aligned}
\end{equation*}
with $I_1$ given by \eqref{In.Integral}.  Using Lemma \ref{lemmaI}, we find that
\begin{equation*}
\begin{aligned}
|\widehat{\bm{\mathcal{N}}_{1,2}}(k)|&\leq\frac{1}{2 R^2}\sum_{k_1\in\mathbb{Z}}\sum_{k_2\in\mathbb{Z}}|\widehat{\p_\theta \bm{X}_c}(k-k_1)^T\widehat{\p_\theta\bm{X}}(k_1-k_2)||\widehat{\bm{F}_L}(k_2)|,
\end{aligned}
\end{equation*}
so following the steps after \eqref{N11f} we conclude that 
\begin{equation*}
\|\bm{\mathcal{N}}_{1,2}\|_{\fzeronenu}\leq\frac{e^{\nu_m }\sqrt{2}\|\bm{X}\|_{\foneonenu}}{2\sqrt{1-\frac12\|\bm{X}\|_{\foneonenu}^2}}\|\bm{F}_L\|_{\fzeronenu}.
\end{equation*}
This completes our bound for $\bm{\mathcal{N}}_{1,2}$.

The term $\bm{\mathcal{N}}_{1,3}$ is given by
\begin{multline}\label{N13}
\bm{\mathcal{N}}_{\!1,3}(\theta)
\\
=\frac{-1}{ R^4}\!\!\int_\mathbb{S}\!\!
\derivdiff(\bm{X}_c)(\theta,\eta)^T\Delta_{\theta-\eta} \bm{X}(\theta)\Delta_{\theta-\eta} \bm{X}_c(\theta)\otimes \Delta_{\theta-\eta} \bm{X}_c(\theta)\bm{F}_L({\theta\!-\!\eta})\frac{d\eta}{2\pi},
\end{multline}
and its Fourier transform by 
\begin{multline*}
\widehat{\bm{\mathcal{N}}_{1,3}}(k)=\frac{-1}{2\pi R^4}\!\sum_{k_1\in\mathbb{Z}}\dots\sum_{k_4\in\mathbb{Z}}\!\!\widehat{\p_\theta \bm{X}_c}(k-k_1)^T\widehat{\p_\theta\bm{X}}(k_1-k_2)\\
\widehat{\p_\theta \bm{X}_c}(k_2-k_3)\otimes \widehat{\p_\theta \bm{X}_c}(k_3-k_4)
\widehat{\bm{F}_L}(k_4)I_2(k,k_1,\dots,k_4),
\end{multline*}
with $I_2(k,k_1,\dots,k_4)$ given by \eqref{In.Integral}. Since $|I_2(k,\dots,k_4)|\leq 2\pi$ from Lemma \ref{lemmaI}, we have that
\begin{multline}\label{N13aux}
|\widehat{\bm{\mathcal{N}}_{1,3}}(k)|\leq \frac{1}{ R^4}\!\sum_{k_1\in\mathbb{Z}}\dots\sum_{k_4\in\mathbb{Z}}\!\!|\widehat{\p_\theta \bm{X}_c}(k-k_1)^T\widehat{\p_\theta\bm{X}}(k_1-k_2)|\\
\|\widehat{\p_\theta \bm{X}_c}(k_2-k_3)\otimes \widehat{\p_\theta \bm{X}_c}(k_3-k_4)\|
|\widehat{\bm{F}_L}(k_4)|.
\end{multline}
Expression \eqref{circlefourier} gives that
\begin{multline*}
\widehat{\p_\theta \bm{X}_c}(k_2-k_3)\otimes \widehat{\p_\theta \bm{X}_c}(k_3-k_4)=\frac{(a+ib)^2}{4}\delta_1(k_2-k_3)\delta_1(k_3-k_4)\begin{bmatrix}
-1&i\\
i&1
\end{bmatrix}\\+\frac{(a-ib)^2}{4}\delta_{-1}(k_2-k_3)\delta_{-1}(k_3-k_4)\begin{bmatrix}
-1&-i\\
-i&1
\end{bmatrix}\\-\frac{(a+ib)(a-ib)}{4}\delta_1(k_2-k_3)\delta_{-1}(k_3-k_4)\begin{bmatrix}
-1&-i\\
i&-1
\end{bmatrix}\\-\frac{(a+ib)(a-ib)}{4}\delta_{-1}(k_2-k_3)\delta_{1}(k_3-k_4)\begin{bmatrix}
-1&i\\
-i&-1
\end{bmatrix}.
\end{multline*}
All the matrices above have norm equal to $2$, so that
\begin{multline}\label{pXcopXc}
\|\widehat{\p_\theta \bm{X}_c}(k_2-k_3)\otimes \widehat{\p_\theta \bm{X}_c}(k_3-k_4)\|\leq \frac{R^2}{2}\delta_{1,-1}(k_2-k_3)\delta_{1,-1}(k_3-k_4).
\end{multline}
Introducing this bound, together with \eqref{pXcpX}, back to \eqref{N13aux}, we find that
\begin{multline*}
|\widehat{\bm{\mathcal{N}}_{1,3}}(k)| \leq \\\frac{\sqrt{2}}{ 4R}\!\sum_{k_1\in\mathbb{Z}}\!\!\!\dots\!\!\sum_{k_4\in\mathbb{Z}}
\delta_{1,-1}(k-k_1)|\widehat{\p_\theta\bm{X}}(k_1-k_2)|
\delta_{1,-1}(k_2-k_3)\delta_{1,-1}(k_3-k_4)
|\widehat{\bm{F}_L}(k_4)|,
\end{multline*}
thus multiplication by the exponential $e^{\nu(t)k}$, Young's inequality and \eqref{radiusbound} yield that
\begin{equation*}
\|\bm{\mathcal{N}}_{1,3}\|_{\fzeronenu}\leq\frac{2\sqrt{2}e^{3\nu_m }\|\bm{X}\|_{\foneonenu}}{\sqrt{1-\frac12\|\bm{X}\|_{\foneonenu}^2}}\|\bm{F}_L\|_{\fzeronenu}.
\end{equation*}
This completes our bound for $\bm{\mathcal{N}}_{1,3}$.

The term $\bm{\mathcal{N}}_{1,4}$ is given by
\begin{multline}\label{N14}
	\bm{\mathcal{N}}_{1,4}(\theta)
	\\
	\!=\!\frac{-1}{ R^4}\!\!\int_\mathbb{S}\!\!\Delta_{\theta-\eta} \bm{X}_c(\theta)^T\!
	\derivdiff(\bm{X})(\theta,\eta)
	\Delta_{\theta-\eta} \bm{X}_c(\theta)\otimes \Delta_{\theta-\eta} \bm{X}_c(\theta)\bm{F}_L({\theta\!-\!\eta})\frac{d\eta}{2\pi}.
\end{multline}
We take the Fourier transform and write the result as
\begin{multline*}
\widehat{\bm{\mathcal{N}}_{1,4}}(k)=
\frac{-1}{2\pi R^4}\!
\sum_{k_1\in\mathbb{Z}}\dots\sum_{k_4\in\mathbb{Z}}\!\!
\widehat{\p_\theta \bm{X}_c}(k_1-k_2)^T
\widehat{\p_\theta\bm{X}}(k-k_1)
\\
\widehat{\p_\theta \bm{X}_c}(k_2-k_3)\otimes \widehat{\p_\theta \bm{X}_c}(k_3-k_4)
\widehat{\bm{F}_L}(k_4)I_2(k,k_1,\dots,k_4),
\end{multline*}
with $I_2(k,k_1,\dots,k_4)$ given by \eqref{In.Integral}.  Since $|I_2|\leq 2\pi$ by Lemma \ref{lemmaI}, 
comparing now with \eqref{N13aux}, we conclude that 
\begin{equation*}
\|\bm{\mathcal{N}}_{1,4}\|_{\fzeronenu}\leq\frac{2\sqrt{2}e^{3\nu_m }\|\bm{X}\|_{\foneonenu}}{\sqrt{1-\frac12\|\bm{X}\|_{\foneonenu}^2}}\|\bm{F}_L\|_{\fzeronenu}.
\end{equation*}
This completes our estimate for $\bm{\mathcal{N}}_{1,4}$.

The remaining terms from $\mathcal{N}_1(\theta)$ in \eqref{N1split} are
\begin{equation*}
\bm{\mathcal{N}}_{\!1,5}(\theta)\!=\!\frac{-1}{ R^4}\!\!\int_\mathbb{S}\!\!\Delta_{\theta-\eta} \bm{X}_c(\theta)^T\!\Delta_{\theta-\eta} \bm{X}(\theta)
\derivdiff(\bm{X}_c)(\theta,\eta)
\otimes \Delta_{\theta-\eta} \bm{X}_c(\theta)\bm{F}_L({\theta\!-\!\eta})\frac{d\eta}{2\pi},
\end{equation*}
\begin{equation*}
\bm{\mathcal{N}}_{\!1,6}(\theta)\!=\!\frac{-1}{ R^4}\!\!\int_\mathbb{S}\!\!\Delta_{\theta-\eta} \bm{X}_c(\theta)^T\!\Delta_{\theta-\eta} \bm{X}(\theta)\Delta_{\theta-\eta} \bm{X}_c(\theta)\otimes 
\derivdiff(\bm{X}_c)(\theta,\eta)
\bm{F}_L({\theta\!-\!\eta})\frac{d\eta}{2\pi},
\end{equation*}
\begin{equation*}
\bm{\mathcal{N}}_{1,7}(\theta)\!=\!\frac{1}{4\pi R^2}\int_\mathbb{S}\!
\derivdiff(\bm{X}_c)(\theta,\eta)
\otimes\Delta_{\theta-\eta} \bm{X}(\theta)\bm{F}_L({\theta\!-\!\eta})d\eta,
\end{equation*}
\begin{equation}\label{N18etc}
\bm{\mathcal{N}}_{1,8}(\theta)=\frac{1}{4\pi R^2}\int_\mathbb{S}\!\Delta_{\theta-\eta}\bm{X}_c(\theta)\otimes 
\derivdiff(\bm{X})(\theta,\eta)
\bm{F}_L({\theta\!-\!\eta})d\eta,
\end{equation}
\begin{equation*}
\bm{\mathcal{N}}_{1,9}(\theta)=\frac{1}{4\pi R^2}\int_\mathbb{S}\!
\derivdiff(\bm{X})(\theta,\eta)
\otimes\Delta_{\theta-\eta} \bm{X}_c(\theta)\bm{F}_L({\theta\!-\!\eta})d\eta,
\end{equation*}
\begin{equation*}
\bm{\mathcal{N}}_{1,10}(\theta)=\frac{1}{4\pi R^2}\int_\mathbb{S}\!\Delta_{\theta-\eta}\bm{X}(\theta)\otimes
\derivdiff(\bm{X}_c)(\theta,\eta)
\bm{F}_L({\theta\!-\!\eta})d\eta.
\end{equation*}
It is not hard to see that $\bm{\mathcal{N}}_{1,5}$ and $\bm{\mathcal{N}}_{1,6}$ are bounded exactly as $\bm{\mathcal{N}}_{1,3}$ in \eqref{N13}, since the bound \eqref{pXcopXc} is also valid for $\derivdiff(\bm{X}_c)(\theta,\eta)\otimes \Delta_{\theta-\eta} \bm{X}_c(\theta)$ or $\Delta_{\theta-\eta} \bm{X}_c(\theta)\otimes \derivdiff(\bm{X}_c)(\theta,\eta)$.

We proceed then with $\bm{\mathcal{N}}_{1,7}$. Comparing with  $\bm{\mathcal{N}}_{1,1}$ in \eqref{N11}, \eqref{N11f}, we obtain that
\begin{equation*}
|\widehat{\bm{\mathcal{N}}_{1,7}}(k)|\leq\frac{1}{2 R^2}\sum_{k_1\in\mathbb{Z}}\sum_{k_2\in\mathbb{Z}}\|\widehat{\p_\theta \bm{X}_c}(k-k_1)\otimes\widehat{\p_\theta\bm{X}}(k_1-k_2)\||\widehat{\bm{F}_L}(k_2)|.
\end{equation*}
Using \eqref{circlefourier}, we find that
\begin{multline}\label{aux3}
\|\widehat{\p_\theta \bm{X}_c}(k-k_1)\otimes\widehat{\p_\theta\bm{X}}(k_1-k_2)\|
\\
\leq \frac{R}2\delta_{1}(k-k_1)\|\begin{bmatrix}
-1\\i
\end{bmatrix}\widehat{\p_\theta\bm{X}}(k_1-k_2)\|
\\
+\frac{R}2\delta_{-1}(k-k_1)\|\begin{bmatrix}
-1\\-i
\end{bmatrix}\widehat{\p_\theta\bm{X}}(k_1-k_2)\|\Big)
\\
\leq \frac{\sqrt{2}}{2}R\delta_{1,-1}(k-k_1)|\widehat{\p_\theta\bm{X}}(k_1-k_2)|,
\end{multline}
where in the last inequality we have used that the matrix norm \eqref{matrixnorm} is bounded by the Frobenius norm.
Therefore we conclude that
\begin{equation*}
\|\bm{\mathcal{N}}_{1,7}\|_{\fzeronenu}\leq\frac{\sqrt{2}e^{\nu_m }\|\bm{X}\|_{\foneonenu}}{2\sqrt{1-\frac12\|\bm{X}\|_{\foneonenu}^2}}\|\bm{F}_L\|_{\fzeronenu}.
\end{equation*}
The bound for $\bm{\mathcal{N}}_{1,8}$ follows in the same way as that of $\bm{\mathcal{N}}_{1,7}$, 
\begin{equation*}
\|\bm{\mathcal{N}}_{1,8}\|_{\fzeronenu}\leq\frac{\sqrt{2}e^{\nu_m }\|\bm{X}\|_{\foneonenu}}{2\sqrt{1-\frac12\|\bm{X}\|_{\foneonenu}^2}}\|\bm{F}_L\|_{\fzeronenu}.
\end{equation*}
Finally, the bounds for $\bm{\mathcal{N}}_{1,9}$ and $\bm{\mathcal{N}}_{1,10}$ are the same as for $\bm{\mathcal{N}}_{1,7}$ and $\bm{\mathcal{N}}_{1,8}$ because
\begin{multline}\label{aux4}
\|\widehat{\p_\theta \bm{X}}(k-k_1)\otimes\widehat{\p_\theta\bm{X}_c}(k_1-k_2)\|
\\
\leq \frac{R}2\delta_{1}(k_1-k_2)\|\widehat{\p_\theta\bm{X}}(k-k_1)\begin{bmatrix}
-1&i
\end{bmatrix}\|
\\
+\frac{R}2\delta_{-1}(k_1\!-\!k_2)\|\widehat{\p_\theta\bm{X}}(k\!-\!k_1)\begin{bmatrix}
-1\!&\!-i
\end{bmatrix}\|
\\
\leq \frac{\sqrt{2}}2R\delta_{1,-1}(k_1-k_2)|\widehat{\p_\theta\bm{X}}(k-k_1)|.
\end{multline}
Joining the bounds for $\bm{\mathcal{N}}_{1,1}$ to $\bm{\mathcal{N}}_{1,10}$, we obtain the bound for $\bm{\mathcal{N}}_{1}$ in \eqref{N1split} as
\begin{equation}\label{N1}
\|\bm{\mathcal{N}}_1\|_{\fzeronenu}\leq 11\sqrt{2}e^{3\nu_m }C_1 \|\bm{X}\|_{\foneonenu}\|\bm{F}_L\|_{\fzeronenu}, 
\end{equation}
where $C_1$ is defined in \eqref{C1}.  This completes our estimates for the $\bm{\mathcal{N}}_{1}$ term.
\\

\noindent\underline{$\bm{\mathcal{N}}_3$ estimates:} Taking a derivative in \eqref{GN}, we split $\bm{\mathcal{N}}_3$ as follows
\begin{equation}\label{N3split}
\bm{\mathcal{N}}_3(\theta)=\sum_{i=1}^{11}\bm{\mathcal{N}}_{3,i},
\end{equation}
where
\begin{equation*}
\bm{\mathcal{N}}_{3,1}(\theta)=\frac{-1}{4\pi}\int_\mathbb{S}\p_\theta\mathcal{R}_1(\Delta_\eta \bm{X}(\theta))\bm{F}_L(\eta)d\eta,
\end{equation*}
\begin{equation*}
\bm{\mathcal{N}}_{3,2}(\theta)\!=\!\frac{1}{4\pi R^2}\!\!\int_\mathbb{S}\!\!\p_\theta\Big(\!\Delta_\eta \bm{X}(\theta)\!\otimes\!\Delta_\eta \bm{X}(\theta)\!\Big)\Big(\!1\!-\!\frac2{R^2}\Delta_\eta \bm{X}_c(\theta)^T\! \Delta_\eta \bm{X}(\theta)\!\Big)\bm{F}_L(\eta)d\eta,
\end{equation*}
\begin{equation*}
\bm{\mathcal{N}}_{3,3}(\theta)=\frac{-1}{2\pi R^4}\int_\mathbb{S}\Delta_\eta \bm{X}(\theta)\otimes\Delta_\eta \bm{X}(\theta)\p_\theta\Big(\Delta_\eta \bm{X}_c(\theta)^T \Delta_\eta \bm{X}(\theta)\Big)\bm{F}_L(\eta)d\eta,
\end{equation*}
\begin{equation*}
\bm{\mathcal{N}}_{3,4}(\theta)=\frac{1}{4\pi R^2}\int_\mathbb{S}\p_\theta\Big(\Delta_\eta \bm{X}(\theta)\otimes\Delta_\eta \bm{X}(\theta)\Big)\mathcal{R}_2(\Delta_\eta \bm{X}(\theta))\bm{F}_L(\eta)d\eta,
\end{equation*}
\begin{equation*}
\bm{\mathcal{N}}_{3,5}(\theta)=\frac{1}{4\pi R^2}\int_\mathbb{S}\Delta_\eta \bm{X}(\theta)\otimes\Delta_\eta \bm{X}(\theta)\p_\theta\mathcal{R}_2(\Delta_\eta \bm{X}(\theta))\bm{F}_L(\eta)d\eta,
\end{equation*}
\begin{equation*}
\begin{aligned}
\bm{\mathcal{N}}_{3,6}(\theta)&=\frac{-1}{2\pi R^4}\int_\mathbb{S}\p_\theta\Big(\Delta_\eta \bm{X}_c(\theta)\otimes\Delta_\eta \bm{X}(\theta)\\
&\hspace{2.5cm}+\!\Delta_\eta \bm{X}(\theta)\!\otimes\!\Delta_\eta \bm{X}_c(\theta)\!\Big)\Delta_\eta \bm{X}_c(\theta)^T\! \Delta_\eta \bm{X}(\theta)\bm{F}_L(\eta)d\eta,
\end{aligned}
\end{equation*}
\begin{equation*}
\begin{aligned}
\bm{\mathcal{N}}_{3,7}(\theta)&=\frac{-1}{2\pi R^4}\int_\mathbb{S}\Big(\Delta_\eta \bm{X}_c(\theta)\otimes\Delta_\eta \bm{X}(\theta)\\
&\hspace{2.1cm}\!+\!\Delta_\eta \bm{X}(\theta)\!\otimes\!\Delta_\eta \bm{X}_c(\theta)\!\Big)\p_\theta\big(\Delta_\eta \bm{X}_c(\theta)^T\! \Delta_\eta \bm{X}(\theta)\big)\bm{F}_L(\eta)d\eta,
\end{aligned}
\end{equation*}
\begin{equation*}
\begin{aligned}
\bm{\mathcal{N}}_{3,8}(\theta)&=\frac{1}{4\pi R^2}\int_\mathbb{S}\p_\theta\Big(\Delta_\eta \bm{X}_c(\theta)\otimes\Delta_\eta \bm{X}(\theta)\\
&\hspace{2.5cm}+\Delta_\eta \bm{X}(\theta)\otimes\Delta_\eta \bm{X}_c(\theta)\Big)\mathcal{R}_2(\Delta_\eta \bm{X}(\theta)) \bm{F}_L(\eta)d\eta,
\end{aligned}
\end{equation*}
\begin{equation*}
\begin{aligned}
\bm{\mathcal{N}}_{3,9}(\theta)&=\frac{1}{4\pi R^2}\int_\mathbb{S}\Big(\Delta_\eta \bm{X}_c(\theta)\otimes\Delta_\eta \bm{X}(\theta)\\
&\hspace{2.5cm}+\Delta_\eta \bm{X}(\theta)\otimes\Delta_\eta \bm{X}_c(\theta)\Big)\p_\theta\mathcal{R}_2(\Delta_\eta \bm{X}(\theta))\bm{F}_L(\eta) d\eta,
\end{aligned}
\end{equation*}
\begin{equation*}
\begin{aligned}
\bm{\mathcal{N}}_{3,10}(\theta)&=\frac{1}{4\pi R^2}\int_\mathbb{S}\p_\theta\Big(\Delta_\eta \bm{X}_c(\theta)\otimes \Delta_\eta \bm{X}_c(\theta) \Big)\mathcal{R}_2(\Delta_\eta \bm{X}(\theta)) \bm{F}_L(\eta)d\eta,\\
\bm{\mathcal{N}}_{3,11}(\theta)&=\frac{1}{4\pi R^2}\int_\mathbb{S}\Delta_\eta \bm{X}_c(\theta)\otimes \Delta_\eta \bm{X}_c(\theta) \p_\theta\mathcal{R}_2(\Delta_\eta \bm{X}(\theta)) \bm{F}_L(\eta)d\eta,
\end{aligned}
\end{equation*}
where $\mathcal{R}_1$ and $\mathcal{R}_2$ were defined in \eqref{R1} and \eqref{R2}.

We proceed with $\bm{\mathcal{N}}_{3,1}$ first. We take the derivative in \eqref{R1} to obtain that
\begin{equation}\label{N31split}
\bm{\mathcal{N}}_{3,1}(\theta)=\bm{O}_1(\theta)+\bm{O}_2(\theta)+\bm{O}_3(\theta),
\end{equation}
where
\begin{multline*}
\bm{O}_1(\theta)=\frac{-1}{8\pi}\int_\mathbb{S}
\sum_{n\geq1}\!\sum_{\substack{m=0 \\ n+m\geq 2}}^{n-1}\!\!\!\!\begin{pmatrix}
n\\m
\end{pmatrix} \!\!\frac{(-1)^{n-1}(n\!-\!m)}{nR^{2n}}
(2\Delta_\eta \bm{X}_c(\theta)^T \Delta_\eta \bm{X}(\theta))^{n-m-1}\\2\p_\theta\Delta_\eta \bm{X}_c(\theta)^T \Delta_\eta \bm{X}(\theta)|\Delta_\eta \bm{X}(\theta)|^{2m}\bm{F}_L(\eta)d\eta,
\end{multline*}
\begin{multline*}
\bm{O}_2(\theta)=\frac{-1}{8\pi}\int_\mathbb{S}
\sum_{n\geq1}\!\sum_{\substack{m=0 \\ n+m\geq 2}}^{n-1}\!\!\!\!\begin{pmatrix}
n\\m
\end{pmatrix} \!\!\frac{(-1)^{n-1}(n\!-\!m)}{nR^{2n}}
(2\Delta_\eta \bm{X}_c(\theta)^T \Delta_\eta \bm{X}(\theta))^{n-m-1}\\
2\Delta_\eta \bm{X}_c(\theta)^T \p_\theta\Delta_\eta \bm{X}(\theta)|\Delta_\eta \bm{X}(\theta)|^{2m}\bm{F}_L(\eta)d\eta,
\end{multline*}
\begin{multline*}
\bm{O}_3(\theta)=\frac{-1}{4\pi}\int_\mathbb{S}
\sum_{n\geq1}\!\sum_{m=1}^n\begin{pmatrix}
n\\m
\end{pmatrix} \!\!\frac{(-1)^{n-1}m}{nR^{2n}}
(2\Delta_\eta \bm{X}_c(\theta)^T \Delta_\eta \bm{X}(\theta))^{n-m}\\
|\Delta_\eta \bm{X}(\theta)|^{2(m-1)}\Delta_\eta \bm{X}(\theta)^T\p_\theta \Delta_\eta \bm{X}(\theta)\bm{F}_L(\eta)d\eta.
\end{multline*}
After performing the change of variables $\eta\leftarrow \theta-\eta$, we take Fourier transform of $\bm{O}_1(\theta)$ to obtain
\begin{multline*}
\widehat{\bm{O}_1}(k)=\frac{-1}{8\pi}\int_\mathbb{S}
\sum_{n\geq1}\!\sum_{\substack{m=0 \\ n+m\geq 2}}^{n-1}\!\!\!\begin{pmatrix}
n\\m
\end{pmatrix} \!\!\frac{(-1)^{n-1}(n\!-\!m)}{nR^{2n}}
*^{n-m-1}\widehat{2\Delta_{\theta-\eta} \bm{X}_c(\theta)^T\Delta_{\theta-\eta} \bm{X}(\theta)}\\
*2\widehat{\derivdiff(\bm{X}_c)(\theta)^T \Delta_{\theta-\eta} \bm{X}}(k)*^m\widehat{ \Delta_{\theta-\eta} \bm{X}(\theta)^T \Delta_{\theta-\eta} \bm{X}(\theta)}*\widehat{\bm{F}_L}(k)d\eta.
\end{multline*}
Using \eqref{multiplier} and \eqref{multiplierSecond}, we rewrite it as follows
\begin{multline}\label{O1aux}
\widehat{\bm{O}_1}(k)
\\
\!=\!\frac{-1}{8\pi}
\sum_{n\geq1}\!\!\sum_{\substack{m=0 \\ n+m\geq 2}}^{n-1}\!\!\!\!\begin{pmatrix}
n\\m
\end{pmatrix} \!\!\frac{(-1)^{n-1}(n\!-\!m)}{nR^{2n}}\sum_{k_1}\!\dots\!\!\sum_{k_{k_{2n}}}\!\!\!\prod_{j=0}^{n-m-2}
\!\!\!2\widehat{\p_\theta\bm{X}_c}(k_{2j+1}-k_{2j+2})^T\\
\widehat{\p_\theta\bm{X}}(k_{2j+2} - k_{2j+3})2
\widehat{\p_\theta\bm{X}_c}(k - k_1)^T
\widehat{\p_\theta\bm{X}}(k_{2n-2m-1} - k_{2n-2m})\\
\prod_{j=n-m}^{n-1}\widehat{\p_\theta\bm{X}}(k_{2j}-k_{2j+1})^T\widehat{\p_\theta\bm{X}}(k_{2j+1}-k_{2j+2})\widehat{\bm{F}_L}(k_{2n})I_n(k,k_1,\dots,k_{2n}),
\end{multline}
with $\left| I_n(k,k_1,\dots,k_{2n})\right| \le 2\pi$ given by \eqref{In.Integral} and using Lemma \ref{lemmaI}.  Above we are using the convention that $\prod_{j=j_1}^{j_2}f(j)\equiv 1$ if $j_2<j_1$.
Recalling estimate \eqref{pXcpX}, distributing the exponential factor $e^{\nu(t) k}$, and applying Young's inequality, we have that 
\begin{multline*}
\|\bm{O}_1\|_{\fzeronenu}
\\
\leq\!\frac{1}{4}\Bigg(\!
\sum_{n\geq1}\!\!\!\sum_{\substack{m=0 \\ n+m\geq 2}}^{n-1}\!\!\!\!\!\begin{pmatrix}
n\\m
\end{pmatrix} \!\!\frac{(n\!-\!m)}{n R^{2n}}(2\sqrt{2})^{n-m}e^{(n-m)\nuTIME }R^{n-m}\|\bm{X}\|_{\foneonenu}^{n-m}\|\bm{X}\|_{\foneonenu}^{2m}\!\!\Bigg)\|\bm{F}_L\|_{\fzeronenu},
\end{multline*}
which can be summed first in $m$ to get
\begin{equation}\label{O1n}
\|\bm{O}_1\|_{\fzeronenu}\leq\frac{1}{4}
\sum_{n\geq2}(2\sqrt{2})^n e^{n\nuTIME }\frac{\|\bm{X}\|_{\foneonenu}^n}{R^n}\Big(1+\frac{\|\bm{X}\|_{\foneonenu}}{2\sqrt{2}e^{\nuTIME }R}\Big)^{n-1} \|\bm{F}_L\|_{\fzeronenu},
\end{equation}
and then sum in in $n$,
\begin{equation*}
\|\bm{O}_1\|_{\fzeronenu}\leq
\frac{2e^{2\nu_m }\Big(1+\frac{\|\bm{X}\|_{\foneonenu}}{2\sqrt{2}e^{\nu_m }R}\Big)}{1-2\sqrt{2}e^{\nu_m }\frac{\|\bm{X}\|_{\foneonenu}}{R}\Big(1+\frac{\|\bm{X}\|_{\foneonenu}}{2\sqrt{2}e^{\nu_m }R}\Big)} \frac{\|\bm{X}\|_{\foneonenu}^2}{ R^2}\|\bm{F}_L\|_{\fzeronenu}.
\end{equation*}
Using estimate \eqref{radiusbound} and the notation \eqref{C1}, we conclude that
\begin{equation}\label{O1bound}
\|\bm{O}_1\|_{\fzeronenu}\leq 2 e^{2\nu_m }C_2C_1^2
\|\bm{X}\|_{\foneonenu}^2\|\bm{F}_L\|_{\fzeronenu},
\end{equation}
with
\begin{equation}\label{C2}
C_2=\frac{1+\frac1{2\sqrt{2}}e^{-\nu_m }C_1\|\bm{X}\|_{\foneonenu}}{1-2\sqrt{2}e^{\nu_m }C_1\|\bm{X}\|_{\foneonenu}\Big(1+\frac1{2\sqrt{2}}e^{-\nu_m }C_1\|\bm{X}\|_{\foneonenu}\Big)},
\end{equation}
where $C_1$ was defined in \eqref{C1}.

We proceed with $\bm{O}_2$ in \eqref{N31split}. We take Fourier transform and, recalling \eqref{multiplier}, we obtain that
\begin{multline}\label{O2aux}
\widehat{\bm{O}_2}(k)
\\
\!=\!
\frac{-1}{8\pi}
\sum_{n\geq1}\!\!\sum_{\substack{m=0 \\ n+m\geq 2}}^{n-1}\!\!\!\!\!\begin{pmatrix}
n\\m
\end{pmatrix} \!\!\frac{(-1)^{n-1}(n\!-\!m)}{n R^{2n}}\sum_{k_1}\!\dots\!\!\sum_{k_{k_{2n}}}\!\!\!\prod_{j=0}^{n-m-2}\!\!\!
2\widehat{\p_\theta\bm{X}_c}(k_{2j+1}-k_{2j+2})^T\\
\widehat{\p_\theta\bm{X}}(k_{2j+2}-k_{2j+3})
2\widehat{\p_\theta\bm{X}_c}(k_{2n-2m-1}-k_{2n-2m})^T
\widehat{\p_\theta\bm{X}}(k-k_1)
\\
\prod_{j=n-m}^{n-1}\widehat{\p_\theta\bm{X}}(k_{2j}-k_{2j+1})^T\widehat{\p_\theta\bm{X}}(k_{2j+1}-k_{2j+2})\widehat{\bm{F}_L}(k_{2n})I_n(k,k_1,\dots,k_{2n}),
\end{multline}
again with $\left| I_n(k,k_1,\dots,k_{2n})\right| \le 2\pi$ from \eqref{In.Integral} and Lemma \ref{lemmaI}. 
Thus, comparing \eqref{O2aux} with \eqref{O1aux}, we find the estimate for $\bm{O}_2$,
\begin{equation}\label{O2bound}
\|\bm{O}_2\|_{\fzeronenu}\leq 2e^{2\nu_m }C_2C_1^2
\|\bm{X}\|_{\foneonenu}^2\|\bm{F}_L\|_{\fzeronenu},
\end{equation}
with $C_2$ defined in \eqref{C2} and $C_1$ in \eqref{C1}.

Repeating these steps for $\bm{O}_3$, we obtain that
\begin{multline*}
\!\|\bm{O}_3\|_{\fzeronenu}
\\
\!\leq\!\frac{1}{2}
\sum_{n\geq1}\!\sum_{m=1}^{n}\!\!\begin{pmatrix}
n\\m
\end{pmatrix} \!\!\frac{m(2\sqrt{2})^{n-m}}{n R^{2n}}e^{\nuTIME (n-m)}R^{n-m}\|\bm{X}\|_{\foneonenu}^{n-m}\|\bm{X}\|_{\foneonenu}^{2(m-1)}\|\bm{X}\|_{\foneonenu}^{2}\|\bm{F}_L\|_{\fzeronenu},
\end{multline*}
which after summation in $m$ the right side above becomes
\begin{multline}\label{O3n}
\|\bm{O}_3\|_{\fzeronenu}
\\
\leq\frac12
\sum_{n\geq1}\frac{1}{n R^{2n}}n(2\sqrt{2}e^{\nuTIME }R\|\bm{X}\|_{\foneonenu}+\|\bm{X}\|_{\foneonenu}^2)^{n-1}\|\bm{X}\|_{\foneonenu}^{2}\|\bm{F}_L\|_{\fzeronenu}
\\
=\frac12
\sum_{n\geq1}\frac{\|\bm{X}\|_{\foneonenu}^{n-1}}{R^{n-1}}\Big(2\sqrt{2}e^{\nuTIME }+\frac{\|\bm{X}\|_{\foneonenu}}{R}\Big)^{n-1}\frac{\|\bm{X}\|_{\foneonenu}^{2}}{R^2}\|\bm{F}_L\|_{\fzeronenu},
\end{multline}
and after summation in $n$ we have
\begin{equation*}
\begin{aligned}
\|\bm{O}_3\|_{\fzeronenu}\leq\frac12\frac{1}{1-2\sqrt{2}e^{\nu_m }\frac{\|\bm{X}\|_{\foneonenu}}{R}\Big(1+e^{-\nu_m }\frac{\|\bm{X}\|_{\foneonenu}}{2\sqrt{2}R}\Big)}\frac{\|\bm{X}\|_{\foneonenu}^{2}}{R^2}\|\bm{F}_L\|_{\fzeronenu}.
\end{aligned}
\end{equation*}
Introducing the bound for $R$ in \eqref{radiusbound}, we obtain 
\begin{equation}\label{O3bound}
\|\bm{O}_3\|_{\fzeronenu}\leq \frac12C_3C_1^2
\|\bm{X}\|_{\foneonenu}^2\|\bm{F}_L\|_{\fzeronenu},
\end{equation}
and using $C_2$ in \eqref{C2} and $C_1$ in \eqref{C1} we have
\begin{equation}\label{C3}
C_3=\frac{C_2}{1+\frac1{2\sqrt{2}}e^{-\nu_m }C_1\|\bm{X}\|_{\foneonenu}}.
\end{equation}
Joining the bounds \eqref{O1bound}, \eqref{O2bound}, and \eqref{O3bound}, we find the estimate for $\bm{\mathcal{N}}_{3,1}$ from \eqref{N31split} as
\begin{equation}\label{N31bound}
\|\bm{\mathcal{N}}_{3,1}\|_{\fzeronenu}\leq \frac92C_4C_1^2
\|\bm{X}\|_{\foneonenu}^2\|\bm{F}_L\|_{\fzeronenu},
\end{equation}
with
\begin{equation}\label{C4}
C_4=\frac29\Big(4e^{2\nu_m}C_2+\frac{C_3}{2}\Big).
\end{equation}
This completes our desired estimate for $\bm{\mathcal{N}}_{3,1}$.

We continue with the next term $\bm{\mathcal{N}}_{3,2}$ from  \eqref{N3split}, which we split in two
\begin{equation*}
\bm{\mathcal{N}}_{3,2}(\theta)=\bm{O}_4(\theta)+\bm{O}_5(\theta),
\end{equation*}
\begin{equation*}
\bm{O}_{4}(\theta)=\frac{1}{4\pi R^2}\int_\mathbb{S}\p_\theta\Big(\Delta_\eta \bm{X}(\theta)\otimes\Delta_\eta \bm{X}(\theta)\Big)\bm{F}_L(\eta)d\eta,
\end{equation*}
\begin{equation*}
\bm{O}_{5}(\theta)=\frac{-1}{2\pi R^4}\int_\mathbb{S}\p_\theta\Big(\Delta_\eta \bm{X}(\theta)\otimes\Delta_\eta \bm{X}(\theta)\Big)\Delta_\eta \bm{X}_c(\theta)^T \Delta_\eta \bm{X}(\theta)\bm{F}_L(\eta)d\eta.
\end{equation*}
The bounds for these terms follows in a similar way to that of $\bm{\mathcal{N}}_{1,2}$ from \eqref{N12term} and $\bm{\mathcal{N}}_{1,4}$ from \eqref{N14}, respectively. Taking into account that
\begin{equation}\label{XotimesX}
\|\widehat{\p_\theta\bm{X}}(k-k_1)\otimes\widehat{\p_\theta\bm{X}}(k_1-k_2)\|\leq |\widehat{\p_\theta\bm{X}}(k-k_1)||\widehat{\p_\theta\bm{X}}(k_1-k_2)|,
\end{equation}
and Lemma \ref{lemmaI}, it is not hard to find that
\begin{equation*}
\begin{aligned}
|\widehat{\bm{O}_{4}}(k)|\leq\frac{1}{R^2}\sum_{k_1\in\mathbb{Z}}\sum_{k_2\in\mathbb{Z}}|\widehat{\p_\theta \bm{X}}(k-k_1)| |\widehat{\p_\theta\bm{X}}(k_1-k_2)||\widehat{\bm{F}_L}(k_2)|,
\end{aligned}
\end{equation*}
and recalling \eqref{pXcpX}, we have
\begin{multline*}
|\widehat{\bm{O}_{5}}(k)|\leq\frac{\sqrt{2}}{R^3}\sum_{k_1\in\mathbb{Z}}\dots\sum_{k_4\in\mathbb{Z}}|\widehat{\p_\theta \bm{X}}(k-k_1)||\widehat{\p_\theta\bm{X}}(k_1-k_2)|\delta_{1,-1}(k_2-k_3)\\
\times |\widehat{\p_\theta\bm{X}}(k_3-k_4)||\widehat{\bm{F}_L}(k_4)|.
\end{multline*}
Therefore, 
\begin{equation*}
\|\bm{O}_{4}\|_{\fzeronenu}\leq \frac{\|\bm{X}\|_{\foneonenu}^2}{R^2}\|\bm{F}_L\|_{\fzeronenu}, \hspace{0.3cm} \|\bm{O}_{5}\|_{\foneonenu}\leq 2\sqrt{2}e^{\nu_m }\frac{\|\bm{X}\|_{\foneonenu}^3}{R^3}\|\bm{F}_L\|_{\fzeronenu},
\end{equation*}
thus
\begin{equation*}
\|\bm{\mathcal{N}}_{3,2}\|_{\fzeronenu}\leq \Big(1+2\sqrt{2}e^{\nu_m }\frac{\|\bm{X}\|_{\foneonenu}}{R}\Big)\frac{\|\bm{X}\|_{\foneonenu}^2}{R^2}\|\bm{F}_L\|_{\fzeronenu},
\end{equation*}
so plugging in the estimate \eqref{radiusbound} yields that
\begin{equation}\label{N32bound}
\|\bm{\mathcal{N}}_{3,2}\|_{\fzeronenu}\leq C_5C_1^2\|\bm{X}\|_{\foneonenu}^2\|\bm{F}_L\|_{\fzeronenu},
\end{equation}
with
\begin{equation}\label{C5}
C_5=1+2\sqrt{2}e^{\nu_m }C_1\|\bm{X}\|_{\foneonenu}.
\end{equation}
This completes our estimate for $\bm{\mathcal{N}}_{3,2}$.

The Fourier transform of $\bm{\mathcal{N}}_{3,3}$ in \eqref{N3split}  can be bounded as follows
\begin{multline*}
|\widehat{\bm{\mathcal{N}}_{3,3}}(k)|\leq\frac{\sqrt{2}}{2 R^3}\sum_{k_1\in\mathbb{Z}}\!\dots\!\sum_{k_4\in\mathbb{Z}}|\widehat{\p_\theta \bm{X}}(k\!-\!k_1)||\widehat{\p_\theta\bm{X}}(k_1\!-\!k_2)|\delta_{1,-1}(k_2\!-\!k_3)\\
\times |\widehat{\p_\theta\bm{X}}(k_3-k_4)||\widehat{\bm{F}_L}(k_4)|,
\end{multline*}
and thus
\begin{equation*}
\|\bm{\mathcal{N}}_{3,3}\|_{\fzeronenu}\leq \sqrt{2}e^{\nu_m } \frac{\|\bm{X}\|_{\foneonenu}^3}{R^3} \|\bm{F}_L\|_{\fzeronenu},
\end{equation*}
which becomes
\begin{equation}\label{N33bound}
\|\bm{\mathcal{N}}_{3,3}\|_{\fzeronenu}\leq \sqrt{2}e^{\nu_m } C_1^3\|\bm{X}\|_{\foneonenu}^3\|\bm{F}_L\|_{\fzeronenu}.
\end{equation}
Similarly, recalling \eqref{R2}, the estimate for $\bm{\mathcal{N}}_{3,4}$ in \eqref{N3split} is
\begin{multline*}
\|\bm{\mathcal{N}}_{3,4}\|_{\fzeronenu}\!\leq\!
\\
\frac{\|\bm{X}\|_{\foneonenu}^2}{ R^2}\sum_{n\geq1}\!\!\sum_{\substack{m=0\\n+m\geq2}}^{n}\!\!\!\!\begin{pmatrix}
n\\m
\end{pmatrix}\frac{(2\sqrt{2})^{n-m}e^{(n-m)\nuTIME }R^{n-m}}{R^{2n}}\|\bm{X}\|_{\foneonenu}^{n-m}\|\bm{X}\|_{\foneonenu}^{2m}\|\bm{F}_L\|_{\fzeronenu},
\end{multline*}
which rewrites as
\begin{multline*}
\|\bm{\mathcal{N}}_{3,4}\|_{\fzeronenu}
\\
\!\leq\! \frac{\|\bm{X}\|_{\foneonenu}^2}{ R^2}\bigg(\sum_{n\geq1}\!\frac{(2\sqrt{2}e^{\nuTIME }R\|\bm{X}\|_{\foneonenu})^n}{R^{2n}}\Big(1+\frac{\|\bm{X}\|_{\foneonenu}}{2\sqrt{2}e^{\nuTIME }R}\Big)^{\!n}\!-\frac{2\sqrt{2}e^{\nuTIME }\|\bm{X}\|_{\foneonenu}}{R}\!\bigg)\|\bm{F}_L\|_{\fzeronenu}\\
=\frac{2\sqrt{2}e^{\nuTIME }\|\bm{X}\|_{\foneonenu}^3}{ R^3}\bigg(\sum_{n\geq1}\Big(\frac{2\sqrt{2}e^{\nuTIME }\|\bm{X}\|_{\foneonenu}}{R}\Big)^{\!n-1}\!\Big(1\!+\!\frac{\|\bm{X}\|_{\foneonenu}}{2\sqrt{2}e^{\nuTIME }R}\Big)^{\!n}\!\!\!-\!1\bigg)\|\bm{F}_L\|_{\fzeronenu}\\
\!=\!\frac{2\sqrt{2}e^{\nuTIME }\|\bm{X}\|_{\foneonenu}^3}{ R^3}\Bigg(\frac{\|\bm{X}\|_{\foneonenu}}{2\sqrt{2}e^{\nuTIME }R}\\
+\Big(1+\frac{\|\bm{X}\|_{\foneonenu}}{2\sqrt{2}e^{\nuTIME }R}\Big)
\sum_{n\geq2}\bigg(\frac{2\sqrt{2}e^{\nuTIME }\|\bm{X}\|_{\foneonenu}}{R}\Big(1+\frac{\|\bm{X}\|_{\foneonenu}}{2\sqrt{2}e^{\nuTIME }R}\Big)\bigg)^{n-1}\Bigg)\|\bm{F}_L\|_{\fzeronenu}.
\end{multline*}
Performing the sum in $n$ and 
using estimate \eqref{radiusbound}, we conclude that
\begin{equation}\label{N34bound}
\|\bm{\mathcal{N}}_{3,4}\|_{\fzeronenu}\leq 9C_6C_1^4
\|\bm{X}\|_{\foneonenu}^4\|\bm{F}_L\|_{\fzeronenu},
\end{equation}
with 
\begin{equation}\label{C6}
C_6=\frac19\Big(1+8e^{2\nu_m }\Big(1+\frac1{2\sqrt{2}}e^{-\nu_m }C_1\|\bm{X}\|_{\foneonenu}\Big)C_2\Big),
\end{equation}
where $C_1$, $C_2$ were defined in \eqref{C1}, \eqref{C2}.  This completes our estimate for $\bm{\mathcal{N}}_{3,4}$.

To deal with the term $\bm{\mathcal{N}}_{3,5}$ in \eqref{N3split}, we have to take a derivative in $\mathcal{R}_2$ from \eqref{R2}. This gives the splitting
\begin{equation}\label{N35split}
\bm{\mathcal{N}}_{3,5}(\theta)=\bm{O}_6(\theta)+\bm{O}_7(\theta)+\bm{O}_8(\theta),
\end{equation}
where
\begin{multline*}
\bm{O}_6(\theta)=\frac{1}{4\pi R^2}\int_\mathbb{S}
\sum_{n\geq1}\!\sum_{\substack{m=0 \\ n+m\geq 2}}^{n-1}\!\!\!\!\begin{pmatrix}
n\\m
\end{pmatrix} \!\!\frac{(-1)^{n}(n\!-\!m)}{R^{2n}}
(\Delta_\eta \bm{X}(\theta)\otimes \Delta_\eta \bm{X}(\theta))\\(2\Delta_\eta \bm{X}_c(\theta)^T \Delta_\eta \bm{X}(\theta))^{n-m-1}2\p_\theta\Delta_\eta \bm{X}_c(\theta)^T \Delta_\eta \bm{X}(\theta)|\Delta_\eta \bm{X}(\theta)|^{2m}\bm{F}_L(\eta)d\eta,
\end{multline*}
\begin{multline*}
\bm{O}_7(\theta)=\frac{1}{4\pi R^2}\int_\mathbb{S}
\sum_{n\geq1}\!\sum_{\substack{m=0 \\ n+m\geq 2}}^{n-1}\!\!\!\!\begin{pmatrix}
n\\m
\end{pmatrix} \!\!\frac{(-1)^{n}(n\!-\!m)}{R^{2n}}(\Delta_\eta \bm{X}(\theta)\otimes \Delta_\eta \bm{X}(\theta))\\
(2\Delta_\eta \bm{X}_c(\theta)^T \Delta_\eta \bm{X}(\theta))^{n-m-1}
2\Delta_\eta \bm{X}_c(\theta)^T \p_\theta\Delta_\eta \bm{X}(\theta)|\Delta_\eta \bm{X}(\theta)|^{2m}\bm{F}_L(\eta)d\eta,
\end{multline*}
\begin{multline*}
\bm{O}_8(\theta)=\frac{1}{2\pi R^2}\int_\mathbb{S}
\sum_{n\geq1}\!\sum_{m=1}^n\begin{pmatrix}
n\\m
\end{pmatrix} \!\!\frac{(-1)^{n}m}{R^{2n}}(\Delta_\eta \bm{X}(\theta)\otimes \Delta_\eta \bm{X}(\theta))\\
(2\Delta_\eta \bm{X}_c(\theta)^T \Delta_\eta \bm{X}(\theta))^{n-m}
|\Delta_\eta \bm{X}(\theta)|^{2(m-1)}\Delta_\eta \bm{X}(\theta)^T\p_\theta \Delta_\eta \bm{X}(\theta)\bm{F}_L(\eta)d\eta.
\end{multline*}
Comparing $\bm{O}_6$ and $\bm{O}_8$ to $\bm{O}_1$ and $\bm{O}_3$, respectively, in \eqref{N31split}, and recalling the bounds \eqref{O1n}, \eqref{O3n}, together with \eqref{XotimesX}, we find that
\begin{equation*}
\|\bm{O}_6\|_{\fzeronenu}\leq\frac{\|\bm{X}\|_{\foneonenu}^2}{2R^2}
\sum_{n\geq2}\frac{(2\sqrt{2})^n e^{n \nuTIME } n}{R^n}\|\bm{X}\|_{\foneonenu}^n\Big(1+\frac{\|\bm{X}\|_{\foneonenu}}{2\sqrt{2}e^{\nuTIME }R}\Big)^{n-1} \|\bm{F}_L\|_{\fzeronenu},
\end{equation*}
\begin{equation*}
\begin{aligned}
\|\bm{O}_8\|_{\fzeronenu}\leq \frac{\|\bm{X}\|_{\foneonenu}^2}{ R^2}
\sum_{n\geq1}\frac{n\|\bm{X}\|_{\foneonenu}^{n-1}}{R^{n-1}}\Big(2\sqrt{2}e^{\nuTIME }+\frac{\|\bm{X}\|_{\foneonenu}}{R}\Big)^{n-1}\frac{\|\bm{X}\|_{\foneonenu}^{2}}{R^2}\|\bm{F}_L\|_{\fzeronenu},
\end{aligned}
\end{equation*}
which after summation in $n$ the right side above becomes
\begin{equation*}
\|\bm{O}_6\|_{\fzeronenu}
\leq 4e^{2\nu_m }\frac{\Big(1+\frac{\|\bm{X}\|_{\foneonenu}}{2\sqrt{2}e^{\nu_m }R}\Big)
\bigg(2-\frac{2\sqrt{2}e^{\nu_m }\|\bm{X}\|_{\foneonenu}}{R}\Big(1+\frac{\|\bm{X}\|_{\foneonenu}}{2\sqrt{2}e^{\nu_m }R}\Big)\bigg)}{\bigg(1-2\sqrt{2}e^{\nu_m }\frac{\|\bm{X}\|_{\foneonenu}}{R}\Big(1+\frac{\|\bm{X}\|_{\foneonenu}}{2\sqrt{2}e^{\nu_m }R}\Big)\bigg)^2}\frac{\|\bm{X}\|_{\foneonenu}^{4}}{R^4}\|\bm{F}_L\|_{\fzeronenu},
\end{equation*}
\begin{equation*}
\|\bm{O}_8\|_{\fzeronenu}\leq \frac{1}{\bigg(1-2\sqrt{2}e^{\nu_m }\frac{\|\bm{X}\|_{\foneonenu}}{R}\Big(1+\frac{\|\bm{X}\|_{\foneonenu}}{2\sqrt{2}e^{\nu_m }R}\Big)\bigg)^2}\frac{\|\bm{X}\|_{\foneonenu}^{4}}{R^4}\|\bm{F}_L\|_{\fzeronenu}.
\end{equation*}
It is now clear that, for the same reason that the bound for $\bm{O}_2$ \eqref{O2aux} was the same than for $\bm{O}_1$ \eqref{O1aux}, the estimate for $\bm{O}_7$ is the same than the one for $\bm{O}_6$. 
 Therefore, with \eqref{radiusbound}, we conclude that
\begin{equation}\label{N35bound}
\|\bm{\mathcal{N}}_{3,5}\|_{\fzeronenu}\leq 17C_7C_1^4\|\bm{X}\|_{\foneonenu}^{4}\|\bm{F}_L\|_{\fzeronenu},
 \end{equation}
with
\begin{multline}\label{C7}
C_7=
\\
\frac{16}{17}e^{2\nu_m }C_2C_3\bigg(1\!-\!C_8^{-1}\sqrt{2}e^{\nu_m }\|\bm{X}\|_{\foneonenu}\Big(1\!+\!\frac{1}{2\sqrt{2}}C_8^{-1}e^{-\nu_m }\|\bm{X}\|_{\foneonenu}\Big)\!\bigg)\!+\frac{C_3^2}{17}\!,
\end{multline}
where we note that $C_7$ is indeed increasing in $\|\bm{X}\|_{\foneonenu}$ as can be seen because the infinite sums in the upper bounds of $\|\bm{O}_6\|_{\fzeronenu}$ and $\|\bm{O}_8\|_{\fzeronenu}$ above are indeed increasing.   
Further above we are also using 
\begin{equation}\label{C8}
C_8\eqdef \sqrt{1+\frac12\|\bm{X}\|_{\foneonenu}^2}.
\end{equation}
And we are further using $C_2$ and $C_3$ from \eqref{C2} and \eqref{C3}.

Recalling the bounds \eqref{aux3} and \eqref{aux4}, the remaining terms $\bm{\mathcal{N}}_{3,6}$ - $\bm{\mathcal{N}}_{3,11}$ can be estimated similarly, using also $C_1$, $C_6$ and $C_7$ from \eqref{C1}, \eqref{C6} and \eqref{C7}, to obtain that
\begin{equation}\label{N3611bound}
\begin{aligned}
\|\bm{\mathcal{N}}_{3,6}\|_{\foneonenu}&\leq 8e^{2\nu_m }C_1^2\|\bm{X}\|_{\foneonenu}^{2}\|\bm{F}_L\|_{\fzeronenu},\\
\|\bm{\mathcal{N}}_{3,7}\|_{\foneonenu}&\leq 8e^{2\nu_m }C_1^2\|\bm{X}\|_{\foneonenu}^{2}\|\bm{F}_L\|_{\fzeronenu},\\
\|\bm{\mathcal{N}}_{3,8}\|_{\foneonenu}&\leq 18\sqrt{2}e^{\nu_m }C_6C_1^3\|\bm{X}\|_{\foneonenu}^{3}\|\bm{F}_L\|_{\fzeronenu},\\
\|\bm{\mathcal{N}}_{3,9}\|_{\foneonenu}&\leq 34\sqrt{2}e^{\nu_m }C_7C_1^3\|\bm{X}\|_{\foneonenu}^{3}\|\bm{F}_L\|_{\fzeronenu},\\
\|\bm{\mathcal{N}}_{3,10}\|_{\foneonenu}&\leq 18e^{2\nu_m }C_6C_1^2\|\bm{X}\|_{\foneonenu}^{2}\|\bm{F}_L\|_{\fzeronenu},\\
\|\bm{\mathcal{N}}_{3,11}\|_{\fzeronenu}&\leq 34e^{2\nu_m}C_7C_1^2\|\bm{X}\|_{\foneonenu}^{2}\|\bm{F}_L\|_{\fzeronenu}.
\end{aligned}
\end{equation}
Therefore, from the splitting \eqref{N3split} and adding all the bounds \eqref{N31bound}, \eqref{N32bound}, \eqref{N33bound}, \eqref{N34bound}, \eqref{N35bound}, and \eqref{N3611bound}, we conclude that
\begin{equation}\label{N3}
\|\bm{\mathcal{N}}_{3}\|_{\fzeronenu}\leq \frac{147}{2}C_9C_1^2\|\bm{X}\|_{\foneonenu}^{2}\|\bm{F}_L\|_{\fzeronenu},
\end{equation}
where
\begin{equation}\label{C9}
\begin{aligned}
&C_9=\frac{2}{147}\Big(\frac92C_4\!+\!C_5\!+\!16e^{2\nu_m}\!+\!18e^{2\nu_m}C_6\!+\!34e^{2\nu_m}C_7\\
&\quad+\big(\sqrt{2}\!+\!18\sqrt{2}C_6\!+\!34\sqrt{2}C_7\big)e^{\nu_m}C_1\xoneonenu\!+\!\big(9C_6\!+\!17C_7\big)C_1^2\xoneonenu^2\Big),
\end{aligned}
\end{equation}
with $C_1$, $C_4$, $C_5$, $C_6$, and $C_7$ defined in \eqref{C1}, \eqref{C4}, \eqref{C5}, \eqref{C6}, and \eqref{C7}. 
\\

\noindent\underline{$\bm{\mathcal{N}}_2$ estimates:}
It is clear from \eqref{Ni} that the previous estimate for $\bm{\mathcal{N}}_3$ in \eqref{N3} is also valid for $\bm{\mathcal{N}}_2$, with $\|\bm{F}_L\|_{\fzeronenu}$ replaced by $\|\bm{F}_0\|_{\fzeronenu}$. Therefore we have that
\begin{equation}\label{N2}
\|\bm{\mathcal{N}}_{2}\|_{\fzeronenu}
\leq  \frac{147}{2}C_9C_1^2\|\bm{X}\|_{\foneonenu}^{2}\|\bm{F}_0\|_{\fzeronenu},
\end{equation}
with $C_9$ defined above in \eqref{C9}.
\\

\noindent\underline{$\bm{\mathcal{N}}_4$ estimates:}
We split the term $\bm{\mathcal{N}}_{4}$ in \eqref{Ni} following the splitting \eqref{Gsplit}:
\begin{equation*}
\bm{\mathcal{N}}_4(\theta)=\bm{\mathcal{N}}_{4,1}(\theta)+\bm{\mathcal{N}}_{4,2}(\theta)+\bm{\mathcal{N}}_{4,3}(\theta),
\end{equation*}
where
\begin{equation*}
\begin{aligned}
\bm{\mathcal{N}}_{4,1}(\theta)&=\int_{\mathbb{S}}\p_\theta\big( G_0(\Delta_\eta\bm{X}_c(\theta))\big)\mb{F}_N(\eta)d\eta,\\
\bm{\mathcal{N}}_{4,2}(\theta)&=\int_{\mathbb{S}}\p_\theta \big(G_L(\Delta_\eta\bm{X}_c(\theta),\Delta_\eta\bm{X}(\theta))(\theta)\big)\mb{F}_N(\eta)d\eta,\\
\bm{\mathcal{N}}_{4,3}(\theta)&=\int_{\mathbb{S}}\p_\theta\big( G_N(\Delta_\eta\bm{X}_c(\theta),\Delta_\eta\bm{X}(\theta)\big)\mb{F}_N(\eta)d\eta.\\
\end{aligned}
\end{equation*}
We notice that the term $\bm{\mathcal{N}}_{4,2}$ can be bounded exactly as $\bm{\mathcal{N}}_1$ in \eqref{Ni}, with $\|\bm{F}_L\|_{\fzeronenu}$ replaced by $\|\bm{F}_N\|_{\fzeronenu}$, that is, from \eqref{N1} we have
\begin{equation*}
\|\bm{\mathcal{N}}_{4,2}\|_{\fzeronenu}\leq 11\sqrt{2}e^{3\nu_m }C_1 \|\bm{X}\|_{\foneonenu}\|\bm{F}_N\|_{\fzeronenu}, 
\end{equation*}
with $C_1$ from \eqref{C1}.  Analogously using the similarity between $\bm{\mathcal{N}}_{4,3}$ and $\bm{\mathcal{N}}_{3}$ in \eqref{N3}, we have
\begin{equation*}
\|\bm{\mathcal{N}}_{4,3}\|_{\fzeronenu}\leq 
 \frac{147}{2}C_9C_1^2\|\bm{X}\|_{\foneonenu}^{2}\|\bm{F}_N\|_{\fzeronenu},
\end{equation*}
where we recall $C_9$ from \eqref{C9}.

Now taking a derivative in \eqref{G0}, the term $\bm{\mathcal{N}}_{4,1}$ can be written as follows
\begin{equation*}
\begin{aligned}
\bm{\mathcal{N}}_{4,1}(\theta)&=-\frac{1}{4\pi}\int_{\mathbb{S}}\frac{F_N(\eta)}{2\tan{\Big(\frac{\theta-\eta}{2}\Big)}}+\frac{1}{4\pi R^2}\int_{\mathbb{S}}\p_\theta\Delta_\eta\bm{X}_c(\theta)\otimes \Delta_\eta\bm{X}_c(\theta)\bm{F}_N(\eta)d\eta\\
&\quad+\frac{1}{4\pi R^2}\int_{\mathbb{S}}\Delta_\eta\bm{X}_c(\theta)\otimes \p_\theta\Delta_\eta\bm{X}_c(\theta)\bm{F}_N(\eta)d\eta,
\end{aligned}
\end{equation*}
and therefore, recalling \eqref{pXcopXc}, we have
\begin{equation*}
\|\bm{\mathcal{N}}_{4,1}\|_{\fzeronenu}\leq \big(\frac14+2e^{2\nu_m }\big)\|\bm{F}_N\|_{\fzeronenu}.
\end{equation*}
We add the previous bounds to obtain that
\begin{equation}\label{N4}
\|\bm{\mathcal{N}}_4\|_{\fzeronenu}\leq \frac94 C_{10}\|\bm{F}_N\|_{\fzeronenu},
\end{equation}
with
\begin{equation}\label{C10}
\begin{aligned}
C_{10}&=\frac49\bigg(\frac14+2e^{2\nu_m }+11\sqrt{2}e^{3\nu_m}C_1 \|\bm{X}\|_{\foneonenu}+\frac{147}{2}C_9C_1^2\|\bm{X}\|_{\foneonenu}^{2}\bigg),
\end{aligned}
\end{equation}
with $C_1$, $C_9$  defined in \eqref{C1} and \eqref{C9}.
Combining the estimates \eqref{N1}, \eqref{N2}, \eqref{N3}, and \eqref{N4}, we conclude from \eqref{Ni} that
\begin{equation*}
\begin{aligned}
\|\bm{\mathcal{N}}\|_{\foneonenu}&\leq \frac{147}{2}C_9C_1^2\|\bm{X}\|_{\foneonenu}\|\bm{F}_0\|_{\fzeronenu}\|\bm{X}\|_{\ftwoonenu}+11\sqrt{2}C_{11}C_1\|\bm{X}\|_{\foneonenu}\|\bm{F}_L\|_{\fzeronenu} \\
&\quad+\frac94 C_{10}\|\bm{F}_N\|_{\fzeronenu},
\end{aligned}
\end{equation*}
where 
\begin{equation}\label{C11}
C_{11}=
\\
\frac{1}{11\sqrt{2}}\Big(11\sqrt{2}e^{3\nu_m}+\frac{147}{2}C_9C_1\xoneonenu\Big),
\end{equation}
and $C_1$, $C_9$ are defined in \eqref{C1} and \eqref{C9}.
Rename the constants 
\begin{equation}\label{D1D2D3}
D_1=C_{11}C_1,\qquad D_2=C_9C_1^2,\qquad D_3=C_{10},
\end{equation}
to get the result \eqref{Nestimate}, where $C_1$, $C_9$, $C_{10}$, and $C_{11}$ are given in \eqref{C1}, \eqref{C9}, \eqref{C10}, and \eqref{C11}.
\end{proof}

\subsection{\textit{A Priori} Estimates on $\bm{F}$}\label{secF}

In this section we will obtain bounds for $\bm{F}_0$, $\bm{F}_L$, and $\bm{F}_N$ in $\mathcal{F}^{0,1}_\nu$. 

\begin{prop}  
	Assume that $\bm{X}\in\ftwoonenu$ and that $\bm{F}$ solves \eqref{viscosityjump}. Then, the functions $\bm{F}_0$ in \eqref{F0}, $\bm{F}_L$ in \eqref{F1} and $\bm{F}_N=\bm{F}-\bm{F}_0-\bm{F}_L$, satisfy the following estimates:
	\begin{equation}\label{F0bound}
    \begin{aligned}
        \|\bm{F}_0\|_{\fzeronenu}&\leq \sqrt{2}e^{\nu_m }C_8\frac{2A_e}{1-A_\mu},
    \end{aligned}
\end{equation}
where $C_8$ is defined in \eqref{C8}.  Further
\begin{equation}\label{F1bound}
    \begin{aligned}
    \|\bm{F}_L\|_{\fzeronenu}&\leq 2A_e\Big(1+\frac{|A_\mu|}{1-A_\mu}\Big)\xtwoonenu,
    \end{aligned}
\end{equation}	
and
\begin{equation}\label{FNbound}
    \|\bm{F}_N\|_{\fzeronenu}\leq 1000\sqrt{2}A_e \frac{|A_\mu|(1+|A_\mu|)}{(1-A_\mu)^2(1+A_\mu)} D_4\xoneonenu\xtwoonenu,
\end{equation}
where $D_4=D_4(\xoneonenu; A_\mu,\nu_m)$ is an increasing function of $\xoneonenu$ as in \eqref{constant.def} such that $$\lim_{\xoneonenu\to0^+}D_4(\xoneonenu;0,0)=1,$$
and is defined in \eqref{D4}.  
\end{prop}

\begin{proof}
First, for a general circle the expression for $\bm{F}_0$ in \eqref{F0star} becomes
\begin{equation}\label{F0}
\bm{F}_0(\theta)=\frac{2A_e}{1-A_\mu}\p_\theta^2 \bm{X}_c.
\end{equation}
 Similar to \eqref{pXcpX} using \eqref{circlefourier} we have for \eqref{F0bound} that
\begin{equation*}
    \begin{aligned}
        \|\bm{F}_0\|_{\fzeronenu}&\leq \sqrt{2}e^{\nu_m }R\frac{2A_e}{1-A_\mu}\leq \sqrt{2} e^{\nu_m }C_8\frac{2A_e}{1-A_\mu},
    \end{aligned}
\end{equation*}
where $C_8$ is given by \eqref{C8} and we used \eqref{radiusbound}.

Further $\bm{F}_L$ is given by \eqref{F1}
and so we have
\begin{equation*}
    \begin{aligned}
    \|\bm{F}_L\|_{\fzeronenu}&\leq 2A_e\xtwoonenu+\frac{2|A_\mu| A_e}{1-A_\mu}\xoneonenu,
    \end{aligned}
\end{equation*}
which gives \eqref{F1bound}.

We proceed with the expansion of the nonlinear terms in \eqref{viscosityjump}. 
First, using \eqref{expand.d}, we write
\begin{equation*}
\frac{1}{|\Delta \bm{X}+\Delta \bm{X}_c|^4}\!=\!\frac1{16R^4\sin^4{\left(\frac{\theta-\eta}{2}\right)}}\Big(1-\frac{4}{R^2}\Delta_\eta \bm{X}_c(\theta)^T \Delta_\eta \bm{X}(\theta)+\mathcal{R}_3(\Delta_\eta \bm{X}(\theta))\Big),
\end{equation*}
where
\begin{equation}\label{R3}
\begin{aligned}
\mathcal{R}_3(\Delta_\eta \bm{X}(\theta))&=-\frac{4}{R^2}\Delta_\eta \bm{X}_c(\theta)^T \Delta_\eta \bm{X}(\theta)\mathcal{R}_2(\Delta_\eta \bm{X}(\theta))
+2\mathcal{R}_2(\Delta_\eta \bm{X}(\theta))\\
&\quad+\frac{4}{R^4}(\Delta_\eta \bm{X}_c(\theta)^T \Delta_\eta \bm{X}(\theta))^2+(\mathcal{R}_2(\Delta_\eta \bm{X}(\theta)))^2,
\end{aligned}
\end{equation}
and $\mathcal{R}_2(\Delta_\eta \bm{X}(\theta))$ is given in \eqref{R2}.
Then, we use the above expansion to rewrite $\bm{\mathcal{S}}(\bm{F},\bm{\mathcal{X}})(\theta)$ from \eqref{S} as follows
\begin{equation}\label{auxS}
\begin{aligned}
\bm{\mathcal{S}}(\bm{F},\bm{\mathcal{X}})(\theta)=\int_{\mathbb{S}}K(\bm{X}_c,\bm{X})(\theta,\eta) \frac{\bm{F}(\theta-\eta)}{2\sin{(\eta/2)}}d\eta,
\end{aligned}
\end{equation}
where
\begin{equation*}
\begin{aligned}
K(\bm{X}_c,\bm{X})(\theta,\eta)=\frac1{\pi R^4}\big(\partial_\theta \bm{\mathcal{X}}(\theta)^\perp\big)^T\Delta_{\theta-\eta} \bm{\mathcal{X}}(\theta)\Delta_{\theta-\eta} \bm{\mathcal{X}}(\theta)\otimes \Delta_{\theta-\eta} \bm{\mathcal{X}}(\theta)\\
\Big(1-\frac{4}{R^2}\Delta_\eta \bm{X}_c(\theta)^T \Delta_\eta \bm{X}(\theta)+\mathcal{R}_3(\Delta_\eta \bm{X}(\theta))\Big),
\end{aligned}
\end{equation*}
and we recall the notation $\bm{\mathcal{X}}(\theta)=\bm{X}_c(\theta)+\bm{X}(\theta)$ and \eqref{deltaeta}. 

We will plug in the splitting for $\bm{F}$ in \eqref{Fsplit} into \eqref{viscosityjump}. We first introduce an analogous splitting for $K$ as
\begin{equation}\label{Ksplit}
K(\bm{X}_c,\bm{X})(\theta,\eta)=K_0(\bm{X}_c)(\theta,\eta)+K_L(\bm{X}_c,\bm{X})(\theta,\eta)+K_N(\bm{X}_c,\bm{X})(\theta,\eta).
\end{equation}
After we remove the zero order \eqref{Fstar2}, and linear order terms \eqref{Flintemp}, then the equation \eqref{viscosityjump}  for the non-linear order terms becomes the following equation for $\bm{F}_N$, 
\begin{equation}\label{FNequation}
\begin{aligned}
\bm{F}_N(\theta)-2A_\mu\int_\mathbb{S} K_0(\bm{X}_c)(\theta,\eta) \frac{\bm{F}_N(\theta-\eta)}{2\sin{(\eta/2)}}d\eta =\bm{J}(\bm{X},\bm{F}_N)(\theta),
\end{aligned}
\end{equation}
with
\begin{multline}\label{J}
\bm{J}(\bm{X},\bm{F}_N)(\theta)
\\
=2A_\mu\int_\mathbb{S}\Big(K_L(\bm{X}_c,\bm{X})(\theta,\eta)\!+\!K_N(\bm{X}_c,\bm{X})(\theta,\eta)\Big) \frac{\bm{F}_N(\theta\!-\!\eta)}{2\sin{(\eta/2)}}d\eta
\\
+2A_\mu\int_\mathbb{S}\big(K_L(\bm{X}_c,\bm{X})(\theta,\eta)+K_N(\bm{X}_c,\bm{X})(\theta,\eta)\big) \frac{\bm{F}_L(\theta-\eta)}{2\sin{(\eta/2)}}d\eta
\\
+2A_\mu\int_\mathbb{S}K_N(\bm{X}_c,\bm{X})(\theta,\eta)\bm{F}_0(\theta\!-\!\eta)\frac{d\eta}{2\sin{(\eta/2)}},
\end{multline}
where the first term in $\bm{J}$ will be treated as a perturbation with $\bm{F}_0$ and $\bm{F}_L$ given in \eqref{F0} and \eqref{F1} respectively. Notice that $K_0$ is given by
\begin{equation*}
 K_0(\bm{X}_c)(\theta,\eta)=\frac1{\pi R^4}\big(\partial_\theta \bm{X}_c(\theta)^\perp\big)^T\Delta_{\theta-\eta} \bm{X}_c(\theta)\Delta_{\theta-\eta} \bm{X}_c(\theta)\otimes \Delta_{\theta-\eta} \bm{X}_c(\theta),
\end{equation*}
where by \eqref{comput1} and \eqref{comput3} we have that
\begin{equation*}
\begin{aligned}
\big(\partial_\theta \bm{X}_c(\theta)^\perp\big)^T\Delta_{\theta-\eta} \bm{X}_c(\theta)=-R^2\sin{(\eta/2)},
\end{aligned}
\end{equation*}
\begin{equation*}
\begin{aligned}
\Delta_{\theta-\eta} \bm{X}_c(\theta)\otimes \Delta_{\theta-\eta} \bm{X}_c(\theta)
=\frac{a^2}2\begin{bmatrix}
1-\cos{(2\theta-\eta)}&-\sin{(2\theta-\eta)}\\
-\sin{(2\theta-\eta)}&1+\cos{(2\theta-\eta)}
\end{bmatrix}\\
+\frac{b^2}2\begin{bmatrix}
1\!+\!\cos{(2\theta-\eta)}&\sin{(2\theta-\eta)}\\
\sin{(2\theta-\eta)}&1\!-\!\cos{(2\theta-\eta)}
\end{bmatrix}\!+\!ab\begin{bmatrix}
\sin{(2\theta-\eta)}&\!-\!\cos{(2\theta-\eta)}\\
-\cos{(2\theta-\eta)}&\!-\!\sin{(2\theta-\eta)}
\end{bmatrix}.
\end{aligned}
\end{equation*}
Therefore, 
\begin{equation*}
\begin{aligned}
\int_\mathbb{S} K_0(\bm{X}_c)(\theta,\eta) &\frac{\bm{F}_N(\theta-\eta)}{2\sin{(\eta/2)}}d\eta=-\frac1{4\pi}\int_\mathbb{S}\bm{F}_N(\theta-\eta)d\eta\\
&-\frac{a^2-b^2}{4\pi R^2}\int_\mathbb{S}\begin{bmatrix}
-\cos{(2\theta-\eta)}&-\sin{(2\theta-\eta)}\\
-\sin{(2\theta-\eta)}&\cos{(2\theta-\eta)}
\end{bmatrix}\bm{F}_N(\theta-\eta)d\eta\\
&-\frac{ab}{2\pi R^2}\int_\mathbb{S}\begin{bmatrix}
\sin{(2\theta-\eta)}&\!-\!\cos{(2\theta-\eta)}\\
-\cos{(2\theta-\eta)}&\!-\!\sin{(2\theta-\eta)}
\end{bmatrix}\bm{F}_N(\theta-\eta)d\eta.
\end{aligned}
\end{equation*}
Then taking the Fourier transform we find that
\begin{equation*}
\begin{aligned}
\mathcal{F}\Big(\int_\mathbb{S} K_0(\bm{X}_c)(\theta,\eta) \frac{\bm{F}_N(\theta-\eta)}{2\sin{(\eta/2)}}d\eta\Big)(k)&
=-\frac12\widehat{\bm{F}_N}(0)\delta_0(k)
\\
&\quad
+\frac{(a+ib)^2}{4R^2}\begin{bmatrix}
1&-i\\
-i&-1
\end{bmatrix}\widehat{\bm{F}}_N(-1)\delta_1(k)\\
&\quad+\frac{(a-ib)^2}{4R^2}\begin{bmatrix}
1&i\\
i&-1
\end{bmatrix}\widehat{\bm{F}}_N(1)\delta_{-1}(k),
\end{aligned}
\end{equation*}
Equation \eqref{FNequation} is then given on the Fourier side by the following expressions:  
\begin{equation}\label{aux5}
\begin{aligned}
    \widehat{\bm{F}}_N(0)&=\frac{1}{1+A_\mu}\widehat{\bm{J}(\bm{X},\bm{F}_N)}(0),\\
    \widehat{\bm{F}}_N(k)&=\widehat{\bm{J}(\bm{X},\bm{F}_N)}(k),\qquad k\geq2,
\end{aligned}
\end{equation}
while for $k=1$ one has that 
\begin{equation*}
\begin{aligned}
\widehat{\bm{F}}_N(1)-A_\mu\frac{(a+ib)^2}{2R^2}\begin{bmatrix}
1&-i\\-i&-1
\end{bmatrix}\widehat{\bm{F}}_N(-1)&=\widehat{\bm{J}(\bm{X},\bm{F}_N)}(1),\\
\widehat{\bm{F}}_N(-1)-A_\mu\frac{(a-ib)^2}{2R^2}\begin{bmatrix}
1&i\\i&-1
\end{bmatrix}\widehat{\bm{F}}_N(1)&=\widehat{\bm{J}(\bm{X},\bm{F}_N)}(-1),
\end{aligned}
\end{equation*}
which gives that
\begin{equation*}
\begin{aligned}
\begin{bmatrix}
1-A_\mu^2/2&-iA_\mu^2/2\\
iA_\mu^2/2&1-A_\mu^2/2
\end{bmatrix}\widehat{\bm{F}}_N(1)&=A_\mu\frac{(a+ib)^2}{2R^2}\begin{bmatrix}
1&-i\\-i&-1
\end{bmatrix}\widehat{\bm{J}(\bm{X},\bm{F}_N)}(-1)\\
&\quad+\widehat{\bm{J}(\bm{X},\bm{F}_N)}(1),
\end{aligned}
\end{equation*}
and thus
\begin{equation*}
\begin{aligned}
\widehat{\bm{F}}_N(1)&=\frac{A_\mu}{1-A_\mu^2}\frac{(a+ib)^2}{2R^2}\begin{bmatrix}
1&-i\\-i&-1
\end{bmatrix}\widehat{\bm{J}(\bm{X},\bm{F}_N)}(-1)\\
&\quad+\frac{1}{1-A_\mu^2}\begin{bmatrix}
1-A_\mu^2/2&iA_\mu^2/2\\
-iA_\mu^2/2&1-A_\mu^2/2
\end{bmatrix}\widehat{\bm{J}(\bm{X},\bm{F}_N)}(1).
\end{aligned}
\end{equation*}
Since we have that
\begin{equation*}
\Big|\Big|\begin{bmatrix}
1&-i\\-i&-1
\end{bmatrix}\Big|\Big|=2,\qquad \Big|\Big|\begin{bmatrix}
1-A_\mu^2/2&iA_\mu^2/2\\
-iA_\mu^2/2&1-A_\mu^2/2
\end{bmatrix}\Big|\Big|=1, 
\end{equation*}
we obtain 
\begin{equation*}
\begin{aligned}
|\widehat{\bm{F}}_N(1)|&\leq \frac{|A_\mu|}{1-A_\mu^2}| \widehat{\bm{J}(\bm{X},\bm{F}_N)}(-1)|+\frac{1}{1-A_\mu^2}|\widehat{\bm{J}(\bm{X},\bm{F}_N)}(1)|\\
&=\frac{1+|A_\mu|}{(1-A_\mu)(1+A_\mu)}|\widehat{\bm{J}(\bm{X},\bm{F}_N)}(1)|,
\end{aligned}
\end{equation*}
which together with \eqref{aux5} implies that 
\begin{equation}\label{FNbound2}
\|\bm{F}_N\|_{\fzeronenu}\leq\frac{1+|A_\mu|}{(1-A_\mu)(1+A_\mu)}\|\bm{J}(\bm{X},\bm{F}_N)\|_{\fzeronenu}.
\end{equation}
This is our estimate for $\bm{F}_N$.
\vspace{0.2cm}

\noindent\underline{$\bm{J}(\bm{X},\bm{F}_N)$ estimate:} 
\vspace{0.2cm}

Notice that $\bm{J}(\bm{X},\bm{F}_N)$ corresponds to the nonlinear terms in $\bm{\mathcal{S}}(\bm{F},\bm{\mathcal{X}})$ except the one in the left-hand side of \eqref{FNequation}. For simplicity in notation, we are going to estimate $\bm{\mathcal{S}}(\bm{F},\bm{\mathcal{X}})$, and later extract from there the corresponding bounds for $\bm{J}(\bm{X},\bm{F}_N)$.

Consider the following splitting for $\bm{\mathcal{S}}(\bm{F},\bm{\mathcal{X}})$ from \eqref{S}:
\begin{equation}\label{splittingS}
    \begin{aligned}
    \bm{\mathcal{S}}(\bm{F},\bm{\mathcal{X}})(\theta)=\bm{\mathcal{S}}_1(\bm{F},\bm{\mathcal{X}})(\theta)+
    \bm{\mathcal{S}}_2(\bm{F},\bm{\mathcal{X}})(\theta)+\bm{\mathcal{S}}_3(\bm{F},\bm{\mathcal{X}})(\theta),
    \end{aligned}
\end{equation}
with
\begin{multline*}
    \bm{\mathcal{S}}_1(\bm{F},\bm{\mathcal{X}})(\theta)
    \\
    \!=\!\frac{1}{\pi R^4}\!\!\int_\mathbb{S}\!\!\big(\partial_\theta \bm{\mathcal{X}}(\theta)^\perp\big)^T\!\Delta_{\theta-\eta} \bm{\mathcal{X}}(\theta)\Delta_{\theta-\eta} \bm{\mathcal{X}}(\theta)\!\otimes\! \Delta_{\theta-\eta} \bm{\mathcal{X}}(\theta)\frac{\bm{F}(\theta\!-\!\eta)}{2\sin{(\eta/2)}}d\eta,
\end{multline*}
and
\begin{equation*}
\begin{aligned}
    \bm{\mathcal{S}}_2(\bm{F},\bm{\mathcal{X}})(\theta)&=
    \frac{-4}{\pi R^6}\int_\mathbb{S}\big(\partial_\theta \bm{\mathcal{X}}(\theta)^\perp\big)^T\Delta_{\theta-\eta} \bm{\mathcal{X}}(\theta)\Delta_{\theta-\eta} \bm{\mathcal{X}}(\theta)\otimes \Delta_{\theta-\eta} \bm{\mathcal{X}}(\theta)\\
    &\hspace{4cm}\cdot\Delta_\eta \bm{X}_c(\theta)^T \Delta_\eta \bm{X}(\theta)\frac{\bm{F}(\theta-\eta)}{2\sin{(\eta/2)}}d\eta,\\
    \bm{\mathcal{S}}_3(\bm{F},\bm{\mathcal{X}})(\theta)&=\frac{1}{\pi R^4}\int_\mathbb{S}\big(\partial_\theta \bm{\mathcal{X}}(\theta)^\perp\big)^T\Delta_{\theta-\eta} \bm{\mathcal{X}}(\theta)\Delta_{\theta-\eta} \bm{\mathcal{X}}(\theta)\otimes \Delta_{\theta-\eta} \bm{\mathcal{X}}(\theta)\\
    &\hspace{4cm}\cdot\mathcal{R}_3(\Delta_\eta \bm{X}(\theta))\frac{\bm{F}(\theta-\eta)}{2\sin{(\eta/2)}}d\eta.
\end{aligned}
\end{equation*}
We take Fourier transform of $\bm{\mathcal{S}}_1(\bm{F},\bm{\mathcal{X}})$ to obtain 
\begin{multline*}
    \widehat{\bm{\mathcal{S}}_1(\bm{F},\bm{\mathcal{X}})}(k)=
    \frac{1}{\pi R^4}\sum_{k_1\in\mathbb{Z}}\cdots\sum_{k_4\in\mathbb{Z}}\big(\widehat{\p_\theta\bm{\mathcal{X}}}(k-k_1)^\perp\big)^T\widehat{\p_\theta\bm{\mathcal{X}}}(k_1-k_2)\\
    \widehat{\p_\theta\bm{\mathcal{X}}}(k_2-k_3)\otimes \widehat{\p_\theta\bm{\mathcal{X}}}(k_3-k_4)\widehat{\bm{F}}(k_4)I_2'(k_1,\dots,k_4),
\end{multline*}
where
\begin{equation*}
\begin{aligned}
|I'_2(k_1,\dots,k_4)|&=|\int_\mathbb{S}\prod_{j=1}^3\frac{\sin{((k_j-k_{j+1})\eta/2)}}{(k_j-k_{j+1})\sin{(\eta/2)}}\frac{e^{-i(k_1+k_4)\eta/2}}{2\sin{(\eta/2)}}d\eta|\\
&=|\int_\mathbb{S}\frac{\sin{((k_1+k_4)\eta/2)}}{\sin{(\eta/2)}}\prod_{j=1}^3\frac{\sin{((k_j-k_{j+1})\eta/2)}}{(k_j-k_{j+1})\sin{(\eta/2)}}d\eta|.
\end{aligned}
\end{equation*}
The integral $I'_2$ turns out to be the previously defined integral in \eqref{Sn.Integral}.  
In Lemma \ref{lemmaI} we show that $|I_2'|\leq 2\pi$.
Using \eqref{circlefourier} and \eqref{pXcpX},
we have that
\begin{equation*}
    \begin{aligned}
    |\big(\widehat{\p_\theta\bm{\mathcal{X}}}&(k-k_1)^\perp\big)^T\widehat{\p_\theta\bm{\mathcal{X}}}(k_1-k_2)|\\
    &\leq \frac{R^2}2\big(\delta_1(k-k_1)\delta_{-1}(k_1-k_2)+\delta_{-1}(k-k_1)\delta_1(k_1-k_2)\big)\\
    &\quad
    +\frac{R}{\sqrt{2}}\delta_{1,-1}(k\!-\!k_1)| \widehat{\p_\theta\bm{X}}(k_1\!-\!k_2)|
    +
    \frac{R}{\sqrt{2}}\delta_{1,-1}(k_1\!-\!k_2)| \widehat{\p_\theta\bm{X}}(k\!-\!k_1)|
    \\
     &\quad
    +
    |\widehat{\p_\theta\bm{X}}(k\!-\!k_1)|| \widehat{\p_\theta\bm{X}}(k_1\!-\!k_2)|,
    \end{aligned}
\end{equation*}
while recalling \eqref{pXcopXc}, \eqref{aux3} and \eqref{aux4}, we obtain that
\begin{equation*}
    \begin{aligned}
    \|\widehat{\p_\theta\bm{\mathcal{X}}}&(k_2-k_3)\otimes \widehat{\p_\theta\bm{\mathcal{X}}}(k_3-k_4)\|\leq \frac{R^2}{2}\delta_{1,-1}(k_2-k_3)\delta_{1,-1}(k_3-k_4)\\
    &+\frac{\sqrt{2}}{2}R\delta_{1,-1}(k_2-k_3)| \widehat{\p_\theta\bm{X}}(k_3-k_4)|+\frac{\sqrt{2}}2R| \widehat{\p_\theta\bm{X}}(k_2-k_3)|
    \delta_{1,-1}(k_3-k_4)\\
    &+|\widehat{\p_\theta\bm{X}}(k_2-k_3)||\widehat{\p_\theta\bm{X}}(k_3-k_4)|.
    \end{aligned}
\end{equation*}
Therefore, Young's inequality for convolutions yields that 
\begin{equation*}
\begin{aligned}
    \|\bm{\mathcal{S}}_1(\bm{F},\bm{\mathcal{X}})\|_{\fzeronenu}&\leq 2\Big(e^{2\nu_m}+2\sqrt{2}e^{\nu_m }\frac{\xoneonenu}{R}+\frac{\xoneonenu^2}{R^2}\Big)\\
    &\quad\cdot\Big(2e^{2\nu_m }+2\sqrt{2}e^{\nu_m }\frac{\xoneonenu}{R}+\frac{\xoneonenu^2}{R^2}\Big)\|\bm{F}\|_{\fzeronenu},
    \end{aligned}
\end{equation*}
which by \eqref{radiusbound} and notation \eqref{C1}, rewrites as
\begin{equation}\label{S1bound}
    \begin{aligned}
    \|\bm{\mathcal{S}}_1(\bm{F},\bm{\mathcal{X}})\|_{\fzeronenu}&\leq 2\Big(2e^{4\nu_m }+6\sqrt{2}e^{3\nu_m}C_1\xoneonenu\\
    &\quad+11C_{12}C_1^2\xoneonenu^2\Big)\|\bm{F}\|_{\fzeronenu},
    \end{aligned}
\end{equation}
with
\begin{equation}\label{C12}
    C_{12}=\frac1{11}\Big(11e^{2\nu_m }+4\sqrt{2}e^{\nu_m }C_1\xoneonenu+C_1^2\xoneonenu^2\Big),
\end{equation}
and $C_1$ is defined in \eqref{C1}.

Following the same steps, for $\widehat{\bm{\mathcal{S}}_2(\bm{F},\bm{\mathcal{X}})}$, one finds that
\begin{equation*}
\begin{aligned}
    \|\bm{\mathcal{S}}_2(\bm{F},\bm{\mathcal{X}})\|_{\fzeronenu}
    &
    \leq 8\sqrt{2}
    e^{\nu_m }\Big(2e^{4\nu_m }+6\sqrt{2}e^{3\nu_m}C_1\xoneonenu\\
    &\quad+11C_{12}C_1^2\xoneonenu^2\Big)C_1\xoneonenu\|\bm{F}\|_{\fzeronenu}.
    \end{aligned}
\end{equation*}
We define
\begin{equation}\label{C13}
    \begin{aligned}
    C_{13}&=\frac12\Big(2e^{4\nu_m }+6\sqrt{2}e^{3\nu_m}C_1\xoneonenu+11C_{12}C_1^2\xoneonenu^2\Big),
    \end{aligned}
\end{equation}
so that
\begin{equation}\label{S2bound}
\begin{aligned}
    \|\bm{\mathcal{S}}_2(\bm{F},\bm{\mathcal{X}})\|_{\fzeronenu}
    &
    \leq 16\sqrt{2}
    e^{\nu_m }C_{13}C_1\xoneonenu\|\bm{F}\|_{\fzeronenu}.
    \end{aligned}
\end{equation}
This completes our $\widehat{\bm{\mathcal{S}}_2(\bm{F},\bm{\mathcal{X}})}$ estimate.  

Next, we proceed with $\widehat{\bm{\mathcal{S}}_3(\bm{F},\bm{\mathcal{X}})}$ in \eqref{splittingS}. We split it accordingly to \eqref{R3} as follows
\begin{equation*}
  \widehat{\bm{\mathcal{S}}_3(\bm{F},\bm{\mathcal{X}})}=\widehat{\bm{\mathcal{S}}_{3,1}(\bm{F},\bm{\mathcal{X}})}+\widehat{\bm{\mathcal{S}}_{3,2}(\bm{F},\bm{\mathcal{X}})} +\widehat{\bm{\mathcal{S}}_{3,3}(\bm{F},\bm{\mathcal{X}})}+\widehat{\bm{\mathcal{S}}_{3,4}(\bm{F},\bm{\mathcal{X}})},
\end{equation*}
with
\begin{equation*}
    \begin{aligned}
    \bm{\mathcal{S}}_{3,1}(\bm{F},\bm{\mathcal{X}})=\frac{-4}{\pi R^6}\int_\mathbb{S}\big(\partial_\theta &\bm{\mathcal{X}}(\theta)^\perp\big)^T\Delta_{\theta-\eta} \bm{\mathcal{X}}(\theta)\Delta_{\theta-\eta} \bm{\mathcal{X}}(\theta)\otimes \Delta_{\theta-\eta} \bm{\mathcal{X}}(\theta)\\
    &\hspace{-0.3cm}\Delta_{\theta-\eta} \bm{X}_c(\theta)^T \Delta_{\theta-\eta} \bm{X}(\theta)\mathcal{R}_2(\Delta_\eta \bm{X}(\theta))
    \frac{\bm{F}(\theta-\eta)}{2\sin{(\eta/2)}}d\eta,\\
    \bm{\mathcal{S}}_{3,2}(\bm{F},\bm{\mathcal{X}})=\frac{2}{\pi R^4}\int_\mathbb{S}\big(\partial_\theta &\bm{\mathcal{X}}(\theta)^\perp\big)^T\Delta_{\theta-\eta} \bm{\mathcal{X}}(\theta)\Delta_{\theta-\eta} \bm{\mathcal{X}}(\theta)\otimes \Delta_{\theta-\eta} \bm{\mathcal{X}}(\theta)\\
    &\hspace{3cm}\cdot{R}_2(\Delta_\eta \bm{X}(\theta))
    \frac{\bm{F}(\theta-\eta)}{2\sin{(\eta/2)}}d\eta,\\
    \bm{\mathcal{S}}_{3,3}(\bm{F},\bm{\mathcal{X}})=\frac{4}{\pi R^8}\int_\mathbb{S}\big(\partial_\theta &\bm{\mathcal{X}}(\theta)^\perp\big)^T\Delta_{\theta-\eta} \bm{\mathcal{X}}(\theta)\Delta_{\theta-\eta} \bm{\mathcal{X}}(\theta)\otimes \Delta_{\theta-\eta} \bm{\mathcal{X}}(\theta)\\
    &\hspace{1.3cm}\cdot(\Delta_{\theta-\eta} \bm{X}_c(\theta)^T\Delta_{\theta-\eta} \bm{X}(\theta))^2
    \frac{\bm{F}(\theta-\eta)}{2\sin{(\eta/2)}}d\eta,\\
    \bm{\mathcal{S}}_{3,4}(\bm{F},\bm{\mathcal{X}})=\frac{1}{\pi R^4}\int_\mathbb{S}\big(\partial_\theta &\bm{\mathcal{X}}(\theta)^\perp\big)^T\Delta_{\theta-\eta} \bm{\mathcal{X}}(\theta)\Delta_{\theta-\eta} \bm{\mathcal{X}}(\theta)\otimes \Delta_{\theta-\eta} \bm{\mathcal{X}}(\theta)\\
    &\hspace{2.5cm}\cdot({R}_2(\Delta_\eta \bm{X}(\theta)))^2
    \frac{\bm{F}(\theta-\eta)}{2\sin{(\eta/2)}}d\eta.
    \end{aligned}
\end{equation*}
The procedure follows the steps used to bound $\bm{\mathcal{N}}_{3,4}(\theta)$ in \eqref{N3split} and \eqref{N34bound}, where the term $\mathcal{R}_2$ from \eqref{R2} was also involved. After taking Fourier transform and using Lemma \ref{lemmaI},  Young's inequality for convolutions and  summation in $m$ and $n$ gives that
\begin{equation*}
    \begin{aligned}
    \|\bm{\mathcal{S}}_{3,1}(\bm{F},\bm{\mathcal{X}})\|_{\fzeronenu}&\leq 144\sqrt{2}
    e^{\nu_m }
    C_6C_{13}C_1^3\xoneonenu^3   \|\bm{F}\|_{\fzeronenu},
    \end{aligned}
\end{equation*}
\begin{equation*}
    \begin{aligned}
    \|\bm{\mathcal{S}}_{3,2}(\bm{F},\bm{\mathcal{X}})\|_{\fzeronenu}&\leq 72
    C_6C_{13}C_1^2\xoneonenu^2   \|\bm{F}\|_{\fzeronenu},
    \end{aligned}
\end{equation*}
\begin{equation*}
    \begin{aligned}
    \|\bm{\mathcal{S}}_{3,3}(\bm{F},\bm{\mathcal{X}})\|_{\fzeronenu}&\leq 32 e^{2\nu_m }
    C_{13}C_1^2\xoneonenu^2   \|\bm{F}\|_{\fzeronenu},
    \end{aligned}
\end{equation*}
\begin{equation*}
    \begin{aligned}
    \|\bm{\mathcal{S}}_{3,4}(\bm{F},\bm{\mathcal{X}})\|_{\fzeronenu}&\leq 324C_{13}C_6^2C_1^4\xoneonenu^4   \|\bm{F}\|_{\fzeronenu}.
    \end{aligned}
\end{equation*}
Joining the above bounds, we obtain that
\begin{equation}\label{S3bound}
    \begin{aligned}
    \|\bm{\mathcal{S}}_{3}(\bm{F},\bm{\mathcal{X}})\|_{\fzeronenu}&\leq 104C_{13}C_{14}C_1^2\xoneonenu^2\|\bm{F}\|_{\fzeronenu}
    \end{aligned}
\end{equation}
with
\begin{equation}\label{C14}
    \begin{aligned}
    C_{14}&\!=\!\frac{1}{104}\Big(72C_6\!+\!32e^{2\nu_m}\!+\!144\sqrt{2}e^{\nu_m}C_6C_1 \xoneonenu\!+\!324C_6^2C_1^2\xoneonenu^2\Big),
    \end{aligned}
\end{equation}
where $C_1$ and $C_6$ previously defined in \eqref{C1} and \eqref{C6}, respectively.
We combine the bounds \eqref{S1bound}, \eqref{S2bound}, and \eqref{S3bound}, and order them as follows
\begin{equation*}
    \begin{aligned}
    \|\bm{\mathcal{S}}(\bm{F},\bm{\mathcal{X}})\|_{\fzeronenu}&\leq 4e^{4\nu_m }\|\bm{F}\|_{\fzeronenu}+28\sqrt{2}e^{5\nu_m }C_1\xoneonenu\|\bm{F}\|_{\fzeronenu}\\
    &\quad+222C_{15}C_1^2\xoneonenu^2\|\bm{F}\|_{\fzeronenu},
        \end{aligned}
\end{equation*}
with
\begin{multline}\label{C15}
    C_{15}
    =\frac{1}{222}\Big(104C_{13}C_{14}\!+\!8\sqrt{2}e^{\nu_m}(6\sqrt{2}e^{3\nu_m}\!+\!11C_{12}C_1\xoneonenu)\!+\!22C_{12}\Big),
\end{multline}
and $C_1$ in \eqref{C1}, $C_{12}$ in \eqref{C12}, $C_{13}$ in \eqref{C13} and $C_{14}$ in \eqref{C14}.

We remark that instead of \eqref{splittingS}, analogously to the splitting for $K$ in \eqref{Ksplit} we can split
\begin{equation}\label{splittingS.L}
    \begin{aligned}
    \bm{\mathcal{S}}(\bm{F},\bm{\mathcal{X}})(\theta)=
    \bm{\mathcal{S}}_0(\bm{F},\bm{\mathcal{X}})(\theta)+\bm{\mathcal{S}}_L(\bm{F},\bm{\mathcal{X}})(\theta)+\bm{\mathcal{S}}_N(\bm{F},\bm{\mathcal{X}})(\theta),
    \end{aligned}
\end{equation}
where from \eqref{auxS} and \eqref{Ksplit} we have 
\begin{equation}\notag
\begin{aligned}
\bm{\mathcal{S}}_0(\bm{F},\bm{\mathcal{X}})(\theta)=\int_{\mathbb{S}}K_0(\bm{X}_c,\bm{X})(\theta,\eta) \frac{\bm{F}(\theta-\eta)}{2\sin{(\eta/2)}}d\eta.
\end{aligned}
\end{equation}
Then $\bm{\mathcal{S}}_L(\bm{F},\bm{\mathcal{X}})(\theta)$ analogously contains $K_L(\bm{X}_c,\bm{X})(\theta,\eta)$ from \eqref{Ksplit} and $\bm{\mathcal{S}}_L$ is linear in $\bm{{X}}$.  Then $\bm{\mathcal{S}}_N(\bm{F},\bm{\mathcal{X}})(\theta)$ similarly contains $K_N(\bm{X}_c,\bm{X})(\theta,\eta)$ from \eqref{Ksplit} and $\bm{\mathcal{S}}_N$ is  nonlinear in $\bm{{X}}$.  Then it is clear from the above that we have the estimates 
\begin{equation*}
    \begin{aligned}
    \|\bm{\mathcal{S}}_0(\bm{F},\bm{\mathcal{X}})\|_{\fzeronenu}&\leq 4e^{4\nu_m }\|\bm{F}\|_{\fzeronenu},
    \\
    \|\bm{\mathcal{S}}_L(\bm{F},\bm{\mathcal{X}})\|_{\fzeronenu}
    &\leq 28\sqrt{2}e^{5\nu_m }C_1\xoneonenu\|\bm{F}\|_{\fzeronenu},
    \\
    \|\bm{\mathcal{S}}_N(\bm{F},\bm{\mathcal{X}})\|_{\fzeronenu}
    &\leq 
    222C_{15}C_1^2\xoneonenu^2\|\bm{F}\|_{\fzeronenu},
        \end{aligned}
\end{equation*}
This splits the estimates into zero order, linear order, and nonlinear which is useful because of \eqref{J}.

Next, recalling the definition of $\bm{J}(\bm{X},\bm{F}_N)$ from  \eqref{J} and its relation with $\bm{\mathcal{S}}(\bm{F},\bm{\mathcal{X}})$ in \eqref{auxS} and \eqref{splittingS.L}, it follows that
\begin{equation}\label{Jbound}
    \begin{aligned}
    \|\bm{J}(\bm{X},\bm{F}_N)\|_{\fzeronenu}&\leq 56\sqrt{2}|A_\mu| C_{16}C_1\xoneonenu\|\bm{F}_N\|_{\fzeronenu}\\
    &\quad+56\sqrt{2}|A_\mu| C_{16}C_1\xoneonenu\|\bm{F}_L\|_{\fzeronenu}\\
    &\quad+444|A_\mu| C_{15}C_1^2\xoneonenu^2\|\bm{F}_0\|_{\fzeronenu},
    \end{aligned}
\end{equation}
with
\begin{equation}\label{C16}
    C_{16}=\frac{1}{28\sqrt{2}}\Big(28\sqrt{2}e^{5\nu_m}+222C_{15}C_1\xoneonenu\Big).
\end{equation}
Finally, bound \eqref{Jbound} allows us to estimate $\bm{F}_N$ from \eqref{FNbound2},
\begin{equation*}
\begin{aligned}
    \|\bm{F}_N\|_{\fzeronenu}&\leq \frac{56\sqrt{2}|A_\mu|(1+|A_\mu|)}{(1-A_\mu)(1+A_\mu)}C_{17}C_{16}C_1\xoneonenu\|\bm{F}_L\|_{\fzeronenu}\\
    &\quad+
    \frac{444|A_\mu|(1+|A_\mu|)}{(1-A_\mu)(1+A_\mu)}C_{17}C_{15}C_1^2\xoneonenu^2\|\bm{F}_0\|_{\fzeronenu},
\end{aligned}
\end{equation*}
where
\begin{equation}\label{C17}
    C_{17}=\Big(1-\frac{56\sqrt{2}|A_\mu|(1+|A_\mu|)}{(1-A_\mu)(1+A_\mu)}C_{16}C_1\xoneonenu\Big)^{-1},
\end{equation}
and the bounds for $\bm{F}_0$, $\bm{F}_L$ are given in \eqref{F0bound} and \eqref{F1bound}. Substituting these bounds we find 
\begin{equation*}
    \begin{aligned}
        \|\bm{F}_N\|_{\fzeronenu}&\leq 112\sqrt{2}A_e\frac{|A_\mu|(1+|A_\mu|)(1-A_\mu+|A_\mu|)}{(1-A_\mu)^2(1+A_\mu)}C_{17}C_{16}C_1\xoneonenu\xtwoonenu\\
        &\quad+ 888\sqrt{2}A_e\frac{|A_\mu|(1+|A_\mu|)}{(1-A_\mu)^2(1+A_\mu)}e^{\nu_m}C_8C_{17}C_{15}C_1^2\xoneonenu\xtwoonenu.
            \end{aligned}
\end{equation*}
 Denoting 
\begin{equation}\label{D4}
    \begin{aligned}
        D_4&=\frac{1}{1000}C_1C_{17}\Big(112(1-A_\mu+|A_\mu|)C_{16}+888e^{\nu_m}C_8C_1C_{15}\Big),
    \end{aligned}
\end{equation}
where $C_1$, $C_8$, $C_{15}$, $C_{16}$, and $C_{17}$ are defined in \eqref{C1}, \eqref{C8}, \eqref{C15}, \eqref{C16}, and \eqref{C17}, we can write the estimate for $\bm{F}_N$ as \eqref{FNbound}.  This completes the proof.
\end{proof}

We will now give the proof of Lemma \ref{lemmaI}.

\begin{proof}[Proof of Lemma \ref{lemmaI}] 
Recalling \eqref{In.Integral} and \eqref{mmult} and using the odd part of the integral we can rewrite $I = I_n$ as
    	\begin{multline*}
I = -i\textup{pv}\!\int_{-\pi}^\pi
\frac{\sin{((k_1 + k_{2n})\eta/2)}-\frac{\sin{((k-k_{1})\eta/2)}}{(k-k_{1})\tan{(\eta/2)}}\sin{((k+k_{2n})\eta/2)}}{2\sin{(\eta/2)}}
\\
\times
\prod_{j=1}^{2n-1}\frac{\sin{((k_j-k_{j+1})\eta/2)}}{(k_j-k_{j+1})\sin{(\eta/2)}}
d\eta
=
\frac{-i}{2}\left(I'-I^{\prime\prime}\right),
	\end{multline*}
where
\begin{equation}\label{def.I31}
    I' \eqdef \textup{pv}\!\int_{-\pi}^\pi
\frac{\sin{((k_1 + k_{2n})\eta/2)}}{\sin{(\eta/2)}}
\prod_{j=1}^{2n-1}\frac{\sin{((k_j-k_{j+1})\eta/2)}}{(k_j-k_{j+1})\sin{(\eta/2)}}
d\eta,
\end{equation}
and
\begin{equation}\label{def.I32}
    I^{\prime\prime} 
    \eqdef \textup{pv}\!\int_{-\pi}^\pi
\cos{(\eta/2)}\frac{\sin{((k+k_{2n})\eta/2)}}{\sin{(\eta/2)}}
\prod_{j=0}^{2n-1}\frac{\sin{((k_j-k_{j+1})\eta/2)}}{(k_j-k_{j+1})\sin{(\eta/2)}}
d\eta.
\end{equation}	
 Note that if $k_1+k_{2n}=0$ then $I'=0$ and if $k+k_{2n}=0$ then $I^{\prime\prime}=0$.  We henceforth assume that $|k_1+k_{2n}|\ge 1$ and $|k+k_{2n}|\ge 1$.  We will calculate \eqref{def.I31} and then \eqref{def.I32}.

Notice that we have $\sin{((k_j-k_{j+1})\eta/2)}= \sign(k_j-k_{j+1})\sin{(|k_j-k_{j+1}|\eta/2)}$ and, since $|k_j-k_{j+1}|\ge 1$, we rewrite the quotient in the product form as follows	
	\begin{multline*}
\frac{\sin{(|k_j-k_{j+1}|\eta/2)}}{\sin{(\eta/2)}}
	=\frac{e^{i|k_j-k_{j+1}|\eta/2}-e^{-i|k_j-k_{j+1}|\eta/2}}{e^{i\eta/2}-e^{-i\eta/2}}
	\\
	=\frac{e^{i|k_j-k_{j+1}|\eta/2}(1-e^{-i|k_j-k_{j+1}|\eta})}{e^{i\eta/2}(1-e^{-i\eta})}
	=e^{i(|k_j-k_{j+1}|-1)\eta/2}\sum_{m=0}^{|k_j-k_{j+1}|-1}e^{-i\eta m}
	\\
	=\sum_{m=0}^{|k_j-k_{j+1}|-1}e^{i(-2m+|k_j-k_{j+1}|-1)\eta/2}.
	\end{multline*}
	We conclude that
	\begin{multline*}
	\prod_{\substack{j=1}}^{2n-1}\frac{\sin{(|k_j-k_{j+1}|\eta/2)}}{\sin{(\eta/2)}}
	=
\prod_{\substack{j=1}}^{2n-1}\sum_{m_j=0}^{|k_j-k_{j+1}|-1}e^{i(-2m_j+|k_j-k_{j+1}|-1)\eta/2}
\\
	=\sum_{m_1=0}^{|k_1-k_{2}|-1}\cdots \sum_{m_{2n-1}=0}^{|k_{2n-1}-k_{2n}|-1}e^{i\left(|k_1-k_{2}|+\dots+|k_{2n-1}-k_{2n}|-2(m_1+\dots+m_{2n-1})-2n\right)\eta/2}.
	\end{multline*}
	In particular following those calculations we can express the integrand of $I'$ in \eqref{def.I31} as follows
		\begin{multline*}
	\frac{\sin{((k_1+k_{2n})\eta/2)}}{\sin{(\eta/2)}}
\prod_{j=1}^{2n-1}\frac{\sin{((k_j-k_{j+1})\eta/2)}}{(k_j-k_{j+1})\sin{(\eta/2)}}
\\
	=\left(\sign(k_1+k_{2n})\prod_{j=1}^{2n-1}\frac{1}{|k_j-k_{j+1}|}\right)\sum_{\substack{m_j=0\\ 1 \le j \le 2n-1}}^{|k_j-k_{j+1}|-1} \sum_{m_{2n}=0}^{|k_1+k_{2n}|-1}e^{iB_1\eta/2},
	\end{multline*}
where to be clear in the sum the $m_j$ indicates a further summation over all $j\in \{1, \ldots, 2n-1\}$.  Also we define $B_1$ above as follows
	\begin{equation*}
	B_1=\sum_{j=1}^{2n-1}|k_j-k_{j+1}|+|k_1+k_{2n}|-2\sum_{j=1}^{2n} m_j-2n.
	\end{equation*}
Notice that no matter what is the sign of any of the terms inside the absolute values above that we always have $\sum_{j=1}^{2n-1}|k_j-k_{j+1}|+|k_1+k_{2N}| =2l$ for some integer $l$ so that $B_1$ is an even integer.  This holds because the sum contains two copies of every $k_1, \ldots, k_{2n}$. 

We further integrate as
$\int_{-\pi}^\pi e^{i B_1\eta/2} d\eta = \frac{4}{B_1}\sin(B_1\pi/2)$, and we notice that since $B_1$ is an even integer, then either $\frac{4}{B_1}\sin(B_1\pi/2)=0$ if $B_1\ne 0$ or $\frac{4}{B_1}\sin(B_1\pi/2)=2\pi$ when $B_1= 0$.
We can then find the following expression for $I'$ as
 	\begin{equation}\label{Sn.equality}
I'=
2\pi\left(\sign(k_1+k_{2n})\prod_{j=1}^{2n-1}\frac{1}{|k_j-k_{j+1}|}\right)\sum_{\substack{m_j=0\\ 1 \le j \le 2n-1}}^{|k_j-k_{j+1}|-1} \sum_{m_{2n}=0}^{|k_1+k_{2n}|-1}1_{B_1=0}.
	\end{equation}
This is our calculation of the integral \eqref{def.I31}.  We note that as a function of the single variable $m_{2n}$ then $B_1$ is decreasing and takes the value zero at most one time.  Thus $\sum_{m_{2n}=0}^{|k_1+k_{2n}|-1}1_{B_1=0} \le 1$.  We conclude that $\left| I' \right| \le 2\pi$.   Note that \eqref{Sn.Integral} is exactly \eqref{def.I31}.  So this proves the second estimate for \eqref{Sn.Integral}  in Lemma \ref{lemmaI}.

We now calculate the integral \eqref{def.I32}, which is rather similar. 
	We obtain
	\begin{multline*}
	\prod_{\substack{j=0}}^{2n-1}\frac{\sin{(|k_j-k_{j+1}|\eta/2)}}{\sin{(\eta/2)}}
\\
	=\sum_{\substack{m_j=0\\ 0 \le j \le 2n-1}}^{|k_j-k_{j+1}|-1}
	e^{i\left(|k_0-k_{1}|+\dots+|k_{2n-1}-k_{2n}|-2(m_0+\dots+m_{2n-1})-2n\right)\eta/2}.
	\end{multline*}
	Then we can express the integrand of $I^{\prime\prime}$ as follows
		\begin{multline*}
\cos{(\eta/2)}\frac{\sin{((k+k_{2n})\eta/2)}}{\sin{(\eta/2)}}
\prod_{j=0}^{2n-1}\frac{\sin{((k_j-k_{j+1})\eta/2)}}{(k_j-k_{j+1})\sin{(\eta/2)}}
\\
	=\frac12\left(\sign(k+k_{2n})\prod_{j=0}^{2n-1}\frac{1}{|k_j-k_{j+1}|}\right)\sum_{\substack{m_j=0\\ 0 \le j \le 2n-1}}^{|k_j-k_{j+1}|-1} \sum_{m_{2n}=0}^{|k+k_{2n}|-1}\left(e^{iB_2\eta/2}+e^{iB_3\eta/2}\right),
	\end{multline*}
We define $B_2$ as
	\begin{equation*}
	B_2=\sum_{j=0}^{2n-1}|k_j-k_{j+1}|+|k + k_{2n}|-2\sum_{j=0}^{2n} m_j-2n,
	\end{equation*}
and $B_3$ as
	\begin{equation*}
	B_3=\sum_{j=0}^{2n-1}|k_j-k_{j+1}|+|k + k_{2n}|-2\sum_{j=0}^{2n} m_j-2n-2,
	\end{equation*}	
Similarly since $\sum_{j=0}^{2n}|k_j-k_{j+1}|+|k + k_{2n}|$ contains two copies of every $k_0, k_1, \ldots, k_{2n}$ then it is always an even integer.  Therefore we conclude that $B_2$ and $B_3$ both are even integers.

We can then similarly find the following expression for $I^{\prime\prime}$ as
 	\begin{equation*}
I^{\prime\prime}
=
\pi\left(\sign(k\!+\!k_{2n})\!\prod_{j=0}^{2n-1}\frac{1}{|k_j-k_{j+1}|}\right)\sum_{\substack{m_j=0\\ 0 \le j \le 2n-1}}^{|k_j-k_{j+1}|-1} \sum_{m_{2n}=0}^{|k+k_{2n}|-1}\!\!\left(1_{B_2=0}\!+\!1_{B_3=0}\right).
	\end{equation*}
Then using the same argument as our upper bound estimate for $I'$ we obtain that $|I^{\prime\prime}| \le 2\pi$.  This completes the proof.
\end{proof}

\begin{remark}
We remark here that one can generally calculate the sum in \eqref{Sn.equality} exactly.  In particular the value of the sums in \eqref{Sn.equality} can be seen as the number of non-negative integer solutions to the equation  
	\begin{equation*}
	m_1+\cdots+m_{2n}
=\frac{1}{2}\sum_{j=1}^{2n-1}|k_j-k_{j+1}|+|k_1+k_{2n}|-n,
	\end{equation*}
with the restrictions that 	
$0\le m_j \le |k_j-k_{j+1}|-1$ for $j=1, \ldots, 2n-1$ and  $0\le m_{2n} \le |k_1+k_{2n}|-1$.  This value can be calculated exactly using the inclusion-exclusion formula. 

Alternatively, if $n=1$ in \eqref{Sn.Integral} then one can calculate, on the region where $I'_1 \ne 0$, that we have exactly
 	\begin{equation}\notag
I'_1=
2\pi
\frac{\min\{|k_1-k_{2}|, |k_1+k_{2}|\}\sign(k_1+k_{2})}{|k_1-k_{2}|} .
	\end{equation}
And this formula is consistent with our estimate in Lemma \ref{lemmaI}.
\end{remark}

\section{Proof of main theorem}\label{secMain}

This section is devoted to the proof of  Theorem \ref{MainTheorem}. In Subsection \ref{existence} we show the scheme of the proof for existence of solutions via a regularization argument. The main part consists in obtaining the \textit{a priori} estimates, in particular the energy inequality from \eqref{balanceX}. Uniqueness is later proved in Subsection \ref{uniqueness}.

\subsection{Existence}\label{existence}

The proof follows a standard regularization argument. We will use a regularization of the equations \eqref{Xteqn} and \eqref{viscosityjump}, written in the form of \eqref{systemfinal}, and the \textit{a priori} estimates of the previous section to find a weak solution in the sense of Definition \ref{weaksol} below. The regularity obtained for the solution will imply that the solution found is indeed a strong solution, which we prove later is unique.
\begin{defn}\label{Fsolution}
For fixed $t\in[0,T]$ and $\bm{\phi}(t)\in W^{2,\infty}(\mathbb{S})$, we say that $\bm{\psi}(t)\in L^\infty(\mathbb{S})$ is a weak solution of 
 \begin{equation*}
\begin{aligned}
\bm{\psi}(\theta,t)+2A_\mu\partial_\theta \bm{\phi}(\theta,t)^\perp\cdot\int_{\mathbb{S}}  \mathcal{T}(\bm{\phi}(\theta,t)-\bm{\phi}(\eta,t))\cdot \bm{\psi}(\eta,t)d\eta
=2A_e\p_\theta^2\bm{\phi}(\theta,t),
\end{aligned}
\end{equation*}
with $\mathcal{T}$ given by \eqref{Tijk}, if for any $\bm{\varphi}\in \mathcal{D}(\mathbb{S})$ it holds that
\begin{equation*}
    \begin{aligned}
\int_{\mathbb{S}}\bm{\psi}(\theta,t)\cdot \bm{\varphi}(\theta)d\theta+2A_\mu\int_{\mathbb{S}}\varphi_i(\theta)\int_{\mathbb{S}}\partial_\theta \phi_j(\theta,t)^\perp  \mathcal{T}_{ijk}(\bm{\phi}(\theta,t)\!-\!\bm{\phi}(\eta,t)) \psi_k(\eta,t)d\eta \hspace{0.05cm}d\theta
\\
=2A_e\int_{\mathbb{S}}\p_\theta^2\bm{\phi}(\theta,t)\cdot \bm{\varphi}(\theta).
\end{aligned}
\end{equation*}
\end{defn}
\begin{defn}\label{weaksol}
We say that $\bm{\mathcal{X}}\in L^\infty([0,T];W^{1,\infty}(\mathbb{S}))\cap L^1([0,T];W^{2,\infty}(\mathbb{S}))$ is a weak solution of \eqref{Xteqn} if for almost every $t\in[0,T]$ the arc-chord condition \eqref{arcchord} is satisfied, 
and if for any $\bm{\varphi}\in \mathcal{D}(\mathbb{S}\times[0,T])$ it holds that
\begin{equation*}
    \begin{aligned}
       \int_{\mathbb{S}}\bm{\mathcal{X}}(\theta,t)\cdot\bm{\varphi}(\theta,t)d\theta-        \int_{\mathbb{S}}\bm{\mathcal{X}}_0(\theta)\cdot\bm{\varphi}(\theta,0)d\theta-        \int_0^t\int_{\mathbb{S}}\bm{\mathcal{X}}(\theta,\tau)\cdot\partial_t\bm{\varphi}(\theta,\tau)d\theta \hspace{0.05cm}d\tau\\
       =\int_0^t\int_{\mathbb{S}}
        \bm{\varphi}(\theta,\tau)\cdot\int_{\mathbb{S}} G(\bm{\mathcal{X}}(\theta,\tau)-\bm{\mathcal{X}}(\eta,\tau))\bm{F}(\eta,\tau)\hspace{0.05cm}d\eta \hspace{0.05cm}d\theta d\tau, 
    \end{aligned}
\end{equation*}
where $G$ is defined in \eqref{G} and $\bm{F}\in L^1([0,T];L^\infty(\mathbb{S}))$ is the solution in the sense of Definition \ref{Fsolution} of \eqref{viscosityjump}.
\end{defn}

We will denote $f_M=\mathcal{J}_M f$ for general $f$ such as $f=\bm{X}$, $f=\bm{X}_c$,  $f=\bm{Y}$ or $f=\bm{F}$, with $\mathcal{J}_M$ the high frequency cut-off defined in \eqref{CutOffHigh}. We start by considering a regularized version of system \eqref{Xteqn}, \eqref{viscosityjump} (where \eqref{Xteqn} is written in \eqref{finalsystem} with the linear and nonlinear terms apart).
For each positive integer $M$,  consider the regularized initial data $\bm{\mathcal{X}}_{0,M}$ and the corresponding solution $\bm{\mathcal{X}}=\bm{X}_M+\bm{X}_{M,c}$ to the regularized system 
\begin{equation}\label{FM}
\begin{aligned}
    \partial_t\bm{\mathcal{X}}_M&=-\frac{A_e}2\big(\Lambda \bm{X}_M+\mathcal{H}\mc{R}^{-1} \bm{X}_M\big)+\mathcal{J}_M\bm{\mathcal{N}}(\bm{X}_{M,c},\bm{X}_{M}),\\
    \bm{F}_M&=2A_\mu \mathcal{J}_M\bm{\mathcal{S}}(\bm{F}_M,\bm{\mathcal{X}}_M)+2A_e\partial_\theta^2\bm{\mathcal{X}}_{M}.
\end{aligned}
\end{equation}
We define correspondingly 
$\bm{\mathcal{Y}}_{0,M}$ and $\bm{\mathcal{Y}}_M=\bm{Y}_M+\bm{Y}_{M,c}$. We recall that  \eqref{finalsystem} could be written in $\bm{\mathcal{Y}}$ variables as \eqref{systemfinal}.
The corresponding regularized system in these variables reads as follows:
\begin{equation}\label{approxsystem}
\begin{aligned}
    \widehat{\bm{Y}}_M(0)&=0,\hspace{0.3cm}\widehat{Y}_{M,2}(1)=0,\hspace{0.3cm} \widehat{\bm{Y}}_{M,c}(k)=0\hspace{0.2cm} k\neq 0,1,\hspace{0.3cm}\widehat{Y}_{M,c,1}(1)=0,\\
    \partial_t\widehat{\bm{Y}}_{M,c}(0)&=
    P(0)^{-1}\widehat{\bm{\mathcal{N}}(\bm{X}_{M,c},\bm{X}_M)}(0),\\
    \partial_t\widehat{Y}_{M,1}(1)&=-A_e \widehat{Y}_{M,1}(1)+\Big(P(1)^{-1}\widehat{\bm{\mathcal{N}}(\bm{X}_{M,c},\bm{X}_M)}(1)\Big)_1,\\
    \partial_t\widehat{\bm{Y}}_M(k)&=-\frac{A_e}2\mathcal{D}(k)\widehat{\bm{Y}}_M(k)+P(k)^{-1}\widehat{\bm{\mathcal{N}}(\bm{X}_{M,c},\bm{X}_M)}(k),\quad 2\leq k\leq M,\\
    |\widehat{Y}_{M,c,2}(1)|^2&=\frac12-\sum_{1\leq k\leq M}k\big(|\widehat{Y}_{M,2}(k)|^2-|\widehat{Y}_{M,1}(k)|^2\big),
    \end{aligned}
\end{equation} 
with $\bm{F}_M$ given by \eqref{FM}.
Since $\bm{X}_{M,c}$ is a circle with radius satisfying \eqref{radiusbound}, the arch-chord condition \eqref{arcchord} is clearly satisfied under the condition $\|\bm{X}_M\|_{\foneonenu}<k(A_\mu)$. Then, with the same size condition, $\bm{F}_M$ is estimated in terms of $\bm{X}_M$ as in Section \ref{secF}. Thus, with $\bm{F}_M$ solved in terms of $\bm{Y}_M$ using the transformation \eqref{Y},
 we obtain an ODE of the form
\begin{equation*}
    \dot{\bm{Y}}_M=\mathcal{J}_M \mathcal{G}\big(\bm{Y}_M\big),\qquad \bm{Y}_M(0)=\bm{Y}_{M,0},
\end{equation*}
for a certain nonlinear function $\mathcal{G}$. 
Notice that the ODE for $\widehat{\bm{Y}}_{M,c}(0)$ is decoupled from the rest because there are no zero modes in  $\widehat{\bm{\mathcal{N}}(\bm{X}_{M,c},\bm{X}_M)}(0)$. Therefore, Picard's theorem on Banach spaces yields the local existence of regularized solutions $\bm{Y}_M\in C^1([0,T_M);H_M^m)$, where $H_M^m=\{f\in H^m(\mathbb{S}): \text{supp}(\widehat{f})\subset [-M,M]\}.$ Since the \textit{a priori} energy estimate \eqref{balanceY} holds for the regularized system, we have uniform bounds for $\bm{Y}_M$ in the space $L^\infty(\mathbb{R}_+; \dot{\mathcal{F}}^{1,1}_\nu)\cap L^1(\mathbb{R}_+; \dot{\mathcal{F}}^{2,1}_\nu)$.
It is not hard to prove that $\bm{Y}_M$ forms a Cauchy sequence in $L^\infty([0,T];\mathcal{F}^{0,1}_\nu)$, so that we have a candidate for solution. One can then apply a version of the Aubin-Lions lemma (see Corollary 6 of \cite{MR916688}) to get the strong convergence, up to a subsequence, of the approximate problems in $L^2([0,T];\dot{\mathcal{F}}^{1,1}_\nu)$. Next, since $\widehat{\bm{Y}}_M(m,t)\rightarrow \widehat{\bm{Y}}(m,t)$ as $M\to\infty$, $\forall m\in\mathbb{Z}$ and almost every $t$, Fatou's lemma allows us to conclude that
\begin{equation*}
    \begin{aligned}
        M(t)=& \|\bm{Y}\|_{\foneonenu}(t)+\frac{A_e}{4}\dissconst \int_0^t\|\bm{Y}\|_{\dot{\mathcal{F}}^{2,1}_\nu}(\tau)d\tau\\
        &\leq \liminf_{M\to+\infty}\Big(\|\bm{Y}_M\|_{\foneonenu}(t)+\frac{A_e}{4}\dissconst \int_0^t\|\bm{Y}_M\|_{\dot{\mathcal{F}}^{2,1}_\nu}(\tau)d\tau\Big)\\
        &\leq \|\bm{Y}_0\|_{\foneone},
    \end{aligned}
\end{equation*}
so we obtain that the limit function $\bm{Y}$ belongs to $L^\infty([0,T];\dot{\mathcal{F}}^{1,1}_\nu)\cap L^1([0,T];\dot{\mathcal{F}}^{2,1}_\nu)$. 
Therefore we have  obtained the strong convergence, up to a subsequence, of $\bm{X}_M\to \bm{X}$ in 
$L^\infty([0,T];\dot{\mathcal{F}}^{1,1}_\nu)\cap L^1([0,T];\dot{\mathcal{F}}^{2,1}_\nu)$, which immediately implies from \eqref{FM}, under the size constraint \eqref{conditiontheorem}, the strong convergence $\bm{F}_M\to \bm{F}$ in $L^1([0,T];\fzeronenu)$. In fact, it suffices to consider $\bm{F}_{M_1}$ and $\bm{F}_{M_2}$, write their difference as
\begin{equation*}    
\begin{aligned}
\bm{F}_{M_1}-\bm{F}_{M_2}&=2A_\mu \big(\mathcal{J}_{M_1}\bm{\mathcal{S}}(\bm{F}_{M_1},\bm{\mathcal{X}}_{M_1})- \mathcal{J}_{M_2}\bm{\mathcal{S}}(\bm{F}_{M_1},\bm{\mathcal{X}}_{M_1})\big)\\
&\quad+2A_\mu \big(\mathcal{J}_{M_2}\bm{\mathcal{S}}(\bm{F}_{M_1},\bm{\mathcal{X}}_{M_1})- \mathcal{J}_{M_2}\bm{\mathcal{S}}(\bm{F}_{M_2},\bm{\mathcal{X}}_{M_1})\big)\\
&\quad+2A_\mu \big(\mathcal{J}_{M_2}\bm{\mathcal{S}}(\bm{F}_{M_2},\bm{\mathcal{X}}_{M_1})- \mathcal{J}_{M_2}\bm{\mathcal{S}}(\bm{F}_{M_2},\bm{\mathcal{X}}_{M_2})\big)\\
&\quad+2A_e(\partial_\theta^2\bm{\mathcal{X}}_{M_1}-\partial_\theta^2\bm{\mathcal{X}}_{M_2}),
\end{aligned}
\end{equation*}
and perform similar estimates to the ones in Subsection \ref{secF} to find that $\bm{F}_M$ forms a Cauchy sequence in $L^1([0,T]; \fzeronenu)$.  
Since $\bm{X}_{M,c}$ is given in terms of $\bm{X}_M$, the above convergence holds for $\bm{\mathcal{X}}_M$. The strong convergence $\bm{\mathcal{X}}_M\to \bm{\mathcal{X}}$ in $L^\infty([0,T];\dot{\mathcal{F}}^{1,1}_\nu)$ together with $\bm{F}_M\to \bm{F}$ in $L^1([0,T];\fzeronenu)$
yields $\bm{\mathcal{X}}$ as a solution to \eqref{Xteqn} in the sense of Definition \ref{weaksol}.   (Moreover, it is easy to check in \eqref{FM} the strong convergence of the right side terms in $L^1([0,T];\dot{\mathcal{F}}^{1,1}_\nu)$.)

We refer to Section 5 of \cite{GG-BS2019} for a  similar approximation argument, including the instant generation of analyticity and the continuity in time. We include it here for completeness. From the strong convergence in $L^1([0,T]; \foneonenu)$ of the right hand side of \eqref{Xteqn}, we must have that $\partial_t\bm{\mathcal{X}}_{M}\to \partial_t\bm{\mathcal{X}}$ in $L^1([0,T]; \foneonenu)$.
Consider $0<t\leq t_1< t_2$. Then, 
\begin{equation*}
    \|\bm{\mathcal{X}}(t_2)-\bm{\mathcal{X}}(t_1)\|_{\dot{\mathcal{F}}^{1,1}_{\nu(t)}}=\|\int_{t_1}^{t_2}\partial_t\bm{\mathcal{X}}(\tau)d\tau\|_{\dot{\mathcal{F}}^{1,1}_{\nu(t)}}\leq \int_{t_1}^{t_2} \|\partial_t\bm{\mathcal{X}}(\tau)\|_{\dot{\mathcal{F}}^{1,1}_{\nu(\tau)}}d\tau,
\end{equation*}
which from the fact that $\partial_t\bm{\mathcal{X}}\in L^1([0,T]; \foneonenu)$
yields that the solution is analytic for all positive times, and $\bm{\mathcal{X}}\in C([\varepsilon,T]; \foneonenu)$ for any $\varepsilon>0$.
Moreover, fix $\tilde{\nu}_m\in(0,\nu_m)$  and denote $\tilde{\nu}(t)$ according to \eqref{nu}, now given any $t_2>0$ choose $0<t_1<t_2$ close enough to $t_2$ that $\tilde{\nu}(t_2)<\nu(t_1)$. Thus, it holds that $\|\bm{\mathcal{X}}(t_2)-\bm{\mathcal{X}}(t_1)\|_{\dot{\mathcal{F}}^{1,1}_{\tilde{\nu}(t_2)}}\to 0$ as $t_1\to t_2$, and therefore we have that $\bm{\mathcal{X}}\in C([0,T];\dot{\mathcal{F}}^{1,1}_{\tilde{\nu}})$. Since $\tilde{\nu}_m\in(0,\nu_m)$ is an arbitrary number in an open interval, we conclude that 
$\bm{\mathcal{X}}\in C([0,T];\dot{\mathcal{F}}^{1,1}_{\nu})$. 
Finally, the analyticity in space for all positive times implies that $\bm{\mathcal{X}}\in C([\varepsilon,T];\dot{\mathcal{F}}^{s,1}_{\tilde{\nu}})$ for any $s\geq0$,  $\varepsilon>0$, and $0<\tilde{\nu}<\nu$. This regularity translates to $\bm{F}$ as well for $t\geq\varepsilon$. Therefore, one can consider $\partial_t\bm{\mathcal{X}}(t_2)-\partial_t\bm{\mathcal{X}}(t_1)$ for arbitrary $t_2, t_1\geq\varepsilon$ to find in particular that $\partial_t\bm{\mathcal{X}}\in C((0,T];\mathcal{F}^{0,1}_\nu)$.

We have proven that $\bm{\mathcal{X}}$ is a strong solution in the sense of Definition \ref{strongsol} as claimed in Theorem \ref{MainTheorem}. In Subsection \ref{uniqueness} we prove that this solution is unique. \\

We now prove the global in time energy inequality in \eqref{balanceX}.

\vspace{0.1cm}
\begin{proof}[Proof of \eqref{balanceX}]
Equations \eqref{eqs} show decay of the higher frequencies \eqref{diff} if we are able to control the nonlinear terms, for which we will need the constraint \eqref{radius}.
Indeed, using \eqref{diff} and the inequality $k(k-1) \ge \frac{k^2}{2}$ for $k\ge 2$ 
implies that
\begin{equation}\label{balance}
\begin{aligned}
\frac{d}{dt}\|\bm{Y}\|_{\foneonenu}&\leq 
-\Big(\frac{A_e}4-\nu'(t)\Big)  \|\bm{Y}\|_{\ftwoonenu}+\|\bm{\mathcal{N}}(\bm{X}_c,\bm{X})\|_{\foneonenu},
\end{aligned}
\end{equation}
where we have used that $\|P(k)^{-1}\|=1$, and we can choose $\nu'(t)$ as small as we need.
The goal is thus to obtain a bound like the following
\begin{equation}\label{nonlinearterm}
\|\bm{\mathcal{N}}(\bm{X}_c,\bm{X})\|_{\foneonenu}\leq C(\|\bm{X}\|_{\foneonenu})\|\bm{Y}\|_{\ftwoonenu},    
\end{equation}
with $C(\|\bm{X}\|_{\foneonenu})\approx\|\bm{X}\|_{\foneonenu}$.   

We proceed to complete the nonlinear estimate \eqref{nonlinearterm} to obtain the adequate sign in the balance \eqref{balance}.
We insert the \textit{a priori} bounds on $\bm{F}$ given by \eqref{F0bound}, \eqref{F1bound}, and \eqref{FNbound}, into the estimate \eqref{Nestimate} to obtain that
\begin{equation*}
\begin{aligned}
\|\bm{\mathcal{N}}\|_{\foneonenu}&\leq 22\sqrt{2}A_e\frac{1-A_\mu+|A_\mu|}{1-A_\mu}D_1\|\bm{X}\|_{\foneonenu}\xtwoonenu\\
&\quad+ 147\sqrt{2}\frac{A_e}{1-A_\mu}e^{\nu_m }D_2C_8\|\bm{X}\|_{\foneonenu}\|\bm{X}\|_{\ftwoonenu}\\
&\quad+2250\sqrt{2}A_e \frac{|A_\mu|(1+|A_\mu|)}{(1-A_\mu)^2(1+A_\mu)}D_3D_4\xoneonenu\xtwoonenu,
\end{aligned}
\end{equation*}
which finally gives the desired estimate
\begin{equation}\label{nonlinearfinal}
\begin{aligned}
\|\bm{\mathcal{N}}\|_{\foneonenu}&\leq 169\sqrt{2}\frac{A_e}{1-A_\mu} D_5\xoneonenu\xtwoonenu,
\end{aligned}
\end{equation}
where
\begin{multline}\label{D5}
    D_5=
    \\
    \frac{1}{169}\Big(22(1-A_\mu+|A_\mu|)D_1+147e^{\nu_m}D_2C_8\!+\!2250\frac{|A_\mu|(1\!+\!|A_\mu|)}{(1\!-\!A_\mu)(1\!+\!A_\mu)}D_3D_4\Big),
\end{multline}
and $C_8$, $D_1$, $D_2$, $D_3$, $D_4$, are given by \eqref{C8}, \eqref{D1D2D3}, and \eqref{D4}.
Recalling the equivalence \eqref{equivalence} and inserting the above bound into \eqref{balance}, we obtain 
\begin{equation}\label{balance2}
\frac{d}{dt}\|\bm{Y}\|_{\foneonenu}
\leq 
-A_e\Big(\frac{1}4-\frac{\nu'(t)}{A_e}-169\sqrt{2}\frac{D_5}{1-A_\mu}\xoneonenu\Big)\|\bm{Y}\|_{\ftwoonenu}.
\end{equation}
Since $\nu'(t)=\frac{\nu_m}{(1+t)^2}$ and $\nu_m>0$ in \eqref{nu} can be chosen arbitrarily small, if the condition
\begin{equation}\label{condition}
1-676\sqrt{2}\frac{D_5(\xoneonenu)}{1-A_\mu}\xoneonenu>0
\end{equation}
holds initially, where $D_5$ is defined in \eqref{D5}, then the fact that $D_{5}$ decreases as $\xoneonenu$ decreases guarantees that this condition is propagated in time.
For the same reasons, this condition can be stated as a smallness condition for $\|\bm{X}_0\|_{\foneone}$ as follows
\begin{equation}\label{kdef}
    \|\bm{X}_0\|_{\foneone}< k(A_\mu),
\end{equation}
with $k$ a function defined implicitly via \eqref{condition} (see also Figure \ref{fig:pic}). Because $D_5$ is increasing on $\xoneonenu$, we have the following explicit lower bound
\begin{equation}\label{kmuform}
    \begin{aligned}
        &k(A_\mu)>\frac{1\!-\!A_\mu}{676\sqrt{2} D_5(0)}\\
        =&\bigg(\frac{588\sqrt{2}}{1\!-\!A_\mu}\!+\!88\sqrt{2}\Big(1\!+\!\frac{|A_\mu|}{1\!-\!A_\mu}\Big)\!+\!\frac{9\sqrt{2}|A_\mu|(1\!+\!|A_\mu|)}{(1\!-\!A_\mu)(1\!+\!A_\mu)}\Big(112\big(1\!+\!\frac{|A_\mu|}{1\!-\!A_\mu}\big)\!+\!\frac{888}{1\!-\!A_\mu}\Big)\!\bigg)^{-1}\!\!\!,
    \end{aligned}
\end{equation}
which for small enough $\xoneonenu$ approximates the actual value of $k(A_\mu)$.
Therefore,
\begin{equation}\label{balanceY}
    \begin{aligned}
    \|\bm{Y}\|_{\foneonenu}(t)+\frac{A_e}4 \dissconst \int_0^t \|\bm{Y}\|_{\ftwoonenu}(\tau)d\tau\leq \|\bm{Y}_0\|_{\foneone},
    \end{aligned}
\end{equation}
with
\begin{equation}\label{C}
    \dissconst=\dissconst\big(\|\bm{X}_0\|_{\foneonenu},A_\mu,\nu_m\big)=1-4\frac{\nu'(t)}{A_e}-676\sqrt{2}\frac{D_5(\|\bm{X}_0\|_{\foneonenu})}{1-A_\mu}\|\bm{X}_0\|_{\foneonenu}.
\end{equation}
Moreover, since $\|\bm{Y}\|_{\foneonenu}\leq \|\bm{Y}\|_{\ftwoonenu}$, the inequality \eqref{balance2} gives that
\begin{equation*}
    \begin{aligned}
    \frac{d}{dt}\|\bm{Y}\|_{\foneonenu}&\leq 
    -\frac{A_e}4 \dissconst \|\bm{Y}\|_{\foneonenu},
\end{aligned}
\end{equation*}
and thus 
\begin{equation}\label{decay}
\begin{aligned}
\|\bm{Y}\|_{\foneonenu}&\leq \|\bm{Y}_0\|_{\foneonenu}e^{-\frac{A_e}4 \dissconst t}.
\end{aligned}
\end{equation}
This completes the decay estimate.

The control of the zero frequency follows from \eqref{zerofreq} with
\begin{equation}\label{zeroaux}
    \begin{aligned}
     |\widehat{\bm{X}_c}(0)|\leq |\widehat{\bm{X}}_{0,c}(0)|+\int_0^t |\widehat{\bm{\mathcal{N}}(\bm{X}_c,\bm{X})}(0)|d\tau.
    \end{aligned}
\end{equation}
Notice that the estimates of the nonlinear terms in $\mathcal{F}^{0,1}$ can be done as in Section \ref{secanalytic}
and yield the bound
\begin{equation*}
    \begin{aligned}
    |\widehat{\bm{\mathcal{N}}(\bm{X}_c,\bm{X})}(0)|\leq \|\bm{\mathcal{N}}(\bm{X}_c,\bm{X})\|_{\fzerone}\leq  A_e\frac{\tilde{D}_5}{1-A_\mu} \xoneonenu\xtwoonenu,
    \end{aligned}
\end{equation*}
where $\tilde{D}_5$ is a constant that plays the role of $D_5$. 
Recalling \eqref{equivalence} and the energy balance \eqref{balanceY}, 
we introduce this bound back to \eqref{zeroaux} to conclude
\begin{equation*}
    \begin{aligned}
     |\widehat{\bm{X}_c}(0)|\leq |\widehat{\bm{X}}_{0,c}(0)|+\tilde{C}\|\bm{X}_0\|_{\dot{\mathcal{F}}^{1,1}}^2,
    \end{aligned}
\end{equation*}
with 
\begin{equation}\label{Ctilde}
    \tilde{C}=\frac{\tilde{D}_5}{(1-A_\mu)\dissconst},
\end{equation}
and $D_5$, $\dissconst$ given in \eqref{D5} and \eqref{C} respectively.

Finally, the  decay \eqref{decay} applied to \eqref{freq12} yields that
\begin{equation*}
    \begin{aligned}
     \frac{R(t)^2}{2}=|\widehat{Y}_{c,2}(1)|^2\to \frac12 \text{ as } t\to +\infty,
\end{aligned}
\end{equation*}
showing the exponentially fast convergence to a uniformly parametrized circle of area $\pi$.
\end{proof}

\subsection{Uniqueness}\label{uniqueness}

Consider two solutions $\bm{\mathcal{X}}=\bm{X}_c+\bm{X}$ and $\tilde{\bm{\mathcal{X}}}=\tilde{\bm{X}_c}+\tilde{\bm{X}}$ with initial data $\bm{\mathcal{X}}_0$ and $\tilde{\bm{\mathcal{X}}}_0$ in $\mathcal{F}^{1,1}$. Recalling the system in the $\bm{\mathcal{Y}}$ variables \eqref{evolutionFourierY}, we have that
\begin{equation}\label{uniq}
\begin{aligned}
\frac{d}{dt}\|\bm{\mathcal{Y}}-\widetilde{\bm{\mathcal{Y}}}\|_{\dot{\mathcal{F}}^{1,1}_\nu}
&\leq 
-\Big(\frac{A_e}4-\nu'(t)\Big)  \|\bm{Y}-\widetilde{\bm{Y}}\|_{\ftwoonenu}\\&\quad+\sqrt{2}\|\bm{\mathcal{N}}(\bm{X}_c,\bm{X})-\bm{\mathcal{N}}(\tilde{\bm{X}}_c,\tilde{\bm{X}})\|_{\dot{\mathcal{F}}^{1,1}_\nu},
\end{aligned}
\end{equation}
and
\begin{equation}\label{uniqcenter}
    \begin{aligned}
        |\widehat{\bm{Y}}_c(0)-\widehat{\widetilde{\bm{Y}}}_c(0)|&\leq |\widehat{\bm{Y}}_{0,c}(0)-\widehat{\widetilde{\bm{Y}}}_{0,c}(0)|\\
        &\quad+\int_0^t \Big|P(0)^{-1} \widehat{\bm{\mathcal{N}}(\bm{X}_c,\bm{X})}(0)-P(0)^{-1}\widehat{\bm{\mathcal{N}}(\tilde{\bm{X}}_c,\tilde{\bm{X}})}(0)  \Big|d\tau.
    \end{aligned}
\end{equation}
Notice that, in comparison with \eqref{balance}, we are including in the left-hand side of \eqref{uniq} the terms corresponding to (the first frequency of) the circle part,
$$2|\widehat{Y}_{c,2}(1)-\widehat{\tilde{Y}}_{c,2}(1)|.$$
Although these terms are neutral with respect to the dissipative linear operator, whenever they appear on the right-hand side we will be able to absorb them by using Gr\"onwall's lemma and \eqref{balanceY} (which both $\bm{Y}$ and $\tilde{\bm{Y}}$ satisfy).  
Notice further that since the nonlinear terms do not contain the zero frequency of $\bm{\mathcal{Y}}$, i.e., $\widehat{\bm{Y}}_c(0)$, equation \eqref{uniqcenter}  implies that
\begin{equation}\label{uniqcenter2}
     |\widehat{\bm{Y}}_c(0)-\widehat{\widetilde{\bm{Y}}}_c(0)|=0
\end{equation}
once we show from \eqref{uniq} that $\|\bm{\mathcal{Y}}-\widetilde{\bm{\mathcal{Y}}}\|_{\dot{\mathcal{F}}^{1,1}_\nu}=0$. Thus we proceed to deal with \eqref{uniq}.

The difference between the nonlinear terms in \eqref{uniq} is split in four, accordingly to \eqref{Ni}, so that we have
\begin{multline}\label{uniqnonlinear}
    \|\bm{\mathcal{N}}(\bm{X}_c,\bm{X})-\bm{\mathcal{N}}(\tilde{\bm{X}_c},\tilde{\bm{X}})\|_{\foneonenu}\leq \|\bm{\mathcal{N}}_1(\bm{X}_c,\bm{X})-\bm{\mathcal{N}}_1(\tilde{\bm{X}_c},\tilde{\bm{X}})\|_{\mathcal{F}^{0,1}_\nu}
    \\
    +\|\bm{\mathcal{N}}_2(\bm{X}_c,\bm{X})\!-\!\bm{\mathcal{N}}_2(\tilde{\bm{X}_c},\tilde{\bm{X}})\|_{\mathcal{F}^{0,1}_\nu}\!+\!\|\bm{\mathcal{N}}_3(\bm{X}_c,\bm{X})\!-\!\bm{\mathcal{N}}_3(\tilde{\bm{X}_c},\tilde{\bm{X}})\|_{\mathcal{F}^{0,1}_\nu}
    \\
    +\|\bm{\mathcal{N}}_4(\bm{X}_c,\bm{X})-\bm{\mathcal{N}}_4(\tilde{\bm{X}_c},\tilde{\bm{X}})\|_{\mathcal{F}^{0,1}_\nu}.
\end{multline}
We start by explaining the estimate corresponding to the first sub-term $\bm{\mathcal{N}}_{1,1}$ in detail (see \eqref{N1split}), and later we will explain the general procedure. We have 
\begin{equation}\label{N1diff}
\begin{aligned}
\bm{\mathcal{N}}_{1,1}(\bm{X}_c,\bm{X})(\theta)-\bm{\mathcal{N}}_{1,1}(\tilde{\bm{X}_c},\tilde{\bm{X}})(\theta)= \bm{Q}_1+\bm{Q}_2+\bm{Q}_3+\bm{Q}_4,
\end{aligned}
\end{equation}
where
\begin{equation*}
    \begin{aligned}
\bm{Q}_1&=\Big(\frac{-1}{4\pi R^2}\!+\!\frac{1}{4\pi \tilde{R}^2}\Big)\!\!\!\int_\mathbb{S}\!\!\p_\theta\Delta_\eta \bm{X}_c(\theta)^T\!\Delta_\eta \bm{X}(\theta)\bm{F}_L(\eta)d\eta,
\\
\bm{Q}_2&=\frac{1}{4\pi \tilde{R}^2}\int_\mathbb{S}\Big(\p_\theta\Delta_\eta \tilde{\bm{X}_c}(\theta)-\p_\theta\Delta_\eta \bm{X}_c(\theta)\Big)^T\Delta_\eta \tilde{\bm{X}}(\theta)\tilde{\bm{F}}_L(\eta)d\eta,
\\
\bm{Q}_3&=\frac{1}{4\pi \tilde{R}^2}\int_\mathbb{S}\p_\theta\Delta_\eta \bm{X}_c(\theta)^T\Big(\Delta_\eta \tilde{\bm{X}}(\theta)-\Delta_\eta \bm{X}(\theta)\Big)\tilde{\bm{F}}_L(\eta)d\eta,
\\
\bm{Q}_4&=\frac{1}{4\pi \tilde{R}^2}\int_\mathbb{S}\p_\theta\Delta_\eta \bm{X}_c(\theta)^T\Delta_\eta \bm{X}(\theta)\Big(\tilde{\bm{F}}_L(\eta)-\bm{F}_L(\eta)\Big)d\eta.
\end{aligned}
\end{equation*}
For the first term, we need to estimate the difference between $R$ and $\tilde{R}$. Recalling \eqref{radius}, where $|\widehat{Y}_{c,2}(1)|^2=\frac{R^2}2$ with $R^2 = a^2+b^2$, we have
\begin{equation*}
    \begin{aligned}
    |R^2-\tilde{R}^2|&=\Big|-2\sum_{k\geq1}k\big(|\widehat{Y}_2(k)|^2-|\widehat{Y}_1(k)|^2\big)+2\sum_{k\geq1}k\big(|\widehat{\tilde{Y}}_2(k)|^2-|\widehat{\tilde{Y}}_1(k)|^2\big)\Big|\\
    &\leq 2\sum_{k\geq1}k\Big( \big| |\widehat{Y}_2(k)|-|\widehat{\tilde{Y}}_2(k)|\big|\big(|\widehat{Y}_2(k)|+|\widehat{\tilde{Y}}_2(k)|\big)\\
    &\quad+\big| |\widehat{Y}_1(k)|-|\widehat{\tilde{Y}}_1(k)|\big|\big(|\widehat{Y}_1(k)|+|\widehat{\tilde{Y}}_1(k)|\big)    \Big).
    \end{aligned}
\end{equation*}
Further note for $j=1,2$ that
$$
\big| |\widehat{Y}_j(k)|-|\widehat{\tilde{Y}}_j(k)|\big|
\le 
\big| \widehat{Y}_j(k)-\widehat{\tilde{Y}}_j(k)\big|
$$
And on the $\mathbb{S}$ domain we have $|\widehat{Y}_j(k)| \le \|Y_j\|_{L^\infty(\mathbb{S})} \le \|Y_j\|_{\dot{\mathcal{F}}^{0,1}}$.  We conclude 
\begin{equation*}
    |R^2-\tilde{R}^2|
    \leq
    \left(\|\bm{Y}\|_{\dot{\mathcal{F}}^{0,1}}+\|\tilde{\bm{Y}}\|_{\dot{\mathcal{F}}^{0,1}}\right)
    \|\bm{Y}-\tilde{\bm{Y}}\|_{\dot{\mathcal{F}}^{1,1}}
\end{equation*}
Therefore, using also \eqref{radius.lower.bd}, we obtain
\begin{equation}\notag
\Big|\frac1{R^2}-\frac1{\tilde{R}^2}\Big|
=
\Big|\frac{\tilde{R}^2-R^2}{R^2\tilde{R}^2}\Big|
\leq
\frac{\|\bm{Y}\|_{\dot{\mathcal{F}}^{0,1}}+\|\tilde{\bm{Y}}\|_{\dot{\mathcal{F}}^{0,1}}}{\sqrt{1-\frac12\|\bm{Y}\|_{\dot{\mathcal{F}}^{\frac12,1}}}\sqrt{1-\frac12\|\tilde{\bm{Y}}\|_{\dot{\mathcal{F}}^{\frac12,1}}}}\|\bm{Y}-\tilde{\bm{Y}}\|_{\dot{\mathcal{F}}^{1,1}}.
\end{equation}
In particular, for a constant $c\big(\|\bm{Y}\|_{\dot{\mathcal{F}}^{1,1}},\|\tilde{\bm{Y}}\|_{\dot{\mathcal{F}}^{1,1}}\big)>0$, we can write
\begin{equation*}
    \Big|\frac1{R^2}-\frac1{\tilde{R}^2}\Big|\leq c\big(\|\bm{Y}\|_{\dot{\mathcal{F}}^{1,1}},\|\tilde{\bm{Y}}\|_{\dot{\mathcal{F}}^{1,1}}\big)\|\bm{Y}-\tilde{\bm{Y}}\|_{\dot{\mathcal{F}}^{1,1}}.
\end{equation*}
Then, the bound for $\bm{Q}_1$ follows as in the estimate for the term \eqref{N11}, we obtain
\begin{equation*}
    \begin{aligned}
    \|\bm{Q}_1\|_{\fzeronenu}\leq c(\|\bm{Y}\|_{\foneonenu},\|\tilde{\bm{Y}}\|_{\foneonenu})\|\bm{F}_L\|_{\fzeronenu}\|\bm{Y}-\tilde{\bm{Y}}\|_{\dot{\mathcal{F}}^{1,1}},
    \end{aligned}
\end{equation*}
and introducing the estimate for $\bm{F}_L$ from \eqref{F1bound} we have
\begin{equation*}
    \begin{aligned}
    \|\bm{Q}_1\|_{\fzeronenu}\leq A_e c(\|\bm{Y}\|_{\foneonenu},\|\tilde{\bm{Y}}\|_{\foneonenu},A_\mu,\nu_m)\|\bm{Y}\|_{\ftwoonenu}\|\bm{Y}-\tilde{\bm{Y}}\|_{\dot{\mathcal{F}}^{1,1}},
    \end{aligned}
\end{equation*}
which is trivially bounded by
\begin{equation*}
    \begin{aligned}
    \|\bm{Q}_1\|_{\fzeronenu}\leq A_e c(\|\bm{Y}\|_{\foneonenu},\|\tilde{\bm{Y}}\|_{\foneonenu},A_\mu,\nu_m)\|\bm{Y}\|_{\ftwoonenu}\|\bm{\mathcal{Y}}-\tilde{\bm{\mathcal{Y}}}\|_{\dot{\mathcal{F}}^{1,1}}.
    \end{aligned}
\end{equation*}
It is now clear that \eqref{balanceY} allows to control this term by Gr\"onwall's lemma in \eqref{uniq}.

We proceed to estimate $\bm{Q}_2$ in \eqref{N1diff}. Following the steps in \eqref{N11}, but maintaining the difference between $\tilde{\bm{X}}_c$ and $\bm{X}_c$ together, we find that
\begin{equation*}
    \begin{aligned}
    \|\bm{Q}_2\|_{\fzeronenu}&\leq c(\|\bm{Y}\|_{\foneonenu},\|\tilde{\bm{Y}}\|_{\foneonenu},A_\mu,\nu_m)\|\tilde{\bm{Y}}\|_{\ftwoonenu}\big|\widehat{Y}_{c,2}(1)-\widehat{\tilde{Y}}_{c,2}(1)\big|\\
    &\leq A_e c(\|\bm{Y}\|_{\foneonenu},\|\tilde{\bm{Y}}\|_{\foneonenu},A_\mu,\nu_m)\|\tilde{\bm{Y}}\|_{\ftwoonenu}\|\bm{\mathcal{Y}}-\tilde{\bm{\mathcal{Y}}}\|_{\dot{\mathcal{F}}^{1,1}},
    \end{aligned}
\end{equation*}
and it is thus controlled in the same way. The bound for $\bm{Q}_3$ follows exactly as in \eqref{N11} and has the same structure as the bound for $\bm{Q}_2$.

Finally, we are left with $\bm{Q}_4$, for which we have
\begin{equation*}
    \begin{aligned}
    \bm{Q}_4\leq c(\|\bm{Y}\|_{\foneonenu},\|\tilde{\bm{Y}}\|_{\foneonenu})\|\tilde{\bm{F}}_L-\bm{F}_L\|_{\fzeronenu}.
    \end{aligned}
\end{equation*}
We emphasize that the constant above is given \textit{exactly} by the one for $\bm{\mathcal{N}}_{1,1}$ in \eqref{N11bound}.
The estimate for $\tilde{\bm{F}}_L-\bm{F}_L$ follows from \eqref{F1} (compare to \eqref{F1bound}), we have
\begin{equation}\label{uniqF1}
    \|\tilde{\bm{F}}_L-\bm{F}_L\|_{\fzeronenu}\leq 2A_e\|\tilde{\bm{X}}-\bm{X}\|_{\ftwoonenu}
    +2A_e \frac{|A_\mu|}{1-A_\mu} \|\tilde{\bm{X}}-\bm{X}\|_{\foneonenu},
\end{equation}
so, moving to $\bm{Y}$ variable, we obtain
\begin{equation*}
    \begin{aligned}
    \bm{Q}_4&\leq A_e c(\|\bm{Y}\|_{\foneonenu},\|\tilde{\bm{Y}}\|_{\foneonenu})\|\tilde{\bm{Y}}-\bm{Y}\|_{\ftwoonenu}\\
    &\quad+A_e c(\|\bm{Y}\|_{\foneonenu},\|\tilde{\bm{Y}}\|_{\foneonenu},A_\mu,\nu_m)\|\tilde{\bm{Y}}-\bm{Y}\|_{\foneonenu},
    \end{aligned}
\end{equation*}
and therefore trivially we have
\begin{equation*}
    \begin{aligned}
    \bm{Q}_4&\leq A_e c(\|\bm{Y}\|_{\foneonenu},\|\tilde{\bm{Y}}\|_{\foneonenu})\|\tilde{\bm{Y}}-\bm{Y}\|_{\ftwoonenu}\\
    &\quad+A_e c(\|\bm{Y}\|_{\foneonenu},\|\tilde{\bm{Y}}\|_{\foneonenu},A_\mu,\nu_m)\|\tilde{\bm{\mathcal{Y}}}-\bm{\mathcal{Y}}\|_{\dot{\mathcal{F}}^{1,1}_\nu}.
    \end{aligned}
\end{equation*}
Although in this section we are denoting $c$ to all constants (possibly depending on $\|\bm{Y}\|_{\foneonenu}$, $\|\tilde{\bm{Y}}\|_{\foneonenu}$, $A_\mu$, $\nu_m$), it is important to notice that the constant in front of the high order term $\|\tilde{\bm{Y}}-\bm{Y}\|_{\ftwoonenu}$ is less or equal than the one appearing in the nonlinear estimates from Section \ref{secanalytic}. This will allow us to absorb these terms using the negative sign coming from the dissipative linear term without additional conditions on the initial data other than the one needed for the earlier existence proof.

In summary, so far we have obtained that 
\begin{equation*}
    \begin{aligned}
    \|\bm{\mathcal{N}}_{1,1}(\bm{X}_c,\bm{X})&-\bm{\mathcal{N}}_{1,1}(\tilde{\bm{X}_c},\tilde{\bm{X}})\|_{\fzeronenu}\leq A_e c(\|\bm{Y}\|_{\foneonenu},\|\tilde{\bm{Y}}\|_{\foneonenu})\|\tilde{\bm{Y}}-\bm{Y}\|_{\ftwoonenu}\\
    &\hspace{-0.5cm}+A_e g(\|\bm{Y}\|_{\foneonenu},\|\tilde{\bm{Y}}\|_{\foneonenu},\|\bm{Y}\|_{\ftwoonenu},\|\tilde{\bm{Y}}\|_{\ftwoonenu},A_\mu,\nu_m)\|\tilde{\bm{\mathcal{Y}}}\!-\!\bm{\mathcal{Y}}\|_{\dot{\mathcal{F}}^{1,1}_\nu},
    \end{aligned}
\end{equation*}
where $g$ is a function whose $L^1$-in-time norm  is bounded independently of time in terms of the initial data $\|\bm{Y}_0\|_{\fzerone}$, $\|\tilde{\bm{Y}}_0\|_{\fzerone}$.   Therefore the second term above can be controlled in \eqref{uniq} after using the Gr\"onwall inequality.

Following the same steps for all the terms corresponding to $\bm{\mathcal{N}}_1$ from \eqref{N1split}, it is clear that one obtains
\begin{equation*}
    \begin{aligned}
    \|\bm{\mathcal{N}}_1(\bm{X}_c,\bm{X})&-\bm{\mathcal{N}}_1(\tilde{\bm{X}_c},\tilde{\bm{X}})\|_{\fzeronenu}\leq  c(\|\bm{Y}\|_{\foneonenu},\|\tilde{\bm{Y}}\|_{\foneonenu})\|\tilde{\bm{F}}_L-\bm{F}_L\|_{\fzeronenu}\\
    &\hspace{-0.5cm}+A_e g(\|\bm{Y}\|_{\foneonenu},\|\tilde{\bm{Y}}\|_{\foneonenu},\|\bm{Y}\|_{\ftwoonenu},\|\tilde{\bm{Y}}\|_{\ftwoonenu},A_\mu,\nu_m)\|\tilde{\bm{\mathcal{Y}}}\!-\!\bm{\mathcal{Y}}\|_{\dot{\mathcal{F}}^{1,1}_\nu},
    \end{aligned}
\end{equation*}
where we use the same letter $g$ to denote another $L^1$-in-time function as explained above and the constant in front of $\|\tilde{\bm{F}}_L-\bm{F}_L\|_{\fzeronenu}$ is exactly given by the one in \eqref{N1}. Since the coefficient of the higher order term in the bound \eqref{uniqF1} is smaller than the one in \eqref{F1bound}, we guarantee that
\begin{multline}\label{uniqaux}
    \|\bm{\mathcal{N}}_1(\bm{X}_c,\bm{X})-\bm{\mathcal{N}}_1(\tilde{\bm{X}_c},\tilde{\bm{X}})\|_{\fzeronenu}
    \\
    \leq  A_e c(\|\bm{Y}\|_{\foneonenu},\|\tilde{\bm{Y}}\|_{\foneonenu})\|\tilde{\bm{Y}}-\bm{Y}\|_{\ftwoonenu}
\\
+A_e g(\|\bm{Y}\|_{\foneonenu},\|\tilde{\bm{Y}}\|_{\foneonenu},\|\bm{Y}\|_{\ftwoonenu},\|\tilde{\bm{Y}}\|_{\ftwoonenu},A_\mu,\nu_m)\|\tilde{\bm{\mathcal{Y}}}\!-\!\bm{\mathcal{Y}}\|_{\dot{\mathcal{F}}^{1,1}_\nu},
\end{multline}
Now, we realize that the same idea applies to the other nonlinear terms in \eqref{uniqnonlinear}. In $\bm{\mathcal{N}}_2$ there are not high order terms to absorb, since the expression of $\bm{F}_0$ \eqref{F0} only depends on the circle part, which has to be controlded via Gr\"onwall. The term $\bm{\mathcal{N}}_3$ will provide an estimate like the one above for $\bm{\mathcal{N}}_1$, where the constant in front of $\|\tilde{\bm{Y}}-\bm{Y}\|_{\ftwoonenu}$ is smaller than \eqref{N3}, for the same reasons given before. Finally,
the same can be said for $\bm{\mathcal{N}}_4$, but with an analogous estimate to \eqref{uniqF1} for the difference $\bm{F}_N-\tilde{\bm{F}}_N$. It follows in the same way as the estimate \eqref{FNbound}, so we omit details to avoid repetition.

The final estimate for the difference of the nonlinear terms in \eqref{uniqnonlinear} has then the form \eqref{uniqaux}, with a coefficient of the highest order norm smaller than the coefficient of the  norm with the highest order derivative in \eqref{nonlinearfinal}. Therefore, under condition \eqref{conditiontheorem}, the highest regularity terms in the nonlinear upper bound can be absorbed by the dissipation in \eqref{uniq} and thus
\begin{equation*}
\begin{aligned}
    \frac{d}{dt}\|\bm{\mathcal{Y}}-\widetilde{\bm{\mathcal{Y}}}&\|_{\foneonenu}
\\
&\leq A_e g(\|\bm{Y}\|_{\foneonenu},\|\tilde{\bm{Y}}\|_{\foneonenu},\|\bm{Y}\|_{\ftwoonenu},\|\tilde{\bm{Y}}\|_{\ftwoonenu},A_\mu,\nu_m)\|\tilde{\bm{\mathcal{Y}}}-\bm{\mathcal{Y}}\|_{\foneonenu}
,
\end{aligned}
\end{equation*}
which provides for all time via Gr\"onwall that
\begin{equation*}
    \|\bm{\mathcal{Y}}-\widetilde{\bm{\mathcal{Y}}}\|_{\foneonenu}\leq c\big(\|\bm{Y}_0\|_{\foneonenu},\|\tilde{\bm{Y}}_0\|_{\foneonenu},A_e, A_\mu,\nu_m\big)\|\bm{\mathcal{Y}}_0-\widetilde{\bm{\mathcal{Y}}}_0\|_{\foneonenu}.
\end{equation*}
We conclude that $\|\bm{\mathcal{Y}}-\widetilde{\bm{\mathcal{Y}}}\|_{\foneonenu}=0$. Together with \eqref{uniqcenter2}, this completes the proof.
\qed

\providecommand{\bysame}{\leavevmode\hbox to3em{\hrulefill}\thinspace}
\providecommand{\href}[2]{#2}




\end{document}